\documentclass[graybox]{svmult}

\usepackage{amsfonts,amsmath,amstext,latexsym,amssymb,physics}
\usepackage{mathptmx}       
\usepackage{helvet}         
\usepackage{courier}        
\usepackage{type1cm}        

\usepackage{newtxtext}       %
\usepackage{newtxmath}       
\usepackage[T1]{fontenc}

\usepackage{graphicx}        
\usepackage{multicol}        
\usepackage[bottom]{footmisc}

\usepackage{cite}
\usepackage{bm}

\usepackage{enumitem}  
\usepackage{calc}
\setlist{labelindent=1pt,itemsep=0.1cm}
\setlist[itemize]{leftmargin=0.7cm}
\setlist[enumerate]{itemindent=0em,leftmargin=0.7cm}

\usepackage{float} 

\usepackage[bookmarks=false,hidelinks,unicode]{hyperref}

\DeclareMathOperator{\esssup}{ess\;sup}
\newcommand*{\QEDB}{\hfill\ensuremath{\square}}%

\smartqed

\allowdisplaybreaks

\begin{document}
\title*{Convergence of commutator of linear integral operators with separable kernel representing monomial covariance type commutation relations in $L_p$}
\titlerunning{Convergence of commutator of linear integral operators with separable }        

\author{Domingos Djinja \and Sergei Silvestrov \and Alex Behakanira Tumwesigye}
\authorrunning{D. Djinja, S. Silvestrov, A. B. Tumwesigye} 

\institute{Domingos Djinja \at
Department of Mathematics and Informatics, Faculty of Sciences, Eduardo Mondlane University, Box 257, Maputo, Mozambique \\
\email{domingos.djindja@uem.ac.mz}
\at
Division of Mathematics and Physics, School of Education, Culture and Communication, M{\"a}lardalen University, Box 883, 72123 V{\"a}ster{\aa}s, Sweden \\
\email{domingos.celso.djinja@mdu.se}
\and
Sergei Silvestrov (corresponding author)
\at Division of Mathematics and Physics, School of Education, Culture and Commu\-nication, M{\"a}lardalen University, Box 883, 72123 V{\"a}ster{\aa}s, Sweden. \\
\email{sergei.silvestrov@mdu.se}
\and
Alex Behakanira Tumwesigye \at
Department of Mathematics, College of Natural Sciences, Makerere University, Box 7062, Kampala, Uganda. \\
\email{alex.tumwesigye@mak.ac.ug}
}



\maketitle

\abstract*{
Representations by linear integral operators on $L_p$ spaces over measure spaces are investigated for the polynomial covariance type commutation relations and more general two-sided generalizations of covariance commutation relations extending simultaneously the covariance and the
reciprocal covariance type commutation relations. Necessary and sufficient conditions for the defining kernels
of the integral operators are obtained for the operators to satisfy such commutation relations associated to pairs of polynomials.
Representations by integral operators are studied both for general polynomial covariance type commutation relations
and for important classes of polynomial covariance commutation relations associated with arbitrary monomials and affine functions.
The convergence of commutators of operators in the sequences of non-commuting operators satisfying
the corresponding monomial covariance commutation relation is investigated, and the sequences of such non-commuting operators
converging to commuting operators, are presented.
\keywords{integral operators, covariance commutation relations, commutator}\\
{\bf MSC 2020:} 47A62, 47G10, 47L80, 47L65}

\abstract{
Representations by linear integral operators on $L_p$ spaces over measure spaces are investigated for the polynomial covariance type commutation relations and more general two-sided generalizations of covariance commutation relations extending simultaneously the covariance and the
reciprocal covariance type commutation relations. Necessary and sufficient conditions for the defining kernels
of the integral operators are obtained for the operators to satisfy such commutation relations associated to pairs of polynomials.
Representations by integral operators are studied both for general polynomial covariance type commutation relations
and for important classes of polynomial covariance commutation relations associated with arbitrary monomials and affine functions.
The convergence of commutators of operators in the sequences of non-commuting operators satisfying
the corresponding monomial covariance commutation relation is investigated, and the sequences of such non-commuting operators
converging to commuting operators, are presented.
\keywords{integral operators, covariance commutation relations, commutator}\\
{\bf MSC 2020:} 47A62, 47G10, 47L80, 47L65}

\section{Introduction}
Commutation relations \index{commutation relation} of the form
\begin{equation} \label{Covrelation}
  AB=B F(A)
\end{equation}
where $A, B$ are elements of an associative algebra and  $F$ is function of the elements of the algebra, are important in many areas of Mathematics and applications. Such commutation relations are usually called covariance relations, crossed product relations or semi-direct product relations. A pair of elements of an algebra that satisfy \eqref{Covrelation} is called a representation of this relation \cite{Samoilenkobook}.

Representations of covariance commutation relations \eqref{Covrelation} by linear operators are important for study of actions and induced representations of groups and semigroups, crossed product operator algebras, dynamical systems, harmonic analysis, wavelets and fractal analysis and, hence have applications in physics and engineering
\cite{BratJorgIFSAMSmemo99,BratJorgbook,JorgWavSignFracbook,JorgOpRepTh88,JorMoore84,MACbook1,MACbook2,MACbook3,OstSambook,Pedbook79,Samoilenkobook}.

A description of the structure of representations for the relation \eqref{Covrelation} and more general families of self-adjoint operators satisfying such relations by bounded and unbounded self-adjoint linear operators on a Hilbert space use reordering formulas for functions of the algebra elements and operators satisfying covariance commutation relation, functional calculus and spectral representation of operators and interplay with dynamical systems generated by iteration of involved in the commutation relations maps is presented in \cite{BratEvansJorg2000,CarlsenSilvExpoMath07,CarlsenSilvAAM09,CarlsenSilvProcEAS10,DutkayJorg3,DJS12JFASilv,DLS09,DutSilvProcAMS,DutSilvSV,JSvT12a,JSvT12b,Mansour16,JMusondaPhdth18,JMusonda19,Musonda20,Nazaikinskii96,OstSambook,PerssonSilvestrov031,PerssonSilvestrov032,PersSilv:CommutRelDinSyst,RST16,RSST16,Samoilenkobook,SaV8894,SilPhD95,STomdynsystype1,SilWallin96,SvSJ07a,SvSJ07b,SvSJ07c,SvT09,Tomiyama87,Tomiama:SeoulLN1992,Tomiama:SeoulLN2part2000,AlexThesis2018,TumwRiSilv:ComCrPrAlgPieccnstfnctreallineSPAS19v2,VaislebSa90}.

Constructions of representations of polynomial covariance commutations relations by pairs of  linear integral operators with general kernels, linear integral operators with separable kernels and linear multiplication operator  have been considered in \cite{DjinjaEtAll_IntOpOverMeasureSpaces,DjinjaEtAll_LinItOpInGenSepKern,DjinjaEtAll_LinMultIntOp}.

In this paper, we construct representations of the covariance type commutation relations \eqref{Covrelation} and more general relation $H(A)B=BF(A)$, where $F,H$ are polynomials, by pairs of linear integral operators with separable kernels on Banach spaces $L_p$ over measure spaces. These results generalize results in \cite{DjinjaEtAll_LinItOpInGenSepKern} and gives also results for the reci\-procal commutation relation $BA=F(A)B$ of \eqref{Covrelation}.
We consider representations of commutation relations $H(A)B=BF(A)$ by linear integral operators defined by kernels satisfying different conditions and derive conditions on such kernel functions so that the corresponding operators satisfy $H(A)B=BF(A)$ for polynomials $F, H$.
Representation by integral operators are studied both for general polynomial covariance commutation relations and for important classes of polynomial covariance commutation relations associated to arbitrary monomials and to affine functions.
Examples of integral operators with separable kernels on $L_p$ spaces representing the relation $AB=\delta BA^2$, $\delta\in\mathbb{R}\setminus\{0\}$  are constructed in \cite{DjinjaEtAll_LinItOpInGenSepKern}. Furthermore, in this paper, convergence of commutators of  those sequences of non-commuting operators which satisfy the corresponding monomial covariance commutation relation is investigated.

This paper is organized in four sections. After the introduction,  we present in Section \ref{SecPreNot} some preliminaries, notations, basic definitions and a useful lemma.
In Section \ref{SecRepreBothLI}, we present representations when both $A$ and $B$ are linear integral ope\-rators with separable kernels acting on the Banach spaces $L_p$. In Section  \ref{SecRepreBothLISeqOfNonCommutOpConvCommuting} we present
sequences of non-commuting operators which satisfy monomial covariance commutation relation and converge to commuting operators.

\section{Preliminaries and notations}\label{SecPreNot}
In this section we present preliminaries, basic definitions and notations for this article
\cite{AdamsG,AkkLnearOperators,BrezisFASobolevSpaces,FollandRA,Kantarovitch,Kolmogorov,KolmogorovVol2,RudinRCA}.

Let $\mathbb{R}$ be the set of all real numbers, $\mathbb{Z}$ the set of all integers and $ X$ a non-empty set.
Let $(X,\Sigma, \mu)$ be a $\sigma$-finite measure space, that is, $\Sigma$ is a $\sigma-$algebra of measurable subsets of $X$,
and $X$ can be covered with at most countable many disjoint sets $E_1,E_2,E_3,\ldots$ such that $ E_i\in \Sigma, \,
\mu(E_i)<\infty$, $i=1,2,\ldots$  and $\mu$ is a measure.
For $1\leq p<\infty,$ we denote by $L_p(X,\mu)$, the set of all classes of equivalent (different on a set of zero measure)
measurable functions $f:X\to \mathbb{R}$ such that
$\int\limits_{X} |f(t)|^p d\mu < \infty.$
This is a Banach space (Hilbert space when $p=2$) with norm
$\| f\|_p= \big( \int\limits_{X} |f(t)|^p d\mu \big)^{\frac{1}{p}}.$
We denote by $L_\infty(X,\mu)$ the set of all classes of equivalent measurable functions $f:X\to \mathbb{R}$ such that exists $\lambda >0$,
$|f(t)|\le \lambda$
almost everywhere. This is a Banach space with norm
$\displaystyle\|f\|_{\infty}=\mathop{\esssup}_{t\in X} |f(t)|.$
The support of a function $f:\, X\to\mathbb{R}$ is  ${\rm supp }\, f = \{t\in X \colon \, f(t)\not=0\}.$

Let $(\mathbb{R},\Sigma,\mu)$ be the standard Lebesgue measure space. We will use the notation
\begin{equation}\label{QGpairingDefinition}
  Q_{\Lambda}(u,v)=\int\limits_{\Lambda} u(t)v(t)d\mu
\end{equation}
where $\Lambda\in \Sigma$, and $u,v:\, \Lambda\to \mathbb{R} $ are such functions that the integral exists and is finite.

The following useful lemma for integral operators proved in \cite{DjinjaEtAll_LinItOpInGenSepKern} will be used throughout the article.
\begin{lemma}[\hspace{-0.1mm}\cite{DjinjaEtAll_LinItOpInGenSepKern}]\label{LemmaAllowInfSetsEqLp}
Let $(X,\Sigma,\mu)$ be a $\sigma$-finite measure space. Let $f, g\in L_q(X,\mu)$ for $1 \leq q \leq \infty$ and let $G_1,G_2\in \Sigma$,  $i=1,2$.
Let $G=G_1\cap G_2$. Then the following statements are equivalent\textup{:}
\begin{enumerate}[label=\textup{\arabic*.}, ref=\arabic*]
  \item \label{LemmaAllowInfSetsEqLp:cond1} For all $x\in L_p(X,\mu)$, $1\leq p \leq \infty$ such that
  $ \frac{1}{p}+\frac{1}{q}=1$,
\begin{equation*}
 Q_{G_1}(f,x)= \int\limits_{G_1} f(t)x(t)d\mu=\int\limits_{G_2} g(t)x(t)d\mu=Q_{G_2}(g,x).
\end{equation*}
\item \label{LemmaAllowInfSetsEqLp:cond3} The following conditions hold\textup{:}
    \begin{enumerate}[label=\textup{\alph*)}]
      \item for almost every $t\in G$, $f(t)=g(t)$,
      \item for almost every $t \in G_1\setminus G,\ f(t)=0,$
      \item  for almost every $t \in G_2\setminus G,\ g(t)=0.$
    \end{enumerate}
\end{enumerate}
\end{lemma}

\section{Representations by linear integral operators}\label{SecRepreBothLI}
Let $(X,\Sigma,\mu)$ be a  $\sigma$-finite measure space. In this section we consider  representations
of the covariance type commutation relation \eqref{Covrelation} when both $A$ and $B$ are linear integral operators \index{integral operator} acting from the Banach space $L_p(X,\mu)$ to itself for a fixed $p$ such that $\ 1\le p\le\infty$ defined as
\begin{equation}\label{OpeqnsSKthmGenSeparedKnlsAB}
  (Ax)(t)= \int\limits_{G_A} \sum_{i=1}^{l_A} a_i(t)c_i(s)x(s)d\mu_s,\quad (Bx)(t)= \int\limits_{G_B}\sum_{j=1}^{l_B} b_j(t)e_j(s)x(s)d\mu_s,
\end{equation}
for almost every $t$, where the index in $\mu_s$ indicates the variable of integration, $G_A\in \Sigma$ and $G_B\in \Sigma$.
When $B=0$, the relation \eqref{Covrelation} is trivially satisfied for any $A$. If $A=0$ then the relation \eqref{Covrelation} reduces to $F(0)B=0$. This implies either ($F(0)=0$ and $B$ can be any well defined operator) or $B=0$. Thus, we focus on construction and properties of non-zero representations of \eqref{Covrelation}.

We follow Folland \cite{FollandRA} to verify that under certain conditions, operators in \eqref{OpeqnsSKthmGenSeparedKnlsAB} are well defined and bounded.
If $a_i, b_j\in L_p(X,\mu)$, $c_i\in L_q(G_A,\mu)$, $e_j\in L_q(G_B,\mu)$, $i,j, l_A,l_B$ are positive integers,  $i=1,\ldots, l_A$, $j=1,\ldots, l_B$,  $1\leq q\leq\infty$, $\frac{1}{p}+\frac{1}{q}=1$,  then operators $A$ and $B$ are well defined and bounded. In fact, if we let
$(A_ix)(t)=\int\limits_{G_A} a_i(t)c_i(s)x(s)d\mu_s$ for almost every $t$ and $(Ax)(t)=\sum\limits_{i=1}^{l_A} (A_ix)(t)$, for almost every $t\in X$, and if $1<p<\infty$, then
by H\"older inequality
we have for all $x\in L_p(X,\mu)$,  $c_i(\cdot)x(\cdot)\in L_1(G_A,\mu)$ and
\begin{align*}
    \|A_ix\|^p_{L_p(X,\mu)} = & 
    \int\limits_{X} \big|\ \int\limits_{G_A} a_i(t)c_i(s)x(s)d\mu_s \big|^p d\mu_t \\
    \leq & \int\limits_{X} |a_i(t)|^p d\mu_t
  \big( \int\limits_{X} |I_{G_A}(s)c_i(s)|^q d\mu_s \big)^{\frac{p}{q}}
  \int\limits_{X} |x(s)|^p d\mu_s
   \\
   = & \int\limits_{X} |a_i(t)|^p d\mu_t \big( \int\limits_{G_A} |c_i(s)|^q d\mu_s \big)^{\frac{p}{q}} \int\limits_{X} |x(s)|^p d\mu_s \\
     = &\|a_i\|^p_{L_p(X,\mu)} \|c_i\|_{L_q(G_A,\mu)}^{p}\|x\|_{L_p(X,\mu)}^p,
\end{align*}
for $1< p <\infty$, $i=1,\ldots, l_A$. For $p=\infty$,  by H\"older inequality, for all $x\in L_\infty(X,\mu)$, $c_i(\cdot)x(\cdot)\in L_1(G_A,\mu)$ and
\begin{align*}
 \|A_ix\|_{L_\infty(X,\mu)} & =  \mathop{\esssup}_{t\in X} \big| \int\limits_{G_A} a_i(t)c_i(s)x(s)d\mu_s \big|
 \\
 & \leq
 \big(\mathop{\esssup}_{t\in X} |a_i(t)|\big) \big(\int\limits_{X} |I_{G_A}(s)c_i(s)|d\mu_s\big) \big( \mathop{\esssup}_{s\in X} |x(s)|\big)
\\
 &= \big( \mathop{\esssup}_{t\in X} |a_i(t)| \big)\big( \int\limits_{G_A} |c_i(s)|d\mu_s\big) \big( \mathop{\esssup}_{s\in X} |x(s)|\big)
 \\
 &
 =\|a_i\|_{L_\infty(X,\mu)}\|c_i\|_{L_1(G_A,\mu)}\|x\|_{L_\infty(X,\mu)},
\end{align*}
for $i=1,\ldots, l_A$. Similarly, one argues on the case $p=1$.   Since $L_p(X,\mu)$, $1\leq p \leq \infty$ is a linear space, we conclude that operators $A$, $B$ are well defined on $L_p(X,\mu)$ and bounded.

The following theorem generalizes  \cite[Theorem 2]{DjinjaEtAll_LinItOpInGenSepKern}
in the sense that $H(z)$ can be any polynomial and not only $H(z)=z$. In particular, this allows us to consider, simultaneously,
both the commutation relations of the form $AB=BF(A)$ and the reciprocal relations of the form  $BA=H(A)B$, for arbitrary polynomials $H$ and $F$.

\begin{theorem}\label{thmBothIntOPGenSeptedKernels}
Let $(X,\Sigma,\mu)$ be $\sigma$-finite measure space. Let $A:L_p(X,\mu)\to L_p(X,\mu)$,  $B:L_p(X,\mu)\to L_p(X,\mu)$, $1\le p\le\infty$ be nonzero operators defined as
\begin{equation} 
  (Ax)(t)= \int\limits_{G_A} \sum _{i=1}^{l_A} a_i(t)c_i(s)x(s)d\mu_s,\quad (Bx)(t)= \int\limits_{G_B}\sum_{j=1}^{l_B} b_j(t)e_j(s)x(s)d\mu_s,
\end{equation}
for almost every $t$, where the index in $\mu_s$ indicates the variable of integration, $G_A\in \Sigma$ and $G_B\in \Sigma$, $a_i,b_j\in L_p(X,\mu)$, $c_i\in L_q(G_A,\mu)$, $e_j\in L_q(G_B,\mu)$, $i,j, l_A,l_B$ are positive integers such that $1\leq i\leq  l_A$, $1\leq j\leq  l_B$ and  $1\leq q\leq\infty$ with $\frac{1}{p}+\frac{1}{q}=1$. Consider polynomials
$F(z)=\sum\limits_{j=0}^{n_f} f_j z^j$ and $H(z)=\sum\limits_{l=0}^{n_h} h_l z^l$ where $f_j, h_l \in\mathbb{R}$, $j=0,\ldots,n_f$,
$l=0,\ldots, n_h$. Let $ G=G_A\cap G_B,$ and
\[
 \gamma_{i_1}=1,\quad \gamma_{i_1,\ldots,i_m}=\prod_{l=1}^{m-1} Q_{G_A} (a_{i_{l+1}},c_{i_l}),\ m\ge 2.
\]
where $Q_{\Lambda}(u,v)$, $\Lambda\in\Sigma$, is defined by \eqref{QGpairingDefinition}. Then
$
  H(A)B=BF(A)
$
if and only if  the following conditions are fulfilled:
\begin{enumerate}[label=\textup{\arabic*.}, ref=\arabic*,wide]
  \item \label{thmBothIntOPGenSeptedKernelsRelHABeqBFA:cond1}
        for almost every $(t,s)\in  X\times G$, we have
            \begin{align*}
           &\resizebox{0.95\hsize}{!}{$\displaystyle (h_0-f_0) \sum_{k=1}^{l_B}  b_k(t)e_k(s)+\sum_{l=1}^{n_h}\sum_{k=1}^{l_B}
           \sum_{m_1,\ldots,m_l=1}^{l_A} h_l
           \gamma_{m_1,\ldots, m_l} a_{m_1}(t)Q_{G_A}(b_k,c_{m_l})e_k(s)$}\\
          & \resizebox{0.67\hsize}{!}{$\displaystyle
=\sum_{k=1}^{l_B}\sum_{j=1}^{n_f}\sum_{i_1,\ldots,i_j=1}^{l_A} f_j b_k(t)Q_{G_B}(e_k,a_{i_1})\gamma_{i_1,\ldots,i_j}c_{i_j}(s) $};
            \end{align*}
\item  \label{thmBothIntOPGenSeptedKernelsRelHABeqBFA:cond2}
for almost every $(t,s)\in  X\times (G_A\setminus G)$, we have
 \begin{equation*}
 \sum_{k=1}^{l_B}\sum_{j=1}^{n_f}
 \sum_{i_1,\ldots, i_j=1}^{l_A} f_j b_k(t)Q_{G_B}(e_k,a_{i_1})
 \gamma_{i_1,\ldots,i_j} c_{i_j}(s)=0;
 \end{equation*}
\item \label{thmBothIntOPGenSeptedKernelsRelHABeqBFA:cond3}
for almost every $(t,s)\in  X\times (G_B\setminus G)$, we have
\begin{equation*}
 \resizebox{0.9\hsize}{!}{$\displaystyle
(f_0-h_0) \sum_{k=1}^{l_B}  b_k(t)e_k(s)=
\sum_{l=1}^{n_h}\sum_{k=1}^{l_B} \sum_{m_1,\ldots m_l=1}^{l_A}
h_l \gamma_{m_1,\ldots,m_l} Q_{G_A}(b_k,c_{m_l})a_{m_1}(t)e_k(s).$}
\end{equation*}
\end{enumerate}
\end{theorem}
\begin{proof} \smartqed
We observe that since for each $i,j$, $1\le i \le l_A$, $1\le j\le l_B$,  $a_i,b_j\in L_p(X,\mu),\ 1\le p\le\infty$, $c_i\in L_q(G_A,\mu)$, $e_j\in L_q(G_B,\mu)$, where $1\le q\le \infty$, with $\frac{1}{p}+\frac{1}{q}=1,$ $L_p(X,\mu)$, $L_q(X,\mu)$ are linear spaces, then  operators $A$ and $B$ are well-defined. By direct calculation, we have
  \begin{eqnarray*}
  (A^2x)(t)&=&\int\limits_{G_A} \sum_{i=1}^{l_A} a_i(t)c_i(s)(Ax)(s)d\mu_s\\
  &=&\int\limits_{G_A} \sum_{i_1, i_2=1}^{l_A}a_{i_1}(t)c_{i_1}(s) a_{i_2}(s)d\mu_s\int\limits_{G_A} c_{i_2}(\tau)x(\tau)d\mu_{\tau}\\
  &=& \sum_{i_1, i_2=1}^{l_A} a_{i_1}(t)Q_{G_A}(a_{i_2},c_{i_1})\int\limits_{G_A} c_{i_2}(\tau)x(\tau)d\mu_\tau
  \\
  &=&
  \int\limits_{G_A} \sum_{i_1,i_2=1}^{l_A}\gamma_{i_1,i_2}a_{i_1}(t)c_{i_2}(\tau)x(\tau)d\mu_\tau,
  \\
    (A^3x)(t)&=&\int\limits_{G_A} \sum_{i_1=1}^{l_A} a_{i_1}(t)c_{i_1}(s)\left(\sum_{i_2,i_3=1}^{l_A} a_{i_2}(s)Q_{G_A}(a_{i_3},c_{i_3})c_{i_2}(\tau)x(\tau)d\mu_\tau \right)d\mu_s \\
 &=& \sum_{i_1,i_2,i_3=1}^{l_A} \int\limits_{G_A} a_{i_1}(t)Q_{G_A}(a_{i_3},c_{i_2})  c_{i_1}(s)a_{i_2}(s)d\mu_s   \int\limits_{G_A} c_{i_3}(\tau)x(\tau)d\mu_\tau \\
 &=& \int\limits_{G_A} \sum_{i_1,i_2,i_3=1}^{l_A}  Q_{G_A}(a_{i_3},c_{i_2})Q_{G_A}(a_{i_2},c_{i_1}) a_{i_1}(t)c_{i_3}(\tau)x(\tau)d\mu_\tau \\
 &=& \int\limits_{G_A} \sum_{i_1,i_2,i_3=1}^{l_A} \gamma_{i_1,i_2,i_3} a_{i_1}(t)c_{i_3}(\tau)x(\tau)d\mu_\tau.
\end{eqnarray*}
 for almost every $t$.  We suppose that
  \begin{equation}\label{AonCompositionIntegralOpGenSeparetedKernels}
  (A^{m}x)(t)=\int\limits_{G_A} \sum_{i_1,\ldots,i_m=1}^{l_A}
  \gamma_{i_1,\ldots,i_m}a_{i_1}(t)c_{i_m}(\tau)x(\tau)d\mu_\tau,\quad m=1,2,\ldots
  \end{equation}
 for almost every $t$. Then
  \begin{align*}\nonumber
    &\resizebox{0.97\hsize}{!}{$\displaystyle (A^{m+1}x)(t)=\int\limits_{G_A}\hspace{0mm}
    \sum_{i_1,\ldots,i_m=1}^{l_A}\gamma_{i_1,\ldots,i_m}a_{i_1}(t)c_{i_m}(\tau)
    \left(\hspace{2mm} \int\limits_{G_A}\hspace{0mm} \sum_{i_{m+1}=1}^{l_A} a_{i_{m+1}}(\tau)c_{i_{m+1}}(s)d\mu_s \right)d\mu_\tau $}
    \\
    &\hspace{1.45cm}= \int\limits_{G_A} \gamma_{i_1,\ldots,i_{m+1}}a_{i_1}(t)c_{i_{m+1}}(\tau)x(\tau)d\mu_\tau.
\end{align*}
for almost every $t$. Therefore, we have
\begin{equation*}
(F(A)x)(t)=f_0 x(t)+\int\limits_{G_A}
\sum_{j=1}^{n_f}\sum_{i_1,\ldots,i_j=1}^{l_A}
\gamma_{i_1,\ldots,i_j} f_j a_{i_1}(t)c_{i_j}(\tau)x(\tau)d\mu_\tau
\end{equation*}
for almost every $t$. Similarly, we have
\begin{equation*}
(H(A)x)(t)=h_0 x(t)+\int\limits_{G_A} \sum_{l=1}^{n_h}
\sum_{m_1,\ldots,m_l=1}^{l_A}
\gamma_{m_1,\ldots,m_l} h_l a_{m_1}(t)c_{m_j}(\tau)x(\tau)d\mu_\tau
\end{equation*}
for almost every $t$.
Then we compute,
  \begin{align}
  & \nonumber
  (H(A)Bx)(t)=\displaystyle h_0\int\limits_{G_B}\sum_{k=1}^{l_B}b_k(t)e_k(s)x(s)d\mu_s
  \\ \nonumber
  &\resizebox{0.89\hsize}{!}{$\displaystyle
  +\int\limits_{G_A} \sum_{l=1}^{n_h}\sum_{m_1,\ldots, m_l=1}^{l_A}
  h_l \gamma_{m_1,\ldots,m_l} a_{m_1}(t)c_{m_l}(s)d\mu_s
  \int\limits_{G_B}\sum_{k=1}^{l_B} b_k(s) e_k(\tau) x(\tau)d\mu_{\tau} $}\\
  &
  =h_0\int\limits_{G_B}\sum_{k=1}^{l_B}b_k(t)e_k(s)x(s)d\mu_s \nonumber \\
  &+
  \int\limits_{G_B} \sum_{l=1}^{n_h}\sum_{m_1,\ldots,m_l=1}^{l_A}
  \sum_{k=1}^{l_B} h_l \gamma_{m_1,\ldots,m_l} Q_{G_A}(b_k,c_{m_l}) a_{m_1}(t) e_k(\tau) x(\tau)d\mu_{\tau},  \label{CompABProofThmBothIntOpGenSeparatedKnDI}
  \end{align}
 for almost every $t$. Similarly,
  \begin{align}\nonumber
   &(BF(A)x)(t)=\int\limits_{G_B} \sum_{k=1}^{l_B} \delta_0 b_k(t)e_k(\tau)x(\tau)d\mu_{\tau} \\ \nonumber
   &
   + \int\limits_{G_B}\sum_{k=1}^{l_B} b_k(t)e_k(s)
   \int\limits_{G_A} \sum_{j=1}^{n_f} \sum_{i_1,\ldots,i_j=1}^{l_A}
   f_j \gamma_{i_1,\ldots,i_j} a_{i_1}(s)c_{i_j}(\tau) x(\tau)d\mu_\tau   \\
&= \int\limits_{G_B}
\sum_{k=1}^{l_B} f_0 b_k(t)e_k(\tau) x(\tau)d\mu_{\tau} \nonumber \\
& + \int\limits_{G_A}\sum_{k=1}^{l_B}\sum_{j=1}^{n_f} \sum_{i_1,\ldots,i_j=1}^{l_A} f_j
\gamma_{i_1,\ldots,i_j} Q_{G_B}(e_k,a_{i_1})b_k(t) c_{i_j}(\tau)x(\tau)d\mu_{\tau},
\label{EqBFAProofThmBothIntOpGenSepratKnDI}
  \end{align}
  for almost every $t$.
  Thus, $(H(A)Bx)(t)=(BF(A)x)(t)$ for all $x\in L_p( X,\mu)$ if and only if
 \begin{align*}
  & \resizebox{0.95\hsize}{!}{$\displaystyle \int\limits_{G_B} (h_0-f_0)\left(\sum_{k=1}^{l_B}  b_k(t)e_k(s)+\sum_{l=1}^{n_h}\sum_{k=1}^{l_B}\sum_{m_1,\ldots,m_l=1}^{l_A} h_l
  \gamma_{m_1,\ldots, m_l} Q_{G_A}(b_k,c_{m_l})a_m(t)e_k(s)\right) x(s)d\mu_s $}\\
  & \resizebox{0.72\hsize}{!}{$\displaystyle
  =\int\limits_{G_A} \left(\sum_{k=1}^{l_B}\sum_{j=1}^{n_f}
  \sum_{i_1,\ldots,i_j=1}^{l_A} f_j \gamma_{i_1,\ldots,i_j} Q_{G_B}(e_k,a_{i_1})b_k(t)c_{i_j}(s)\right)x(s)d\mu_s.$}
 \end{align*}
Then by applying Lemma \ref{LemmaAllowInfSetsEqLp}
we conclude that $H(A)B=BF(A)$ if and only if
   \begin{enumerate}[label=\textup{\arabic*.}, ref=\arabic*]
  \item for almost every $(t,s)\in  X\times G$,
            \begin{align*}
            & (h_0-f_0) \sum_{k=1}^{l_B}  b_k(t)e_k(s)\\
            &+\sum_{l=1}^{n_h}
            \sum_{k=1}^{l_B}\sum_{m_1,\ldots,m_l=1}^{l_A}
            h_l \gamma_{m_1,\ldots, m_l} a_{m_1}(t)Q_{G_A}(b_k,c_{m_l})e_k(s)\\
            &
            =\sum_{k=1}^{l_B}\sum_{j=1}^{n_f}\sum_{i_1,\ldots,i_j=1}^{l_A}
            f_j b_k(t)Q_{G_B}(e_k,a_{i_1})\gamma_{i_1,\ldots,i_j}c_{i_j}(s).
            \end{align*}
\item  for almost every $(t,s)\in  X\times (G_A\setminus G)$, we have
 \begin{equation*}
 \sum_{k=1}^{l_B}\sum_{j=1}^{n_f} \sum_{i_1,\ldots, i_j=1}^{l_A}
 f_j b_k(t)Q_{G_B}(e_k,a_{i_1}) \gamma_{i_1,\ldots,i_j} c_{i_j}(s)=0.
 \end{equation*}
\item  for almost every $(t,s)\in  X\times (G_B\setminus G)$, we have
\begin{align}
& (f_0-h_0) \sum_{k=1}^{l_B}  b_k(t)e_k(s) \nonumber \\
& =\sum_{l=1}^{n_h}\sum_{k=1}^{l_B} \sum_{m_1,\ldots m_l=1}^{l_A}
h_l \gamma_{m_1,\ldots,m_l} Q_{G_A}(b_k,c_{m_l})a_{m_1}(t)e_k(s).
\tag*{\qed} \end{align}
\end{enumerate}
\end{proof}
\begin{remark}\label{RemOpDefInSameIntervalGenSepKernels}
In Theorem \ref{thmBothIntOPGenSeptedKernels}, when $G_A=G_B=G$, conditions \ref{thmBothIntOPGenSeptedKernelsRelHABeqBFA:cond2} and \ref{thmBothIntOPGenSeptedKernelsRelHABeqBFA:cond3} are taken on set of measure zero so we can ignore them. Thus, we only remain with condition \ref{thmBothIntOPGenSeptedKernelsRelHABeqBFA:cond1}. When $G_A\not=G_B$ we need to check also conditions \ref{thmBothIntOPGenSeptedKernels:cond2} and \ref{thmBothIntOPGenSeptedKernels:cond3} outside the intersection  $G=G_A\cap G_B$.
\end{remark}

The following corollary was proved in \cite[Corollary 3]{DjinjaEtAll_LinItOpInGenSepKern} and we notice that it is also a corollary for Theorem \ref{thmBothIntOPGenSeptedKernels} when $H(z)=z$, for all $z\in\mathbb{R}$.
\begin{corollary}[\cite{DjinjaEtAll_LinItOpInGenSepKern}]\label{corHAAthmBothIntOPGenSeptedKernels}
Let $(X,\Sigma,\mu)$ be $\sigma$-finite measure space. Let $A:L_p(X,\mu)\to L_p(X,\mu)$,  $B:L_p(X,\mu)\to L_p(X,\mu)$, $1\le p\le\infty$ be nonzero operators defined as follows
\begin{equation*}
  (Ax)(t)= \int\limits_{G_A} \sum _{i=1}^{l_A} a_i(t)c_i(s)x(s)d\mu_s,\quad (Bx)(t)= \int\limits_{G_B}\sum_{j=1}^{l_B} b_j(t)e_j(s)x(s)d\mu_s,
\end{equation*}
for almost every $t$, where the index in $\mu_s$ indicates the variable of integration, $G_A\in \Sigma$ and $G_B\in \Sigma$, $a_i,b_j\in L_p(X,\mu)$, $c_i\in L_q(G_A,\mu)$, $e_j\in L_q(G_B,\mu)$, $i,j, l_A,l_B$ are positive integers such that $1\le i\le  l_A$, $1\leq j\leq  l_B$ and  $1\leq q\leq\infty$ with $\frac{1}{p}+\frac{1}{q}=1$. Consider a polynomial defined by
$F(z)=\sum\limits_{j=0}^{n} f_j z^j$, where $f_j \in\mathbb{R}$, $j=0,\ldots,n$. Let $ G=G_A\cap G_B,$ and
\[
 \gamma_{i_1}=1,\quad \gamma_{i_1,\ldots,i_m}=\prod_{l=1}^{m-1} Q_{G_A} (a_{i_{l+1}},c_{i_l}),\ m\ge 2.
\]
where $Q_{\Lambda}(u,v)$, $\Lambda\in\Sigma$, is defined by \eqref{QGpairingDefinition}. Then
$
  AB=BF(A)
$
if and only if  the following conditions are fulfilled:
\begin{enumerate}[label=\textup{\arabic*.}, ref=\arabic*]
  \item \label{corHAAthmBothIntOPGenSeptedKernelsRelABeqBFA:cond1}
        for almost every $(t,s)\in  X\times G$, we have
            \begin{gather*}
            - \sum_{k=1}^{l_B} \delta_0 b_k(t)e_k(s)+\sum_{k=1}^{l_B}\sum_{m=1}^{l_A}a_m(t)Q_{G_A}(b_k,c_m)e_k(s)\\
             =\sum_{k=1}^{l_B}\sum_{j=1}^{n}\sum_{i_1,\ldots,i_j=1}^{l_A} f_j b_k(t)Q_{G_B}(e_k,a_{i_1})\gamma_{i_1,\ldots,i_j}c_{i_j}(s);
            \end{gather*}
\item  \label{corHAAthmBothIntOPGenSeptedKernelsABeqBFA:cond2}  for almost every $(t,s)\in  X\times (G_A\setminus G)$, we have
     \begin{equation*}
 \sum_{k=1}^{l_B}\sum_{j=1}^{n} \sum_{i_1,\ldots, i_j=1}^{l_A} f_j b_k(t)Q_{G_B}(e_k,a_{i_1}) \gamma_{i_1,\ldots,i_j} c_{i_j}(s)=0;
 \end{equation*}

\item \label{corHAAthmBothIntOPGenSeptedKernelsABeqBFA:cond3} for almost every $(t,s)\in  X\times (G_B\setminus G)$, we have
      \begin{equation*}
 \sum_{k=1}^{l_B} f_0 b_k(t)e_k(s)=\sum_{k=1}^{l_B} \sum_{m=1}^{l_A} Q_{G_A}(b_k,c_m)a_m(t)e_k(s).
\end{equation*}
\end{enumerate}
\end{corollary}


\begin{corollary}
Let $(X,\Sigma,\mu)$ be $\sigma$-finite measure space. Let $A:L_p(X,\mu)\to L_p(X,\mu)$,  $B:L_p(X,\mu)\to L_p(X,\mu)$, $1\le p\le\infty$ be nonzero operators defined as
\begin{equation} 
  (Ax)(t)= \int\limits_{G_A} \sum _{i=1}^{l_A} a_i(t)c_i(s)x(s)d\mu_s,\quad (Bx)(t)= \int\limits_{G_B}\sum_{j=1}^{l_B} b_j(t)e_j(s)x(s)d\mu_s,
\end{equation}
for almost every $t$, where the index in $\mu_s$ indicates the variable of integration, $G_A\in \Sigma$ and $G_B\in \Sigma$, $a_i,b_j\in L_p(X,\mu)$, $c_i\in L_q(G_A,\mu)$, $e_j\in L_q(G_B,\mu)$, $i,j, l_A,l_B$ are positive integers such that $1\leq i\leq  l_A$, $1\leq j\leq  l_B$ and  $1\leq q\leq\infty$ with $\frac{1}{p}+\frac{1}{q}=1$. Consider polynomial $H:\mathbb{R}\to \mathbb{R}$ defined by
 $H(z)=\sum\limits_{l=0}^{n_h} h_l z^j$ where $ h_l \in\mathbb{R}$,  $l=0,\ldots, n_h$. Let $ G=G_A\cap G_B,$ and
\[
 \gamma_{i_1}=1,\quad \gamma_{i_1,\ldots,i_m}=\prod_{l=1}^{m-1} Q_{G_A} (a_{i_{l+1}},c_{i_l}),\ m\ge 2.
\]
where $Q_{\Lambda}(u,v)$, $\Lambda\in\Sigma$, is defined by \eqref{QGpairingDefinition}. Then
$
 BA= H(A)B
$
if and only if  the following conditions are fulfilled:
\begin{enumerate}[label=\textup{\arabic*.}, ref=\arabic*]
  \item \label{thmBothIntOPGenSeptedKernels:cond1}
        for almost every $(t,s)\in  X\times G$, we have
            \begin{gather*}
            h_0 \sum_{k=1}^{l_B}  b_k(t)e_k(s)+\sum_{l=1}^{n_h}\sum_{k=1}^{l_B}\sum_{m_1,\ldots,m_l=1}^{l_A} h_l \gamma_{m_1,\ldots, m_l} a_{m_1}(t)Q_{G_A}(b_k,c_{m_l})e_k(s)\\
             =\sum_{k=1}^{l_B}\sum_{i=1}^{l_A} b_k(t)Q_{G_B}(e_k,a_{i})c_{i}(s);
            \end{gather*}
\item  \label{thmBothIntOPGenSeptedKernels:cond2}  for almost every $(t,s)\in  X\times (G_A\setminus G)$, we have
 \begin{equation*}
 \sum_{k=1}^{l_B} \sum_{i=1}^{l_A} b_k(t)Q_{G_B}(e_k,a_{i}) c_{i}(s)=0;
 \end{equation*}
\item \label{thmBothIntOPGenSeptedKernels:cond3} for almost every $(t,s)\in  X\times (G_B\setminus G)$, we have
\begin{equation*}
 \resizebox{0.9\hsize}{!}{$\displaystyle
-h_0 \sum_{k=1}^{l_B}  b_k(t)e_k(s)=
\sum_{l=1}^{n_h}\sum_{k=1}^{l_B}
\sum_{m_1,\ldots,m_l=1}^{l_A}
h_l \gamma_{m_1,\ldots,m_l} Q_{G_A}(b_k,c_{m_l})a_{m_1}(t)e_k(s).$}
\end{equation*}
\end{enumerate}
\end{corollary}

\begin{proof}
  It follows by Theorem \ref{thmBothIntOPGenSeptedKernels} when $F(z)=z$, for all $z\in\mathbb{R}$.
  \QEDB
\end{proof}

\begin{remark}
Let $(G,\Sigma,\mu)$ be a $\sigma-$finite measure space.
Let $A:L_p(G,\mu)\to L_p(G,\mu)$, $B:L_p(G,\mu)\to L_p(G,\mu)$, $1\leq p<\infty$  be linear integral operators with kernels
$k_A(t,s)$, $k_B(t,s)$, $(t,s)\in G\times G$, $G\in \Sigma$,
$F(z)=\sum\limits_{j=0}^{n} f_jz^j$, $j=0,\ldots, n$, $z\in\mathbb{R}$.
If the correspondent adjoint operators exist and are bounded
$A^*:L_q(G,\mu)\to L_q(G,\mu)$, $B^*: L_q(G,\mu)\to L_q(G,\mu)$, $1<q\leq \infty$, such that $\frac{1}{p}+\frac{1}{q}=1$,
then $AB=BF(A)$ if and only if $B^*A^*=F(A^*)B^*$, where $k^*_A(t,s)=k_A(s,t)$, $k^*_B(t,s)=k_B(s,t)$,
$(s,t)\in G\times G$. In fact, this follows by \cite[Proposition 1.4]{ConwayFunctionalAnalysis} and
\cite[Example 1.5 and 1.6]{ConwayFunctionalAnalysis}.
By setting  $A^*=C$, $B^*=D$ then $B^*A^*=B^*F(A^*)$ is equivalent to $DC=F(C)D$, thus this stabilises
a relation between commutation relations $CD=DF(C)$ and $DC=F(C)D$.
\end{remark}

The following corollary of Theorem \ref{thmBothIntOPGenSeptedKernels} is concerned with  representations by integral operators of another important family of covariance commutation relations associated to monomials $F$. It was proved in \cite[Theorem 2]{DjinjaEtAll_LinItOpInGenSepKern} and we notice that it is a corollary of  Theorem \ref{thmBothIntOPGenSeptedKernels} when $n=d$, $\delta_j=0$, for all integers $j\in \{0,\ldots, n\}\setminus \{d\}$, $\delta_d=\delta$ and $H(z)=z$. We also construct sequences of non-commuting operators that satisfy monomial covariance commutation relations in Section \ref{SecRepreBothLISeqOfNonCommutOpConvCommuting}.
\begin{corollary}[\cite{DjinjaEtAll_LinItOpInGenSepKern}]\label{corDiedroRelBothIntOPGenSeptedKernels}
Let $(X,\Sigma,\mu)$ be $\sigma$-finite measure space. Let $A:L_p(X,\mu)\to L_p(X,\mu)$,  $B:L_p(X,\mu)\to L_p(X,\mu)$, $1\le p\le\infty$ be nonzero operators defined as
\begin{equation}\label{OpeqnsSKthm}
  (Ax)(t)= \int\limits_{G_A} \sum _{i=1}^{l_A} a_i(t)b_i(s)x(s)d\mu_s,\quad (Bx)(t)= \int\limits_{G_B}\sum_{j=1}^{l_B} c_j(t)e_j(s)x(s)d\mu_s,
\end{equation}
for almost every $t$, where the index in $\mu_s$ indicates the variable of integration, $G_A\in \Sigma$ and $G_B\in \Sigma$, $a_i,b_j\in L_p(X,\mu)$, $c_i\in L_q(G_A,\mu)$, $e_j\in L_q(G_B,\mu)$, $i,j, l_A,l_B$ are positive integers such that $1\le i\le  l_A$, $1\le j\le  l_B$ and  $1\le q\le\infty$ with $\frac{1}{p}+\frac{1}{q}=1$. Consider a polynomial defined by
$F(z)=\delta z^d$, where $\delta \in\mathbb{R}$, $d\in\mathbb{Z}$, $d>0$. Let $ G=G_A\cap G_B,$ and
\[
 \gamma_{i_1}=1,\quad \gamma_{i_1,\ldots,i_m}=\prod_{l=1}^{m-1} Q_{G_A} (a_{i_{l+1}},c_{i_l}),\ m\ge 2.
\]
where $Q_{\Lambda}(u,v)$, $\Lambda\in\Sigma$, is defined by \eqref{QGpairingDefinition}. Then
$
  AB=BF(A)
$
if and only if  the following conditions are fulfilled:
\begin{enumerate}[leftmargin=*, label=\textup{\arabic*.}, ref=\arabic*]
  \item
        for almost every $(t,s)\in  X\times G$, we have
            \begin{equation*}\hspace{-0.4cm}
           \sum_{k=1}^{l_B}\sum_{m=1}^{l_A}a_m(t)Q_{G_A}(b_k,c_m)e_k(s)=
             \sum_{k=1}^{l_B}\sum_{i_1,\ldots,i_d=1}^{l_A} \delta b_k(t)Q_{G_B}(e_k,a_{i_1})\gamma_{i_1,\ldots,i_d}c_{i_d}(s)
            \end{equation*}
\item  for almost every $(t,s)\in  X\times (G_A\setminus G)$, we have
 \begin{equation*}
 \sum_{k=1}^{l_B} \sum_{i_1,\ldots, i_d=1}^{l_A} \delta b_k(t)Q_{G_B}(e_k,a_{i_1}) \gamma_{i_1,\ldots,i_d} c_{i_d}(s)=0.
 \end{equation*}
\item for almost every $(t,s)\in  X\times (G_B\setminus G)$, we have
\begin{equation*}
\sum_{k=1}^{l_B} \sum_{m=1}^{l_A} Q_{G_A}(b_k,c_m)a_m(t)e_k(s)=0.
\end{equation*}
\end{enumerate}
\end{corollary}

The following statement was proved in \cite[Proposition 1]{DjinjaEtAll_LinItOpInGenSepKern}. We will use this proposition in Section \ref{SecRepreBothLISeqOfNonCommutOpConvCommuting} to compute and study convergence of commutator.
\begin{proposition}\label{PropSimilarthmBothIntOpGensepKernellsCommutativityCommutator}
Let $(X,\Sigma,\mu)$ be $\sigma$-finite measure space. Let $A:L_p(X,\mu)\to L_p(X,\mu)$,  $B:L_p(X,\mu)\to L_p(X,\mu)$, $1\le p\le\infty$ be nonzero operators defined as
\begin{equation*}
  (Ax)(t)= \int\limits_{G} \sum _{i=1}^{l_A} a_i(t)c_i(s)x(s)d\mu_s,\quad (Bx)(t)= \int\limits_{G}\sum_{j=1}^{l_B} b_j(t)e_j(s)x(s)d\mu_s,
\end{equation*}
for almost every $t$, where the index in $\mu_s$ indicates the variable of integration, $G_A\in \Sigma$ and $G_B\in \Sigma$, $a_i,b_j\in L_p(X,\mu)$, $c_i\in L_q(G,\mu)$, $e_j\in L_q(G,\mu)$, $i,j, l_A,l_B$ are positive integers such that $1\le i\le  l_A$, $1\le j\le  l_B$ and  $1\le q\le\infty$ with $\frac{1}{p}+\frac{1}{q}=1$. Let $ G=G_A\cap G_B$ and
$Q_{\Lambda}(u,v)$, $\Lambda\in\Sigma$, is defined by \eqref{QGpairingDefinition}. Then,  for almost every $t\in  X$,
\begin{align*}
(AB)x(t) & = \int\limits_{G} \sum\limits_{k=1}^{l_B}\sum\limits_{m=1}^{l_A}a_m(t)Q_{G_A}(b_k,c_m)e_k(s) x(s)d\mu_s,\\
(BA)x(t) & = \int\limits_{G} \sum\limits_{k=1}^{l_B}\sum\limits_{i_1=1}^{l_A}  b_k(t)Q_{G_B}(e_k,a_{i_1})c_{i_1}(s)x(s)d\mu_s, \\
(AB-BA)x(t) &= \int\limits_{G} \big(\sum\limits_{k=1}^{l_B}\sum\limits_{m=1}^{l_A}a_m(t)Q_{G_A}(b_k,c_m)e_k(s)\\
  & \hspace{2.5cm} -\sum\limits_{k=1}^{l_B}\sum\limits_{i_1=1}^{l_A}  b_k(t)Q_{G_B}(e_k,a_{i_1})c_{i_1}(s)\big)x(s)d\mu_s.
\end{align*}
\end{proposition}

\section{Sequences of non commuting operators that converge to commuting operator}\label{SecRepreBothLISeqOfNonCommutOpConvCommuting}


In this section  we construct families of non-commuting operators $\{C_m\}$ or $\{D_m\}$ such that satisfy the commutation relation $CD=\delta DC^2$ and converge in norm to an operator $C$ or $D$ ($\|C_m-C\|_E\to 0$ or $\|D_m-D\|_E\to 0$, $m\to \infty$), respectively, such that $C$ and $D$ commute. We first present two lemmas which relate conditions of commutativity and convergence of sequence of operators. We assume that the reader is familiar with definitions of convergence of operator sequences and norm properties. For further reading about the norm definition, uniform convergence of operators and their basic properties, please read \cite{BrezisFASobolevSpaces,FollandRA,Kolmogorov,KolmogorovVol2}.

\begin{lemma}\label{LemmaOnseSequenceConvOptorsCommute}
Let $E$ be a normed vector space. Let  $C_m:\, E\to E$ for each positive integer be a bounded linear operators and
let $D:\, E\to E$ be a bounded linear operator. Suppose that the sequence $\{C_m\}$ converge in norm to
an operator $C$. Then, the sequence $\{C_mD-DC_m\}$ converge in norm to zero if and only if $CD=DC$.
\end{lemma}

\begin{proof} \smartqed
\noindent Necessity (only if): Suppose that the sequence of bounded linear operators $\{C_mD-DC_m\}$ converge in norm to zero. Then, by applying linearity and norm properties, we have, for each positive integer $m$, the following:
  \begin{eqnarray*}
    \|CD-DC\|_E&=& \|CD-C_mD+C_mD-DC_m+DC_m-DC\|_E \\
    &=& \|(C-C_m)D+C_mD-DC_m+D(C_m-C)\|_E\\
    & \le &
    2\|C-C_m\|_E\cdot \|D\|_E  +\|C_mD-DC_m\|_E.
  \end{eqnarray*}
  Therefore,
  \begin{equation*}
    \|CD-DC\|_E \le \inf_{m\ge 1} \{2\|C-C_m\|_E\cdot \|D\|_E  +\|C_mD-DC_m\|_E\}=0,
  \end{equation*}
  since $\|C-C_m\|_E\to 0$ and $\|C_mD-C_mD\|_E\to 0$ when $m\to\infty$. Hence we conclude that $CD=DC$.

\noindent {Sufficiency (if):} Suppose that bounded linear operators  $C$ and $D$ commute. By using this, linearity of operators and norm properties,  we have
  the following:
  \begin{eqnarray*}
    \|C_mD-DC_m\|_E  &=&  \|C_mD-CD+CD-DC_m\|_E=\|(C_m-C)D+D(C-C_m)\|_E\\
    & \le & 2\|C_m-C\|_E\cdot \|D\|_E \to 0,\, m\to \infty.
  \end{eqnarray*}
\end{proof}

\begin{remark}
In Lemma \ref{LemmaOnseSequenceConvOptorsCommute},  the convergence of the sequence of bounded linear operators $\{C_mD-DC_m\}$ to zero does not imply that the sequence $\{C_m\}$ converge. For instance, one can take $D$ as identity or zero operator and $C_m$ is a sequence that does not converge but the sequences of norms is bounded. Another example is in the following case of linear operators in $\mathbb{R}^2$  given by matrices:
\begin{gather*}
  C_m=\left(\begin{array}{cc}
        (-1)^m & 0 \\
        0 & 1
      \end{array}\right),\ D=\left(\begin{array}{cc}
                               1 & 0 \\
                               0 & 0
                             \end{array}\right).
\end{gather*}
We have
$
  C_mD=DC_m=\left(\begin{array}{cc}
         (-1)^m & 0 \\
         0 & 0
       \end{array}\right),
$
and so $C_mD-DC_m = 0$, $m\to \infty$. But the sequence of matrices $C_m$ does not converge with respect to the matrix norm induced by the Euclidian norm since the sequence
$(-1)^m$ does not converge when $m\rightarrow \infty$.
\end{remark}

\begin{lemma}\label{LemmaTwoOpSequencesConvOptorsCommute}
Let $E$ be a normed vector space. Let  $C_m:\, E\to E$, $D_m:\, E\to E$ for each positive integer be bounded linear operators and
let $C:\, E\to E$, $D:\, E\to E$ be bounded linear operators. Suppose that the sequences $\{C_m\}$, $\{D_m\}$  converge in norm to
 operators $C$ and $D$, respectively. Then, the sequence $\{C_mD_m-D_mC_m\}$ converges in norm to zero if and only if $CD=DC$.
\end{lemma}

\begin{proof} \smartqed
\noindent \textbf{Necessity (only if):} Suppose that the sequence of bounded linear operators $\{C_mD_m-D_mC_m\}$  converge in norm to zero. Then, by applying linearity and norm properties, we have the following:
  \begin{eqnarray*}
    \|C_mD-DC_m\|_E&=& \|C_mD-C_mD_m+C_mD_m-DC_m+D_mC_m-D_mC_m\|_E \\
    &=& \|C_m(D-D_m)+(D_m-D)C_m+D_mC_m-D_mC_m\|_E\\
    & \le &
    2\|D-D_m\|_E\cdot \|C_m\|_E  +\|C_mD_m-D_mC_m\|_E\to 0,\, m\to\infty,
  \end{eqnarray*}
  since $\|C_m\|_E$ is bounded, $\|D-D_m\|_E\to 0$ and $\|C_mD_m-D_mC_m\|_E\to 0$ when $m\to\infty$.
  Therefore, by applying Lemma \ref{LemmaOnseSequenceConvOptorsCommute} we conclude that $CD=DC$.

\noindent \textbf{Sufficiency (If):} Suppose that bounded linear operators  $C$ and $D$ commute. We first prove that the sequence of bounded linear operator $\{D_mC_m\}$ converges in norm to $CD$. In fact,
  \begin{eqnarray*}
    \|D_mC_m-DC\|_E&\leq & \|D_mC_m-D_mC+D_mC-DC\|_E \\
    &\leq & \|D_m\|_E\cdot\|C_m-C\|_E+ \|D_m-D\|_E\|C\|_E\to 0,
  \end{eqnarray*}
  when $m\to+\infty$, since the sequence $\{\|D_m\|\}$ is bounded and $\{C_m\}$, $\{D_m\}$ converge in norm to $C$ and $D$, respectively.
   By using this, commutativity of operators $C$, $D$, linearity, norm properties and Lemma \ref{LemmaOnseSequenceConvOptorsCommute}, we have
  the following:
  \begin{eqnarray*}
    \|C_mD_m-D_mC_m\|_E  &=&  \|C_mD_m-C_mD+C_mD+DC-DC-D_mC_m\|_E\\
          &=&\|C_m(D_m-D)+(C_m-C)D+(DC-D_mC_n)\|_E \\
    & \leq & \|C_m\|_E\cdot\|D_m-D\|_E+\|C_m-C\|_E\cdot \|D\|_E \\
    &  & +\|DC-D_mC_m\|_E \to 0,
  \end{eqnarray*}
  when $m\to \infty$, since the sequence $\{\|C_m\|\}$ is bounded, the sequences $\{C_m\}$, $\{D_m\}$, $\{D_mC_m\}$ converge to $C$, $D$ and $DC$, respectively.
\QEDB
\end{proof}

The following proposition was proved in \cite{DjinjaEtAll_LinItOpInGenSepKern}. Further, in this section, we will apply this result to compute and study convergence of commutator.
\begin{proposition}[\cite{DjinjaEtAll_LinItOpInGenSepKern}]
 Let $A:L_p(\mathbb{R},\mu)\to L_p(\mathbb{R},\mu)$, $B:L_p(\mathbb{R},\mu)\to L_p(\mathbb{R},\mu)$, $1\le p\le \infty$ be operators defined as
  \begin{eqnarray}\nonumber
     (A x)(t)&=&\int\limits_{\alpha_1}^{\beta_1}I_{[\alpha,\beta]}(t)(\theta_{A,1}\sin(\omega t)\cos(\omega s)+\theta_{A,2}\cos(\omega t)\cos(\omega s)\\ \label{OpAFourerSeq}
     && +\theta_{A,3}\sin(\omega t)\sin(\omega s)+\theta_{A,4}\cos(\omega t)\sin(\omega s))x(s)d\mu_s,
   \\ \nonumber
      (Bx)(t)&=&  \int\limits_{\alpha_1}^{\beta_1}I_{[\alpha,\beta]}(t)(\theta_{B,1}\sin(\omega t)\cos(\omega s)+\theta_{B,2}\cos(\omega t)\cos(\omega s) \\ \nonumber
     &&  +\theta_{B,3}\sin(\omega t)\sin(\omega s)+\theta_{B,4}\cos(\omega t)\sin(\omega s)) x(s)d\mu_s,
  \end{eqnarray}
 for almost every $t$, where $\theta_{A,1}, \theta_{B,1}, \theta_{B,3},\, \omega \in\mathbb{R}$,  $\delta \in\mathbb{R}\setminus\{0\}$, $\alpha,\beta,\beta_1,\beta_1 \in\mathbb{R}$, $\alpha_1<\beta_1$, $\alpha\le \alpha_1$, $\beta\ge \beta_1$, $I_E(\cdot)$ is the indicator function of the set $E$, the number $\frac{\omega }{\pi}(\beta_1-\alpha_1)\in \mathbb{Z}$ or $\frac{\omega }{\pi}(\beta_1+\alpha_1)\in \mathbb{Z}$.
 Set
 \begin{align}\label{ConstantSigma1Case2Omega}
     \sigma_1  &=\int\limits_{\alpha_1}^{\beta_1} (\sin(\omega s))^2 d\mu_s=\left\{\begin{array}{cc}
                         0, & \mbox{ if } \omega=0 \\
                         \frac{\beta_1-\alpha_1}{2}-\frac{\cos(\omega(\alpha_1+\beta_1))\sin(\omega(\beta_1-\alpha_1))}{2\omega}, & \mbox{ if } \omega\not=0,
                       \end{array} \right.\\ \label{ConstantSigma2Case2Omega}
    \sigma_2 &=\int\limits_{\alpha_1}^{\beta_1} (\cos(\omega s))^2 d\mu_s=\beta_1-\alpha_1-\sigma_1\\ \nonumber
    &=\left\{\begin{array}{cc}
                         \beta_1-\alpha_1, & \mbox{ if } \omega=0 \\
\frac{\beta_1-\alpha_1}{2}+\frac{\cos(\omega(\alpha_1+\beta_1))\sin(\omega(\beta_1-\alpha_1))}{2\omega}, & \mbox{ if } \omega\not=0.
                       \end{array} \right.
        \end{align}
Then,
  \begin{eqnarray}\label{CommutatorABFourierSequenceKernelOmegaOrtog}
  &&  (AB-BA)x(t) \\ \nonumber
   &&\resizebox{0.98\hsize}{!}{$\displaystyle = \int\limits_{\alpha_1}^{\beta_1}I_{[\alpha,\beta]}(t)\big((\theta_{A,3}\theta_{B,1}\sigma_1-\theta_{B,3}\theta_{A,1}\sigma_1
   +\theta_{A,1}\theta_{B,2}\sigma_2-\theta_{B,1}\theta_{A,2}\sigma_2)\sin(\omega t)\cos(\omega s) $}\\ \nonumber
  && \resizebox{0.92\hsize}{!}{$\displaystyle + (\theta_{B,1}\theta_{A,4} - \theta_{A,1}\theta_{B,4})\sigma_1\cos(\omega t)\cos(\omega s) + (\theta_{A,1}\theta_{B,4} - \theta_{B,1}\theta_{A,4})\sigma_2\sin(\omega t)\sin(\omega s)$}\\
  \nonumber
&&  \resizebox{0.92\hsize}{!}{$\displaystyle  +(\theta_{A,4}\theta_{B,3}\sigma_1-\theta_{B,4}\theta_{A,3}\sigma_1
   +\theta_{A,2}\theta_{B,4}\sigma_2-\theta_{B,2}\theta_{A,4}\sigma_2)\cos(\omega t)\sin(\omega s)\big)x(s)d\mu_s
 $}
  \end{eqnarray}
 for almost every $t$. Moreover if $\sigma_1\not=0$, $\sigma_2\not=0$, then $AB=BA$ if and only if $\theta_{A,3}\theta_{B,1}\sigma_1-\theta_{B,3}\theta_{A,1}\sigma_1=\theta_{B,1}\theta_{A,2}\sigma_2-\theta_{A,1}\theta_{B,2}\sigma_2$,
  $\theta_{A,4}\theta_{B,1}=\theta_{B,4}\theta_{A,1}$ and $\theta_{A,4}\theta_{B,3}\sigma_1-\theta_{B,4}\theta_{A,3}\sigma_1=\theta_{B,2}\theta_{A,4}\sigma_2-\theta_{A,2}\theta_{B,4}\sigma_2$.
\end{proposition}


\begin{remark}\label{RemarkConstantsEstimationOpFourierSeq}
Note that operator defined in \eqref{OpAFourerSeq} satisfy conditions of \cite[Theorem 6.18]{FollandRA}. In fact,
for almost every $t\in \mathbb{R}$ we have
\begin{multline*}
 \int\limits_{\alpha_1}^{\beta_1}I_{[\alpha,\beta]}(t)|\theta_{A,1}\sin(\omega t)\cos(\omega s)+\theta_{A,2}\cos(\omega t)\cos(\omega s)\\
      +\theta_{A,3}\sin(\omega t)\sin(\omega s)+\theta_{A,4}\cos(\omega t)\sin(\omega s)|d\mu_s \leq
     |\beta_1-\alpha_1|\sum_{i=1}^{4}|\theta_{A,i}|=\lambda_1
  \end{multline*}
  and for almost every $s\in [\alpha,\beta]$ we have
   \begin{align*}
  &\int\limits_{\mathbb{R}}I_{[\alpha,\beta]}(t)|\theta_{A,1}\sin(\omega t)\cos(\omega s)+\theta_{A,2}\cos(\omega t)\cos(\omega s)
\\
     &\hspace{5mm} +\theta_{A,3}\sin(\omega t)\sin(\omega s)+\theta_{A,4}\cos(\omega t)\sin(\omega s)|d\mu_t
     \\
      =& \int\limits_{\alpha}^{\beta}|\theta_{A,1}\sin(\omega t)\cos(\omega s)+\theta_{A,2}\cos(\omega t)\cos(\omega s)\\
     &\hspace{5mm} +\theta_{A,3}\sin(\omega t)\sin(\omega s)+\theta_{A,4}\cos(\omega t)\sin(\omega s)|d\mu_s
    \\
     \leq &
     |\beta-\alpha|\sum_{i=1}^{4}|\theta_{A,i}|=\lambda_2.
\end{align*}
Put $ \lambda=\max\{\lambda_1,\lambda_2\}<\infty$. By applying \cite[Theorem 6.18]{FollandRA}
one can conclude that $\|A\|_{L_p}\leq \lambda $, $1\leq p\leq\infty$.
\end{remark}

From now on we will present sequences of non-commuting operators that satisfy the monomial covariance commutation relation $CD=\delta DC^2$ which converge to commuting operators. These operators were constructed in \cite{DjinjaEtAll_LinItOpInGenSepKern} and here we study convergence of commutator when varying parameters.
\begin{theorem}\label{Proposition0ExampleTheoremIntOpRepGenKernLpConvSeqComOp}
 Let $A:L_p(\mathbb{R},\mu)\to L_p(\mathbb{R},\mu)$, $B:L_p(\mathbb{R},\mu)\to L_p(\mathbb{R},\mu)$, $1\le p\le \infty$ be operators defined as
  \begin{eqnarray*}
     (A x)(t)&=&\int\limits_{\alpha_1}^{\beta_1}I_{[\alpha,\beta]}(t)\theta_{A,1}\sin(\omega t)\cos(\omega s)x(s)d\mu_s,
   \\
      (Bx)(t)&=&  \int\limits_{\alpha_1}^{\beta_1}I_{[\alpha,\beta]}(t)\left(\theta_{B,1}\sin(\omega t)\cos(\omega s)+\theta_{B,3}\sin(\omega t)\sin(\omega s)\right) x(s)d\mu_s,
  \end{eqnarray*}
for almost every $t$, where $\theta_{A,1}, \theta_{B,1}, \theta_{B,3},\, \omega \in\mathbb{R}$,  $\delta \in\mathbb{R}\setminus\{0\}$, $\alpha,\beta,\alpha_1,\beta_1 \in\mathbb{R}$, $\alpha_1<\beta_1$, $\alpha\le \alpha_1$, $\beta\ge \beta_1$, $I_E(\cdot)$ is the indicator function of the set $E$ and, either the number $\frac{\omega }{\pi}(\beta_1-\alpha_1)\in \mathbb{Z}$ or $\frac{\omega }{\pi}(\beta_1+\alpha_1)\in \mathbb{Z}$,
 $\sigma_1,\sigma_2\in\mathbb{R}\setminus\{0\}$ are defined in \eqref{ConstantSigma1Case2Omega} and \eqref{ConstantSigma2Case2Omega}, respectively.
  Let $\{A_n\}: L_p(\mathbb{R},\mu)\to L_p(\mathbb{R},\mu)$, $\{B_n\}: L_p(\mathbb{R},\mu)\to L_p(\mathbb{R},\mu)$, with $1\le p \le \infty$ be sequences of operators
    \begin{eqnarray*}
     (A_n x)(t)&=&\int\limits_{\alpha_1}^{\beta_1}I_{[\alpha,\beta]}(t)\theta_{n}\sin(\omega t)\cos(\omega s)x(s)d\mu_s,
   \\
     (B_nx)(t)&=&  \int\limits_{\alpha_1}^{\beta_1}I_{[\alpha,\beta]}(t)\left(\theta_{B,1}\sin(\omega t)\cos(\omega s)+\varsigma_{n}\sin(\omega t)\sin(\omega s)\right) x(s)d\mu_s,
     \end{eqnarray*}
    for almost every $t$, $\{\theta_n\}$ and $\{\varsigma_n\}$ are number sequences. Let
         \begin{align*}
         & \Lambda=\big\{(\theta_{A,1},\theta_{B,1},\theta_{B,3}, \alpha_1,\beta_1,\alpha,\beta,\sigma_1,\sigma_2,\delta, \omega)\in\mathbb{R}^{11}: \delta\not=0, \alpha \leq \alpha_1, \beta\geq \beta_1,\\
         &  \int\limits_{\alpha_1}^{\beta_1}\sin(\omega s)\cos(\omega s)d\mu_s=0,  \sigma_1=\int\limits_{\alpha_1}^{\beta_1}(\sin(\omega s))^2d\mu_s\not=0,
         \sigma_2=\beta_1-\alpha_1-\sigma_1\not=0 \big\}.
         \end{align*}
         Then
    \begin{enumerate}[label=\textup{\arabic*.}, ref=\arabic*]
    \item\label{Proposition0ExampleTheoremIntOpRepGenKernLpConvSeqComOp:item1}
    $AB=\delta BA^2=0$, $AB=0$, $A_nB=\delta BA_n^2=0$, $AB_n=\delta B_nA^2=0$, $A_nB_n=\delta B_nA_n^2=0$ for each positive integer $n$.
        Moreover, it follows that for all $x\in L_p(\mathbb{R},\mu)$, $1\le p\le \infty$
       \begin{equation*}
       (AB-BA)x(t)=-(BAx)(t)=-\sigma_1\theta_{A,1}\theta_{B,3}
       \int\limits_{\alpha_1}^{\beta_1} I_{[\alpha,\beta]}(t)\sin(\omega t)\cos(\omega s) x(s)d\mu_s
       \end{equation*}
       for almost every $t$.
    \item\label{Proposition0ExampleTheoremIntOpRepGenKernLpConvSeqComOp:item2} if $\theta_n\to 0$ when $n\to \infty$, then $A_n\to 0$ (converge in norm) and so $A_nB-BA_n\to 0$ (converge in norm);
    \item\label{Proposition0ExampleTheoremIntOpRepGenKernLpConvSeqComOp:item3} if $\varsigma_n\to 0$ when $n\to \infty$, then $B_n\to \tilde{B}$ (converge in norm) defined as :
        \begin{equation*}
          (\tilde Bx)(t)= \int\limits_{\alpha_1}^{\beta_1}I_{[\alpha,\beta]}(t)\theta_{B,1}\sin(\omega t)\cos(\omega s)x(s)d\mu_s,
        \end{equation*}
        for almost every $t$. Moreover, $AB_n-B_nA\to 0$ (converge in norm) and $\tilde{B}$ commutes with $A$.
          \item\label{Proposition0ExampleTheoremIntOpRepGenKernLpConvSeqComOp:item4} if  $\theta_n\to 0$ and $\varsigma_n\to 0$ when $n\to \infty$, then $A_nB_n -B_nA_n \to 0$ (converge in norm).
            \item\label{Proposition0ExampleTheoremIntOpRepGenKernLpConvSeqComOp:item5} if $\theta_{B,3}   \theta_{A,1}\sigma_1 \tilde{\lambda}\to 0$ when
            $(\theta_{A,1},\theta_{B,1},\theta_{B,3},  \alpha_1,\beta_1,\alpha,\beta,\sigma_1,\sigma_2,\delta,\omega)\to \lambda_0 \in cl(\Lambda)$ (the closure of $\Lambda$), where
            $(\theta_{A,1},\theta_{B,1},\theta_{B,3},  \alpha_1,\beta_1,\alpha,\beta,\sigma_1,\sigma_2,\delta,\omega)\in \Lambda$ and, $\tilde{\lambda}=|\beta-\alpha|^{\frac{1}{p}}\cdot|\beta_1-\alpha_1|^{\frac{1}{q}}$, for $1< p < \infty$ with $\frac{1}{p}+\frac{1}{q}=1$, $\tilde{\lambda}=|\beta_1-\alpha_1|$, for $p = \infty$ and $\tilde{\lambda}=|\beta-\alpha|$, for $p = 1$, then for all $x\in L_p(\mathbb{R},\mu)$, $1\leq p\leq \infty$, it holds that $\|(AB-BA)x\|_{L_p}\to 0$. 
         \end{enumerate}
         \end{theorem}
\begin{proof}
 \noindent\ref{Proposition0ExampleTheoremIntOpRepGenKernLpConvSeqComOp:item1}. It follows from direct computation.

 \noindent\ref{Proposition0ExampleTheoremIntOpRepGenKernLpConvSeqComOp:item2}. By applying H\"older inequality we  have for all $x\in L_p(\mathbb{R},\mu)$, $1< p<\infty $ the following estimations:
         \begin{align*}
          \|A_nx\|^p_{L_p}= & \int\limits_{\mathbb{R}} \big| \int\limits_{\alpha_1}^{\beta_1}I_{[\alpha,\beta]}(t)\theta_n\sin(\omega t)\cos(\omega s)x(s)d\mu_s\big|^pd\mu_t\\
          = &\int\limits_{\mathbb{R}} |I_{[\alpha,\beta]}(t)\theta_n\sin(\omega t)|^p d\mu_t \big| \int\limits_{\alpha_1}^{\beta_1}\cos(\omega s)x(s)d\mu_s\big|^p\\
         \leq & |\,\theta_n|^p |\beta-\alpha| \big|\int\limits_{\mathbb{R}} I_{[\alpha_1,\beta_1]}(s) \cos(\omega s)x(s)d\mu_s\big|^p 
         \\
         \leq  &|\theta_n |^p |\beta-\alpha| \big(\int\limits_{\mathbb{R}} |I_{[\alpha_1,\beta_1]}(s)\cos(\omega s)|^qd\mu_s\big)^{p/q}\cdot \int\limits_{\mathbb{R}} |x(s)|^p d\mu_s
        \\
         \leq & |\theta_n |^p |\beta-\alpha|\cdot|\beta_1-\alpha_1|^{p/q}\|x\|^p_{L_p}.
         \end{align*}
         Therefore, for $1< p<\infty $ we have the following:
         \begin{equation*}
          \|A_n\|_{L_p}\le  \left|\theta_n\right|\cdot|\beta-\alpha|^{\frac{1}{p}}\cdot|\beta_1-\alpha_1|^{\frac{1}{q}}\to 0,
         \end{equation*}
         when $n\to \infty$, since $\theta_n\to 0$. By applying H\"older inequality we have for all $x\in L_\infty(\mathbb{R},\mu)$ the following:
         \begin{eqnarray*}
         && \|A_n x\|_{L_\infty}=  \mathop{\esssup}_{t\in \mathbb{R}} \big|\ \int\limits_{\alpha_1}^{\beta_1}I_{[\alpha,\beta]}(t) \theta_n \sin(\omega t)\cos(\omega s)x(s)d\mu_s\big|\\
         &&\leq|\theta_n|\cdot \int\limits_{\mathbb{R}} \big|I_{[\alpha_1,\beta_1]}(s) \cos(\omega s)\big|d\mu_s \|x\|_{L_\infty} \le |\theta_n||\beta_1-\alpha_1|\|x\|_{L_\infty}.
         \end{eqnarray*}
          Therefore, for $p=\infty $ we have the following:
         \begin{equation*}
          \|A_n\|_{L_\infty}\le  |\theta_n |\cdot|\beta_1-\alpha_1|\to 0,
         \end{equation*}
         when $n\to\infty$ since $\theta_n\to 0$.
          If $p=1$ we have for all $x\in L_1(\mathbb{R},\mu)$ the following:
         \begin{align*}
          \|A_n x\|_{L_1}=& \int\limits_{\mathbb{R}} \big|\ \int\limits_{\alpha_1}^{\beta_1}I_{[\alpha,\beta]}(t) \theta_n \sin(\omega t)\cos(\omega s)x(s)d\mu_s\big|d\mu_t\\
         \leq &|\theta_n| \int\limits_{\mathbb{R}} |I_{[\alpha,\beta]}(t)\sin(\omega t)|d\mu_t   \int\limits_{\alpha_1}^{\beta_1}\left| \cos(\omega s)x(s)\right|d\mu_s
         \\
         \le & |\theta_n|\cdot |\beta-\alpha|\cdot\|x\|_{L_1}.
         \end{align*}
         Therefore, for $p=1 $ we have the following:
         \begin{equation*}
          \|A_n\|_{L_1}\le  |\theta_n |\cdot|\beta-\alpha|\to 0,
         \end{equation*}
         when $n\to\infty$ since $\theta_n\to 0$. Therefore, $A_n\to 0$ converge in norm in $L_p$, $1\le p\le \infty$.

         Now we prove the convergence for the sequence $\{A_nB-BA_n\}$. By noticing that $AB_n=0$, using item \ref{Proposition0ExampleTheoremIntOpRepGenKernLpConvSeqComOp:item2}),  norm properties and above result we have for each  $1\le p \le \infty,$
         \begin{equation*}
           \|A_nB-BA_n\|_{L_p}=\|BA_n\|_{L_p}\le \|B\|_{L_p}\cdot \|A_n\|_{L_p}\to 0,
         \end{equation*}
         when $n\to \infty$, since $B$ is a bounded linear operator in $L_p$  and $\|A_n\|\to 0$.

         \noindent\ref{Proposition0ExampleTheoremIntOpRepGenKernLpConvSeqComOp:item3}.
         We have
         \begin{eqnarray*}
           (B_n-\tilde B)x(t)&=&\int\limits_{\alpha_1}^{\beta_1}I_{[\alpha_1,\beta_1]}(t)\varsigma_n \sin(\omega t)\sin(\omega s)x(s)d\mu_s,
         \end{eqnarray*}
         almost everywhere. Therefore, by applying the same procedure as on item \ref{Proposition0ExampleTheoremIntOpRepGenKernLpConvSeqComOp:item2} for $1< p<\infty $ we have the following:
         \begin{equation*}
          \|B_n-\tilde B\|_{L_p}\le  \left|\varsigma_n\right|\cdot|\beta-\alpha|^{\frac{1}{p}}\cdot|\beta_1-\alpha_1|^{\frac{1}{q}}\to 0,
         \end{equation*}
         when $n\to \infty$, since $\varsigma_n\to 0$. For all $x\in L_\infty(\mathbb{R},\mu)$ we have the following:
        \begin{equation*}
          \|B_n-\tilde B\|_{L_\infty}\le  \left|\varsigma_n\right|\cdot|\beta_1-\alpha_1|\to 0,
         \end{equation*}
         when $n\to\infty$ since $\varsigma_n\to 0$. Moreover, for all $x\in L_1(\mathbb{R},\mu)$ we have the following:
        \begin{equation*}
          \|B_n-\tilde B\|_{L_1}\le  \left|\varsigma_n\right|\cdot|\beta-\alpha|\to 0,
         \end{equation*}
         when $n\to\infty$ since $\varsigma_n\to 0$. Therefore, $B_n\to\tilde B$ converge in norm in $L_p$, $1\le p\le \infty$.

         Now we prove the convergence for the sequence $\{AB_n-B_nA\}$. By direct computation we have
          $A\tilde B=\tilde B A=0$, in fact:
          \begin{eqnarray*}
            &&(A\tilde{B}x)(t)=\int\limits_{\alpha_1}^{\beta_1} I_{[\alpha,\beta]}(t)\theta_{A,1}\sin(\omega t)\cos(\omega s)\\
            && \hspace{3cm} \cdot\big(\int\limits_{\alpha_1}^{\beta_1}I_{[\alpha,\beta]}(s)\theta_{B,1}\sin(\omega s)\cos(\omega \tau)
            x(\tau)d\mu_\tau \big) d\mu_s
            \\
            &&
            =\theta_{A,1}\theta_{B,1}\int\limits_{\alpha_1}^{\beta_1} I_{[\alpha,\beta]}(t)\sin(\omega t)\big(\int\limits_{\alpha_1}^{\beta_1}\sin(\omega s)\cos(\omega s)d\mu_s \big)
            \cos(\omega \tau)x(\tau)d\mu_\tau=0,\\
            &&(\tilde{B}Ax)(t)=\int\limits_{\alpha_1}^{\beta_1} I_{[\alpha,\beta]}(t)\theta_{B,1}\sin(\omega t)\cos(\omega s)\\
            &&\hspace{2cm} \cdot\big(\int\limits_{\alpha_1}^{\beta_1}I_{[\alpha,\beta]}(s)\theta_{A,1}\sin(\omega s)\cos(\omega \tau)
             x(\tau)d\mu_\tau \big) d\mu_s
             \\
             &&
            =\theta_{A,1}\theta_{B,1}\int\limits_{\alpha_1}^{\beta_1} I{[\alpha,\beta]}(t)\sin(\omega t)\big(\int\limits_{\alpha_1}^{\beta_1}\sin(\omega s)\cos(\omega s)d\mu_s \big)
             \cos(\omega \tau)x(\tau)d\mu_\tau=0,
          \end{eqnarray*}
          for almost every $t$, because $\int\limits_{\alpha_1}^{\beta_1} \cos(\omega s)\sin(\omega s)d\mu_s=0$.
          Since $B_n\to \tilde{B}$, by Lemma \ref{LemmaOnseSequenceConvOptorsCommute}, we conclude that $\{AB_n-B_nA\}$ converges (converge in norm) to 0.

         \noindent\ref{Proposition0ExampleTheoremIntOpRepGenKernLpConvSeqComOp:item4}. By noticing that $A_nB_n=0$, using results from items
         \ref{Proposition0ExampleTheoremIntOpRepGenKernLpConvSeqComOp:item2}, \ref{Proposition0ExampleTheoremIntOpRepGenKernLpConvSeqComOp:item3} and norms properties we have for all $1\leq p\leq \infty$ the following:
         \begin{equation*}
           \|A_nB_n-B_nA_n\|_{L_p}= \|B_nA_n\|_{L_p}
           \leq  \|B_n\|_{L_p}\|A_n\|_{L_p}\to 0,
         \end{equation*}
         when $n\to\infty$, since $\|B_n\|_{L_p}$ is bounded (because $B_n$ converges in norm to $\tilde{B}$) and $\|A_n\|\to 0$.

        \noindent\ref{Proposition0ExampleTheoremIntOpRepGenKernLpConvSeqComOp:item5}. By following the same procedure of the item \ref{Proposition0ExampleTheoremIntOpRepGenKernLpConvSeqComOp:item2}, the operator $AB-BA$ has the following estimation. For all $x\in L_p(\mathbb{R},\mu)$ and for all 
        $$(\theta_{A,1},\theta_{B,1},\theta_{B,3},  \alpha_1,\beta_1,\alpha,\beta,\sigma_1,\sigma_2,\delta,\omega)\in \Lambda
        $$
        we have
         \begin{equation*}
           \|(AB-BA)x\|_{L_p}\leq |\theta_{B,3}\theta_{A,1}\sigma_1|\cdot \tilde{\lambda} \cdot \|x\|_{L_p}\to 0,
         \end{equation*}
           when   $\theta_{B,3}\theta_{A,1}\sigma_1\tilde{\lambda} \to 0$, where $\tilde{\lambda}=|\beta-\alpha|^{\frac{1}{p}}\cdot|\beta_1-\alpha_1|^{\frac{1}{q}}$, if $1< p < \infty$, $\tilde{\lambda}=|\beta_1-\alpha_1|$, if $p=\infty$ and $\tilde{\lambda}=|\beta-\alpha|$, if $p=1$. 
          \QEDB
\end{proof}

\begin{theorem}
 Let $A:L_p(\mathbb{R},\mu)\to L_p(\mathbb{R},\mu)$, $B:L_p(\mathbb{R},\mu)\to L_p(\mathbb{R},\mu)$, $1\le p\le \infty$ be operators defined as
  \begin{eqnarray*}
     (Ax)(t)&=&\int\limits_{\alpha_1}^{\beta_1}I_{[\alpha,\beta]}(t)\big(\theta_{A,1}\sin(\omega t)\cos(\omega s)+ \theta_{A,2}\cos(\omega t)\cos(\omega s) \\
          & & \hspace{1cm}+\frac{1}{\delta\sigma_1}\sin(\omega t)\sin(\omega s)\big)x(s)d\mu_s,
   \\
      (Bx)(t)&=&\int\limits_{\alpha_1}^{\beta_1}I_{[\alpha,\beta]}(t)\big(\theta_{B,1}\sin(\omega t)\cos(\omega s) \\
         && \hspace{1cm} -\frac{\theta_{B,1}(\delta\sigma_2\theta_{A,2}-1)}{\delta\theta_{A,1}\sigma_1}\sin(\omega t)\sin(\omega s)\big)
         x(s)d\mu_s,
  \end{eqnarray*}
 for almost every $t$, where $ \theta_{B,1}, \theta_{B,3},\,\omega \in\mathbb{R}$,  $\theta_{A,1},\delta\in\mathbb{R}\setminus\{0\}$, $\alpha,\beta,\alpha_1,\beta_1 \in\mathbb{R}$, $\alpha_1<\beta_1$, $\alpha\le \alpha_1$, $\beta\ge \beta_1$, $I_E(\cdot)$ is the indicator function of the set $E$ and, either the number $\frac{\omega }{\pi}(\beta_1-\alpha_1)\in \mathbb{Z}$ or $\frac{\omega }{\pi}(\beta_1+\alpha_1)\in \mathbb{Z}$, $\sigma_1,\sigma_2\in\mathbb{R}\setminus\{0\}$ are defined in \eqref{ConstantSigma1Case2Omega} and \eqref{ConstantSigma2Case2Omega}, respectively.
  Let $\{A_n\}: L_p(\mathbb{R},\mu)\to L_p(\mathbb{R},\mu)$, $\{B_n\}: L_p(\mathbb{R},\mu)\to L_p(\mathbb{R},\mu)$, with $1\le p \le \infty$ be sequences of operators
    \begin{eqnarray*}
     (A_nx)(t)&=&\int\limits_{\alpha_1}^{\beta_1}I_{[\alpha,\beta]}(t)\big(\theta_{A,1}\sin(\omega t)\cos(\omega s)+ \theta_{n}\cos(\omega t)\cos(\omega s) \\
          & & \hspace{1cm}+\frac{1}{\delta\sigma_1}\sin(\omega t)\sin(\omega s)\big)x(s)d\mu_s,
   \\
      (B_nx)(t)&=&\int\limits_{\alpha_1}^{\beta_1}I_{[\alpha,\beta]}(t)\big(\varsigma_{n}\sin(\omega t)\cos(\omega s) \\
         &&\hspace{1cm}  -\frac{\varsigma_{n}(\delta\sigma_2\theta_{A,2}-1)}{\delta\theta_{A,1}\sigma_1}\sin(\omega t)\sin(\omega s)\big)
         x(s)d\mu_s,
     \end{eqnarray*}
    for almost every $t$, $\{\theta_n\}$ and $\{\varsigma_n\}$ are number sequences. Let
         \begin{align*}
          & \Lambda =\big\{(\theta_{A,1},\theta_{A,2},\theta_{B,1}, \alpha_1,\beta_1,\alpha,\beta,\sigma_1,\sigma_2,\delta,\omega)\in\mathbb{R}^{11}: \theta_{A,1}\not=0, \delta\not=0,\alpha \leq \alpha_1, \beta\geq \beta_1,\\
         &  \int\limits_{\alpha_1}^{\beta_1}\sin(\omega s)\cos(\omega s)d\mu_s=0,   \sigma_1=\int\limits_{\alpha_1}^{\beta_1}(\sin(\omega s))^2d\mu_s\neq0,
         \sigma_2=\beta_1-\alpha_1-\sigma_1\neq0 \big\}.
         \end{align*}
         Then
    \begin{enumerate}[label=\textup{\arabic*.}, ref=\arabic*]
    \item\label{Proposition1ExampleTheoremIntOpRepGenKernLpConvSeqComOp:item1}  $AB=\delta BA^2$, $A_nB=\delta BA_n^2$, $AB_n=\delta B_nA^2 $ for each positive integer $n$. Moreover, for all $x\in L_p(\mathbb{R},\mu)$, $1\le p\le \infty$ we have
       \begin{multline*}
         (AB-BA)x(t)=\frac{\theta_{B,1}(\delta \theta_{A,2}\sigma_2-1)}{\delta^2 \theta_{A,1}\sigma_1}\int\limits_{\alpha_1}^{\beta_1} I_{[\alpha,\beta]}(t)\big(\sin(\omega t)\cos(\omega s)
         \\
         -\sin(\omega t)\sin(\omega s)\big)x(s)d\mu_s,
       \end{multline*}
      for almost every $t$.
           \item\label{Proposition1ExampleTheoremIntOpRepGenKernLpConvSeqComOp:item2} if $\theta_n\to \frac{1}{\delta \sigma_2}$ when $n\to \infty$, then $A_n\to \tilde{A}$ (converges in norm)
               defined as
            \begin{multline*}
             (\tilde A x)(t)=\int\limits_{\alpha_1}^{\beta_1}I_{[\alpha,\beta]}(t)\big(\theta_{A,1}\sin(\omega t)\cos(\omega s)+ \frac{1}{\delta\sigma_2}\cos(\omega t)\cos(\omega s) \\
           +\frac{1}{\delta\sigma_1}\sin(\omega t)\sin(\omega s)\big)x(s)d\mu_s,
            \end{multline*}
           for almost every $t$, so $A_nB-BA_n\to 0$ (converges in norm) and $\tilde{A}B=B\tilde{A}$;
          \item\label{Proposition1ExampleTheoremIntOpRepGenKernLpConvSeqComOp:item3} if $\varsigma_n\to 0$ when $n\to \infty$, then $B_n\to 0$ (converges in norm) and $AB_n-B_nA\to 0$ (converges in norm).

          \item\label{Proposition1ExampleTheoremIntOpRepGenKernLpConvSeqComOp:item4} if $(\frac{\theta_{B,1}(\delta \theta_{A,2}\sigma_2-1)}{\delta^2 \theta_{A,1}\sigma_1})\tilde{\lambda}\to 0$ when
            $$(\theta_{A,1},\theta_{A,2},\theta_{B,2},  \alpha_1,\beta_1,\alpha,\beta,\sigma_1,\sigma_2,\delta,\omega)\to \lambda_0\in cl(\Lambda) \mbox{ ( the closure of } \Lambda),
            $$
            where
            $(\theta_{A,1},\theta_{A,2},\theta_{B,2},  \alpha_1,\beta_1,\alpha,\beta,\sigma_1,\sigma_2,\delta,\omega)\in\Lambda$  and, $\tilde{\lambda}=2|\beta-\alpha|$ for $p=1$,  $\tilde{\lambda}=2|\beta-\alpha|^{\frac{1}{p}}|\beta_1-\alpha_1|^{\frac{1}{q}}$ for $1<p <\infty$ and  $\tilde{\lambda}=2|\beta_1-\alpha_1|$ for $p=\infty$,  then for all $x\in L_p(\mathbb{R},\mu)$, $1\leq p\leq \infty$, it holds that $\|(AB-BA)x\|_{L_p}\to 0$. 
         \end{enumerate}
         \end{theorem}

\begin{proof}
\noindent\ref{Proposition1ExampleTheoremIntOpRepGenKernLpConvSeqComOp:item1}. It follows from direct computation.

\noindent\ref{Proposition1ExampleTheoremIntOpRepGenKernLpConvSeqComOp:item2}.  We have
         \begin{eqnarray*}
           (A_n-\tilde A)x(t)&=&\int\limits_{\alpha_1}^{\beta_1}I_{[\alpha_1,\beta_1]}(t)\left(\theta_n-\frac{1}{\delta \sigma_2}\right) \cos(\omega t)\cos(\omega s)x(s)d\mu_s,
         \end{eqnarray*}
         for almost every $t$. Moreover, by applying H\"older inequality,  we  have for all $x\in L_p(\mathbb{R},\mu)$, $1< p<\infty $, the following estimations:
         \begin{eqnarray*}
         && \resizebox{0.85\hsize}{!}{$\|(A_n-\tilde A)x\|^p_{L_p}=  \int\limits_{\mathbb{R}} \big| \int\limits_{\alpha_1}^{\beta_1}I_{[\alpha,\beta]}(t)\left(\theta_n-\frac{1}{\delta \sigma_2}\right)\cos(\omega t)\cos(\omega s)x(s)d\mu_s\big|^pd\mu_t$}\\
         &&\resizebox{0.982\hsize}{!}{$\displaystyle \leq \big|\theta_n-\frac{1}{\delta \sigma_2}\big|^p |\beta-\alpha| \big|\int\limits_{\mathbb{R}} I_{[\alpha_1,\beta_1]}(s) \cos(\omega s)x(s)d\mu_s\big|^p
         \leq\big|\theta_n-\frac{1}{\delta \sigma_2}\big|^p |\beta-\alpha||\beta_1-\alpha_1|^{\frac{p}{q}}\|x\|^p_{L_p}.$}
         \end{eqnarray*}
         Therefore, for $1< p<\infty $ we have the following:
         \begin{equation*}
          \|A_n-\tilde A\|_{L_p}\le  \big|\theta_n-\frac{1}{\delta \sigma_2}\big|\cdot|\beta-\alpha|^{\frac{1}{p}}\cdot|\beta_1-\alpha_1|^{\frac{1}{q}}\to 0,
         \end{equation*}
         when $n\to \infty$, since $\theta_n\to \frac{1}{\delta \sigma_2}$. By applying H\"older inequality we have for all $x\in L_\infty(\mathbb{R},\mu)$ the following:
         \begin{eqnarray*}
         && \|(A_n-\tilde A) x\|_{L_\infty}= \mathop{\esssup}_{t\in \mathbb{R}} \big|\ \int\limits_{\alpha_1}^{\beta_1}I_{[\alpha,\beta]}(t)\left(\theta_n-\frac{1}{\delta \sigma_2}\right) \cos(\omega t)\cos(\omega s)x(s)d\mu_s\big|\\
         &&\leq\big|\theta_n-\frac{1}{\delta \sigma_2}\big|\cdot \int\limits_{\mathbb{R}} \big|I_{[\alpha_1,\beta_1]}(s) \cos(\omega s)\big|d\mu_s\cdot \|x\|_{L_\infty} \le \big|\theta_n-\frac{1}{\delta \sigma_2}\big|\cdot|\beta_1-\alpha_1|\cdot\|x\|_{L_\infty}.
         \end{eqnarray*}
          Therefore, for $p=\infty $ we have the following:
         \begin{equation*}
          \|A_n-\tilde A\|_{L_\infty}\le  \big|\theta_n-\frac{1}{\delta \sigma_2}\big|\cdot|\beta_1-\alpha_1|\to 0,
         \end{equation*}
         when $n\to\infty$ since $\theta_n\to \frac{1}{\delta \sigma_2}$.
         We have for all $x\in L_1(\mathbb{R},\mu)$ the following:
         \begin{align*}
        &  \|(A_n-\tilde A) x\|_{L_1}= \int\limits_{\mathbb{R}} \big|\ \int\limits_{\alpha_1}^{\beta_1}I_{[\alpha,\beta]}(t)\left(\theta_n-\frac{1}{\delta \sigma_2}\right) \cos(\omega t)\cos(\omega s)x(s)d\mu_s\big|d\mu_t\\
        & \leq  \big|\theta_n-\frac{1}{\delta \sigma_2}\big|\int\limits_{\mathbb{R}} |I_{[\alpha,\beta]}(t)\cos(\omega t)|d\mu_t\cdot \int\limits_{\mathbb{R}} \big|I_{[\alpha_1,\beta_1]}(s) \cos(\omega s)x(s) \big|d\mu_s
         \\
        & \leq  \big|\theta_n-\frac{1}{\delta \sigma_2}\big|\cdot|\beta-\alpha|\cdot\|x\|_{L_1}.
         \end{align*}
         Therefore, for $p=1 $ we have the following:
         \begin{equation*}
          \|A_n-\tilde A\|_{L_1}\le  \big|\theta_n-\frac{1}{\delta \sigma_2}\big|\cdot|\beta_1-\alpha_1|\to 0,
         \end{equation*}
         when $n\to\infty$ since $\theta_n\to \frac{1}{\delta \sigma_2}$.
         Therefore, $A_n\to\tilde A$ converges in norm in $L_p$, with $1\le p\le \infty$.

         Now we prove the convergence for the sequence $\{A_nB-BA_n\}$. By applying H\"older inequality we have for all $x\in L_p(\mathbb{R},\mu)$, $1< p<\infty $ the following:
         \begin{align*}
          & \|(A_nB-BA_n)x\|^p_{L_p}= \int\limits_{\mathbb{R}}\big|\left(\frac{\theta_{B,1}(\delta\sigma_2\theta_n-1)}{\delta^2 \theta_{A,1}\sigma_1}\right)\int\limits_{\alpha_1}^{\beta_1} I_{[\alpha,\beta]}(t)\sin(\omega t)\\
         & \hspace{4cm}
         \cdot(\cos(\omega s)-\sin(\omega s))x(s)d\mu_s\big|^p d\mu_t      \\
         &
            \leq \left|\frac{\theta_{B,1}(\delta\sigma_2\theta_n-1)}{\delta^2 \theta_{A,1}\sigma_1}\right|^p\cdot
           |\beta-\alpha|
           \big|\int\limits_{\mathbb{R}} I_{[\alpha_1,\beta_1]}(s)(\cos(\omega s)-\sin(\omega s))x(s)d\mu_s\big|^p
          \\
          & \quad   \leq
           \left|\frac{\theta_{B,1}(\delta\sigma_2\theta_n-1)}{\delta^2 \theta_{A,1}\sigma_1}\right|^p |\beta-\alpha| \cdot
            2^p|\beta_1-\alpha_1|^{\frac{p}{q}}\cdot \|x\|^p.
         \end{align*}
         Therefore, for $1< p <\infty$ we have
         \begin{equation*}
         \|A_nB-BA_n\|_{L_p} \le  \left|\frac{\theta_{B,1}(\delta\sigma_2\theta_n-1)}{\delta^2 \theta_{A,1}\sigma_1}\right|\cdot
         |\beta-\alpha|^{\frac{1}{p}}\cdot 2|\beta_1-\alpha_1|^{\frac{1}{q}}\to 0,\, n\to\infty,
         \end{equation*}
         since $\theta_n\to \frac{1}{\delta \sigma_2}$.
         By applying H\"older inequality, we have for  all $x\in L_\infty(\mathbb{R},\mu)$ the following:
         \begin{align*}
         & \|(A_nB-BA_n)x\|_{L_\infty}=\mathop{\esssup}_{t\in \mathbb{R}} \big|\ \frac{\theta_{B,1}(\delta\sigma_2\theta_n-1)}{\delta^2 \theta_{A,1}\sigma_1}\cdot I_{[\alpha,\beta]}(t)\sin(\omega t) \\
         &\hspace{3cm}  \int\limits_{\alpha_1}^{\beta_1}(\cos(\omega s)-\sin(\omega s))x(s)d\mu_s\big|
          \\
          &
          \leq  \left|\frac{\theta_{B,1}(\delta\sigma_2\theta_n-1)}{\delta^2 \theta_{A,1}\sigma_1}\right| \int\limits_{\mathbb{R}}| I_{[\alpha_1,\beta_1]}(s)
         (\cos(\omega s)-\sin(\omega s))|d\mu_s \|x\|_{L_\infty}
           \\
           &
           \leq
           \left|\frac{\theta_{B,1}(\delta\sigma_2\theta_n-1)}{\delta^2 \theta_{A,1}\sigma_1}\right|\cdot 2|\beta_1-\alpha_1|\|x\|_{L_\infty}.
         \end{align*}
         Therefore, for $p=\infty$ we have
           \begin{equation*}
           \|A_nB-BA_n\|_{L_\infty}\le \left|\frac{\theta_{B,1}(\delta\sigma_2\theta_n-1)}{\delta^2 \theta_{A,1}\sigma_1}\right|\cdot 2|\beta_1-\alpha_1|\to 0,\, n\to\infty,
         \end{equation*}
         because $\theta_n\to \frac{1}{\delta\sigma_2}$. We have for  all $x\in L_1(\mathbb{R},\mu)$ the following:
         \begin{align*}
         & \|(A_nB-BA_n)x\|_{L_1}=\int\limits_{\mathbb{R}}  \left|\ \frac{\theta_{B,1}(\delta\sigma_2\theta_n-1)}{\delta^2 \theta_{A,1}\sigma_1}\cdot I_{[\alpha,\beta]}(t)\sin(\omega t) \right. \\
         &\hspace{4cm} \cdot \left.\int\limits_{\alpha_1}^{\beta_1}(\cos(\omega s)-\sin(\omega s))x(s)d\mu_s\right| d\mu_t
          \\
          &
          \leq  \left|\frac{\theta_{B,1}(\delta\sigma_2\theta_n-1)}{\delta^2 \theta_{A,1}\sigma_1}\right| |\beta-\alpha| \int\limits_{\mathbb{R}}| I_{[\alpha_1,\beta_1]}(s)
         (\cos(\omega s)-\sin(\omega s))x(s)|d\mu_s
          \\ & \leq
            \left|\frac{\theta_{B,1}(\delta\sigma_2\theta_n-1)}{\delta^2 \theta_{A,1}\sigma_1}\right|\cdot 2|\beta-\alpha|\|x\|_{L_1}.
         \end{align*}
         Therefore, for $p=1$ we have
           \begin{equation*}
           \|A_nB-BA_n\|_{L_1}\leq \left|\frac{\theta_{B,1}(\delta\sigma_2\theta_n-1)}{\delta^2 \theta_{A,1}\sigma_1}\right|\cdot 2|\beta-\alpha|\to 0,\, n\to\infty,
         \end{equation*}
         because $\theta_n\to \frac{1}{\delta\sigma_2}$. Since the sequence of operators  $\{A_n\}$  converges in norm to $\tilde{A}$, then by Lemma \ref{LemmaOnseSequenceConvOptorsCommute}, $\tilde{A}$ commutes with $B$.

       \noindent\ref{Proposition1ExampleTheoremIntOpRepGenKernLpConvSeqComOp:item3}.
         We have
         \begin{eqnarray*}
          && (B_n)x(t)=\int\limits_{\alpha_1}^{\beta_1}I_{[\alpha_1,\beta_1]}(t) \varsigma_n \sin(\omega t)\cos(\omega s)x(s)d\mu_s\\
          & &- \int\limits_{\alpha_1}^{\beta_1}I_{[\alpha_1,\beta_1]}(t) \frac{\varsigma_{n}(\delta\theta_{A,2}\sigma_2-1)}{\delta\theta_{A,1}\sigma_1} \sin(\omega t)\sin(\omega s)x(s)d\mu_s
          =(B_{n,1})x(t)+(B_{n,2})x(t),
         \end{eqnarray*}
         respectively, for almost every $t$. Moreover, by applying H\"older inequality we  have for all $x\in L_p(\mathbb{R},\mu)$, $1< p<\infty $ the following estimations:
         \begin{eqnarray*}
          && \|(B_{n,1})x\|^p_{L_p}=  \int\limits_{\mathbb{R}} \big| \int\limits_{\alpha_1}^{\beta_1}I_{[\alpha,\beta]}(t)\varsigma_n  \sin(\omega t)\cos(\omega s)x(s)d\mu_s\big|^pd\mu_t\\
         &&\leq |\varsigma_n |^p |\beta-\alpha| \big|\int\limits_{\mathbb{R}} I_{[\alpha_1,\beta_1]}(s) \cos(\omega s)x(s)d\mu_s\big|^p  \leq |\varsigma_n |^p |\beta-\alpha||\beta_1-\alpha_1|^{\frac{p}{q}}\|x\|^p_{L_p}.
         \end{eqnarray*}
         Therefore, for $1< p<\infty $ we have the following norm estimation:
         \begin{equation*}
          \|B_{n,1}\|_{L_p}\le  |\varsigma_n |\cdot|\beta-\alpha|^\frac{1}{p}\cdot|\beta_1-\alpha_1|^{\frac{1}{q}}.
         \end{equation*}
         By applying H\"older inequality we have for all $x\in L_\infty(\mathbb{R},\mu)$ the following:
         \begin{eqnarray*}
         && \|(B_{n,1})x\|_{L_\infty}= \mathop{\esssup}_{t\in \mathbb{R}} \big| \int\limits_{\alpha_1}^{\beta_1}I_{[\alpha,\beta]}(t)\varsigma_n  \sin(\omega t)\cos(\omega s)x(s)d\mu_s\big|\\
         &&\leq
         \left|\varsigma_n \right|\cdot \int\limits_{\mathbb{R}} \left|I_{[\alpha_1,\beta_1]}(s) \cos(\omega s)\right|d\mu_s\cdot \|x\|_{L_\infty} \le \left|\varsigma_n \right|\cdot|\beta_1-\alpha_1|\cdot\|x\|_{L_\infty}.
         \end{eqnarray*}
          Therefore, for $ p=\infty $ we have the following norm estimation:
         \begin{equation*}
          \|B_{n,1}\|_{L_\infty}\leq  \left|\varsigma_n \right|\cdot|\beta_1-\alpha_1|.
         \end{equation*}
         We have for all $x\in L_1(\mathbb{R},\mu)$ the following:
         \begin{eqnarray*}
         && \|(B_{n,1})x\|_{L_1}= \int\limits_{\mathbb{R}}\big| \int\limits_{\alpha_1}^{\beta_1}I_{[\alpha,\beta]}(t)\varsigma_n  \sin(\omega t)\cos(\omega s)x(s)d\mu_s\big|d\mu_t\\
         &&\leq
         |\varsigma_n |\cdot |\beta-\alpha| \int\limits_{\mathbb{R}} \big|I_{[\alpha_1,\beta_1]}(s) \cos(\omega s)x(s)\big|d\mu_s \leq |\varsigma_n |\cdot|\beta-\alpha|\cdot\|x\|_{L_1}.
         \end{eqnarray*}
          Therefore, for $ p=1 $ we have the following norm estimation:
         \begin{equation*}
          \|B_{n,1}\|_{L_1}\leq  \left|\varsigma_n \right|\cdot|\beta-\alpha|.
         \end{equation*}
         Analogously, we  have for  $1< p<\infty $ the following estimations:
         \begin{equation*}
          \|(B_{n,2})\|_{L_p}\le  \left|\frac{\varsigma_{n}(\delta\theta_{A,2}\sigma_2-1)}{\delta\theta_{A,1}\sigma_1}\right| \cdot|\beta-\alpha|^{\frac{1}{p}}\cdot|\beta_1-\alpha_1|^{\frac{1}{q}}
         \end{equation*}
         and for $p=\infty$ and $p=1$ we have, respectively,
         \begin{equation*}
          \|(B_{n,2})\|_{L_\infty}\le  \left|\frac{\varsigma_{n}(\delta\theta_{A,2}\sigma_2-1)}{\delta\theta_{A,1}\sigma_1}\right| \cdot|\beta_1-\alpha_1|,\, \|(B_{n,2})\|_{L_1}\le  \left|\frac{\varsigma_{n}(\delta\theta_{A,2}\sigma_2-1)}{\delta\theta_{A,1}\sigma_1}\right| \cdot|\beta-\alpha|.
         \end{equation*}
         Therefore, by Minkowski inequality \cite{AdamsG,FollandRA,KolmogorovVol2} we have for $1< p <\infty$
         \begin{eqnarray*}
             \|B_{n}\|_{L_p}&\leq & \|B_{n,1}\|_{L_p}+\|B_{n,2}\|_{L_p} \\
             &\leq & |\varsigma_n |\cdot |\beta-\alpha|^{\frac{1}{p}}\cdot|\beta_1-\alpha_1|^{\frac{1}{q}}  \left(1+\frac{|\delta\theta_{A,2}\sigma_2-1|}{|\delta\theta_{A,1}\sigma_1|}
               \right)\to 0,
         \end{eqnarray*}
         when $n\to \infty$, since $\varsigma_n\to 0$. On the other hand, by triangle inequality we have for $p=\infty$ and $p=1$
         \begin{eqnarray*}
             \|B_n\|_{L_\infty}&\leq & \|B_{n,1}\|_{L_\infty}+\|B_{n,2}\|_{L_\infty} \\
             &\leq & |\varsigma_n| \cdot|\beta_1-\alpha_1|  \left(1+\frac{|\delta\theta_{A,2}\sigma_2-1|}{|\delta\theta_{A,1}\sigma_1|} \right)\to 0,\\
             \|B_n\|_{L_1}&\leq & \|B_{n,1}\|_{L_1}+\|B_{n,2}\|_{L_1} \\
             &\leq & |\varsigma_n| \cdot|\beta-\alpha|  \left(1+\frac{|\delta\theta_{A,2}\sigma_2-1|}{|\delta\theta_{A,1}\sigma_1|} \right)\to 0,
         \end{eqnarray*}
         when $n\to \infty$ since $\varsigma_n\to 0$. Therefore, $B_n$ converges in norm to $\tilde{B}$ in $L_p$, $1\le p\le \infty$.

         In order to prove that $\{AB_n-B_nA\}$ converges in norm to 0. We apply the same procedure as on item \ref{Proposition1ExampleTheoremIntOpRepGenKernLpConvSeqComOp:item2} for $1< p<\infty $, so we get for $1< p<\infty$
         \begin{equation*}
           \|AB_n-B_nA\|_{L_p}\le \left|\frac{\varsigma_n(\delta\sigma_2\theta_{A,2}-1)}{\delta^2 \theta_{A,1}\sigma_1}\right|\cdot |\beta-\alpha|^{\frac{1}{p}}\cdot 2|\beta_1-\alpha_1|^{\frac{1}{q}}\to 0,\, n\to\infty,
         \end{equation*}
         since $\varsigma_n\to 0$. For $p=\infty$ and $p=1$we get
           \begin{align*}
           &\|AB_n-B_nA\|_{L_\infty}\le \left|\frac{\varsigma_n(\delta\sigma_2\theta_{A,2}-1)}{\delta^2 \theta_{A,1}\sigma_1}\right|\cdot 2|\beta_1-\alpha_1|\to 0,\, n\to\infty,
           \\
           &\|AB_n-B_nA\|_{L_1}\le \left|\frac{\varsigma_n(\delta\sigma_2\theta_{A,2}-1)}{\delta^2 \theta_{A,1}\sigma_1}\right|\cdot 2|\beta-\alpha|\to 0,\, n\to\infty,
         \end{align*}
         because $\varsigma_n\to 0$.  Since the sequence of operators  $\{B_n\}$ converges in norm to $\tilde{B}$, then by Lemma \ref{LemmaOnseSequenceConvOptorsCommute}, $\tilde{B}$ commutes with $A$.

         \noindent\ref{Proposition1ExampleTheoremIntOpRepGenKernLpConvSeqComOp:item4}. By following the same procedure of the item \ref{Proposition1ExampleTheoremIntOpRepGenKernLpConvSeqComOp:item2}, the operator $AB-BA$ has the following estimation. For all $x\in L_p(\mathbb{R},\mu)$ and for all parameters $$
         (\theta_{A,1},\theta_{A,2},\theta_{B,2},  \alpha_1,\beta_1,\alpha,\beta,\sigma_1,\sigma_2,\delta,\omega)\in\Lambda
         $$
         we have
         \begin{equation*}
           \|(AB-BA)x\|_{L_p}\leq \left|\frac{\theta_{B,1}(\delta \theta_{A,2}\sigma_2-1)}{\delta^2 \theta_{A,1}\sigma_1}\right|\cdot \tilde{\lambda} \cdot \|x\|_{L_p}\to 0,
         \end{equation*}
           when   $\frac{\theta_{B,1}(\delta \theta_{A,2}\sigma_2-1)}{\delta^2 \theta_{A,1}\sigma_1}\tilde{\lambda}\to 0$, where $\tilde{\lambda}=|\beta-\alpha|^{\frac{1}{p}}\cdot 2|\beta_1-\alpha_1|^{\frac{1}{q}}$, if $1< p < \infty$, $\tilde{\lambda}=2|\beta_1-\alpha_1|$, if $p=\infty$ and $\tilde{\lambda}=2|\beta-\alpha|$, if $p=1$. 
\QEDB
\end{proof}

\begin{theorem}
 Let $A:L_p(\mathbb{R},\mu)\to L_p(\mathbb{R},\mu)$, $B:L_p(\mathbb{R},\mu)\to L_p(\mathbb{R},\mu)$, $1\le p\le \infty$ be operators defined as
  \begin{eqnarray*}
     (Ax)(t)&=&\int\limits_{\alpha_1}^{\beta_1}I_{[\alpha,\beta]}(t)\big(\theta_{A,1}\sin(\omega t)\cos(\omega s)- \frac{1}{\delta\sigma_2}\cos(\omega t)\cos(\omega s)  \\ 
         & &
          +\frac{1}{\delta\sigma_1}\sin(\omega t)\sin(\omega s)\big)x(s)d\mu_s,\\
      (B x)(t)&=&\int\limits_{\alpha_1}^{\beta_1}I_{[\alpha,\beta]}(t)\left(\theta_{B,1}\sin(\omega t)\cos(\omega s)
         + \theta_{B,3}\sin(\omega t)\sin(\omega s)\right)x(s)d\mu_s
  \end{eqnarray*}
 for almost every $t$, where $\theta_{A,1}, \theta_{B,1}, \theta_{B,3}, \omega \in\mathbb{R}$,  $\delta\in\mathbb{R}\setminus\{0\}$, $\alpha,\beta,\alpha_1,\beta_1 \in\mathbb{R}$, $\alpha_1<\beta_1$, $\alpha\leq \alpha_1$, $\beta\geq \beta_1$, $I_E(\cdot)$ is the indicator function of the set $E$ and, either the number $\frac{\omega }{\pi}(\beta_1-\alpha_1)\in \mathbb{Z}$ or $\frac{\omega }{\pi}(\beta_1+\alpha_1)\in \mathbb{Z}$, $\sigma_1,\sigma_2\in\mathbb{R}\setminus\{0\}$ are defined in \eqref{ConstantSigma1Case2Omega} and \eqref{ConstantSigma2Case2Omega}, respectively.
  Let  $\{B_n\}: L_p(\mathbb{R},\mu)\to L_p(\mathbb{R},\mu)$ $1\le p \le \infty$ be a sequence of operators
    \begin{eqnarray*}
     (B_n x)(t)&=&\int\limits_{\alpha_1}^{\beta_1}I_{[\alpha,\beta]}(t)\left(\varsigma_n\sin(\omega t)\cos(\omega s)+ \theta_{B,3}\sin(\omega t)\sin(\omega s)\right)x(s)d\mu_s,
     \end{eqnarray*}
    for almost every $t$,  $\{\varsigma_n\}$ are number sequences.
    Let
         \begin{align*}
         & \Lambda=\big\{(\theta_{A,1},\theta_{B,1}, \theta_{B,3}, \alpha_1,\beta_1,\alpha,\beta,\sigma_1,\sigma_2,\delta,\omega)\in\mathbb{R}^{11}:\, \delta\not=0,\alpha \leq \alpha_1,\ \beta\geq \beta_1,\\
         &\  \int\limits_{\alpha_1}^{\beta_1}\sin(\omega s)\cos(\omega s)d\mu_s=0,  \sigma_1=\int\limits_{\alpha_1}^{\beta_1}(\sin(\omega s))^2d\mu_s\not=0,   \sigma_2=\beta_1-\alpha_1-\sigma_1\not=0 \big\}.
         \end{align*}
    Then
    \begin{enumerate}[label=\textup{\arabic*.}, ref=\arabic*]
    \item\label{Proposition2ExampleTheoremIntOpRepGenKernLpConvSeqComOp:item1}  $AB=\delta BA^2$, $AB_n=\delta B_nA^2 $ for each positive integer $n$. Moreover, for every $x\in L_p(\mathbb{R},\mu)$, $1\le p\le \infty$ we have
       \begin{equation*}
         (AB-BA)x(t)=\left(\frac{2\theta_{B,1}}{\delta}-\theta_{A,1}\theta_{B,3}\sigma_1\right)\int\limits_{\alpha_1}^{\beta_1} I_{[\alpha,\beta]}(t)\sin(\omega t)\cos(\omega s) x(s)d\mu_s,
       \end{equation*}
       for almost every $t$.
    \item\label{Proposition2ExampleTheoremIntOpRepGenKernLpConvSeqComOp:item2} if $\varsigma_n\to \frac{\delta}{2}\theta_{A,1}\theta_{B,3}\sigma_1$ when $n\to \infty$, then $B_n\to \tilde{B}$ (converges in norm), where
        \begin{equation*}
             (\tilde{B} x)(t)=\int\limits_{\alpha_1}^{\beta_1}I_{[\alpha,\beta]}(t)\frac{\delta}{2}\theta_{A,1}\theta_{B,3}\sigma_1\sin(\omega t)\cos(\omega s)  x(s)d\mu_s,
            \end{equation*}
            for almost every $t$. Moreover $AB_n-B_nA\to 0$ (converges in norm) and $A\tilde{B}=\tilde{B}A$.
      \item\label{Proposition2ExampleTheoremIntOpRepGenKernLpConvSeqComOp:item3} if $(\frac{2\theta_{B,1}}{\delta}-\theta_{A,1}\theta_{B,3}\sigma_1)\tilde{\lambda}\to 0$ in $\mathbb{R}$ when
          $$
          (\theta_{A,1},\theta_{B,1}, \theta_{B,3}, \alpha_1,\beta_1,\alpha,\beta,\sigma_1,\sigma_2,\delta,\omega)\to \lambda_0 \in cl(\Lambda) \mbox{ \ \rm (the closure of } \Lambda),
          $$
          where
          $(\theta_{A,1},\theta_{B,1}, \theta_{B,3}, \alpha_1,\beta_1,\alpha,\beta,\sigma_1,\sigma_2,\delta,\omega)\in\Lambda$ and,  $\tilde{\lambda}=|\beta-\alpha|^{\frac{1}{p}}|\beta_1-\alpha_1|^{\frac{1}{q}}$ for $1< p <\infty$ with
          $\frac{1}{p}+\frac{1}{q}=1$,  $\tilde{\lambda}=|\beta_1-\alpha_1|$ for $p=\infty$ and $\tilde{\lambda}=|\beta-\alpha|$ for $p=1$, then for all $x\in L_p(\mathbb{R},\mu)$, $1\leq p\leq \infty$, it holds that $\|(AB-BA)x\|_{L_p}\to 0$. 
         \end{enumerate}
         \end{theorem}

\begin{proof}
\noindent\ref{Proposition2ExampleTheoremIntOpRepGenKernLpConvSeqComOp:item1}. It follows from direct computation.

\noindent\ref{Proposition2ExampleTheoremIntOpRepGenKernLpConvSeqComOp:item2}. We have
         \begin{eqnarray*}
           (B_n-\tilde B)x(t)&=&\int\limits_{\alpha_1}^{\beta_1}I_{[\alpha_1,\beta_1]}(t)\left(\varsigma_n-\frac{\delta}{2}\theta_{A,1}\theta_{B,3}\sigma_1\right) \sin(\omega t)\cos(\omega s)x(s)d\mu_s,
         \end{eqnarray*}
         almost everywhere. Moreover, by applying H\"older inequality we  have for all $x\in L_p(\mathbb{R},\mu)$, $1< p<\infty $, the following estimations:
         \begin{align*}
         & \|(B_n-\tilde B)x\|^p_{L_p}=  \int\limits_{\mathbb{R}} \big| \int\limits_{\alpha_1}^{\beta_1}I_{[\alpha,\beta]}(t)\left(\varsigma_n-\frac{\delta}{2}\theta_{A,1}\theta_{B,3}\sigma_1\right) \sin(\omega t)\cos(\omega s)
          x(s)d\mu_s\big|^pd\mu_t
          \\
          &
          \leq
         \big|\varsigma_n-\frac{\delta}{2}\theta_{A,1}\theta_{B,3}\sigma_1\big|^p |\beta-\alpha|\cdot \big|\int\limits_{\mathbb{R}} I_{[\alpha_1,\beta_1]}(s) \sin(\omega s)x(s)d\mu_s\big|^p \\
         &\, \le \big|\varsigma_n-\frac{\delta}{2}\theta_{A,1}\theta_{B,3}\sigma_1\big|^p |\beta-\alpha|\cdot|\beta_1-\alpha_1|^{\frac{p}{q}}\cdot\|x\|^p_{L_p}.
         \end{align*}
         Therefore, for $1< p<\infty $, we get the following
         \begin{equation*}
          \|B_n-\tilde B\|_{L_p}\le  \big|\varsigma_n-\frac{\delta}{2}\theta_{A,1}\theta_{B,3}\sigma_1\big|\cdot|\beta-\alpha|^{\frac{1}{p}}\cdot|\beta_1-\alpha_1|^{\frac{1}{q}}\to 0,
         \end{equation*}
         when $n\to \infty$, since $\theta_n\to 0$. Applying H\"older inequality yields, for every $x\in L_\infty(\mathbb{R},\mu)$, the following estimate
         \begin{align*}
         & \|(B_n-\tilde B) x\|_{L_\infty}= \mathop{\esssup}_{t\in \mathbb{R}} \big|\ \int\limits_{\alpha_1}^{\beta_1}I_{[\alpha,\beta]}(t)\left(\varsigma_n-\frac{\delta}{2}\theta_{A,1}\theta_{B,3}\sigma_1\right) \sin(\omega t)\cos(\omega s) x(s)d\mu_s\big|
         \\
         &
         \leq\big|\varsigma_n-\frac{\delta}{2}\theta_{A,1}\theta_{B,3}\sigma_1\big|\int\limits_{\mathbb{R}} \left|I_{[\alpha_1,\beta_1]}(s) \sin(\omega s)\right|d\mu_s\cdot \|x\|_{L_\infty}  \\
         &  \leq\big|\varsigma_n-\frac{\delta}{2}\theta_{A,1}\theta_{B,3}\sigma_1\big|\cdot|\beta_1-\alpha_1|\cdot\|x\|_{L_\infty}.
         \end{align*}
         Therefore, for $ p=\infty $, we get
          \begin{equation*}
          \|B_n-\tilde B\|_{L_\infty}\leq  \big|\varsigma_n-\frac{\delta}{2}\theta_{A,1}\theta_{B,3}\sigma_1\big|\cdot|\beta_1-\alpha_1|\to 0.
         \end{equation*}
         For every $x\in L_1(\mathbb{R},\mu)$,
         \begin{align*}
         & \|(B_n-\tilde B) x\|_{L_1}= \int\limits_{\mathbb{R}} \big|\ \int\limits_{\alpha_1}^{\beta_1}I_{[\alpha,\beta]}(t)\big(\varsigma_n-\frac{\delta}{2}\theta_{A,1}\theta_{B,3}\sigma_1\big) \sin(\omega t)\cos(\omega s) x(s)d\mu_s\big|d\mu_t\\
         & \hspace{2cm}
         \leq\big|\varsigma_n-\frac{\delta}{2}\theta_{A,1}\theta_{B,3}\sigma_1 \big||\beta-\alpha|\int\limits_{\mathbb{R}} |I_{[\alpha_1,\beta_1]}(s) \sin(\omega s)||x(s)|d\mu_s \\
         & \hspace{2cm} \leq\big|\varsigma_n-\frac{\delta}{2}\theta_{A,1}\theta_{B,3}\sigma_1\big|\cdot|\beta-\alpha|\cdot\|x\|_{L_1}.
         \end{align*}
          Therefore, for $ p=1 $, we get
          \begin{equation*}
          \|B_n-\tilde B\|_{L_\infty}\leq  \big|\varsigma_n-\frac{\delta}{2}\theta_{A,1}\theta_{B,3}\sigma_1\big|\cdot|\beta_1-\alpha_1|\to 0,
         \end{equation*}
         when $n\to\infty$ since $\varsigma_n\to \frac{\delta}{2}\theta_{A,1}\theta_{B,3}\sigma_1$. Thus, $B_n\to\tilde B$ converges in norm in $L_p$,  $1\le p\le \infty$.

         Now we prove the convergence for the sequence $\{AB_n-B_nA\}$. Applying H\"older inequality yields, for all $x\in L_p(\mathbb{R},\mu)$, $1< p<\infty $,
         \begin{align*}
          & \|(AB_n-B_nA)x\|^p_{L_p}= \int\limits_{\mathbb{R}}\big|\left(\frac{2\varsigma_n}{\delta}-\theta_{A,1}\theta_{B,3}\sigma_1\right)
          \\
         & \hspace{4cm} \cdot\int\limits_{\alpha_1}^{\beta_1} I_{[\alpha,\beta]}(t)\sin(\omega t)\cos(\omega s) x(s)d\mu_s\big|^p d\mu_t
          \\
          & \leq  \big|\frac{2\varsigma_n}{\delta}-\theta_{A,1}\theta_{B,3}\sigma_1\big|^p
           |\beta-\alpha|
           \big|\int\limits_{\mathbb{R}} I_{[\alpha_1,\beta_1]}(s)\cos(\omega s)x(s)d\mu_s\big|^p \\
           &  \leq
           \big|\frac{2\varsigma_n}{\delta}-\theta_{A,1}\theta_{B,3}\sigma_1\big|^p |\beta-\alpha| \cdot
            |\beta_1-\alpha_1|^{\frac{p}{q}}\cdot \|x\|^p.
         \end{align*}
         Therefore, for $1< p <\infty$ we have
         \begin{equation*}
         \|AB_n-B_nA\|_{L_p} \le  \left|\frac{2\varsigma_n}{\delta}-\theta_{A,1}\theta_{B,3}\sigma_1\right|\cdot
         |\beta-\alpha|^{\frac{1}{p}}\cdot |\beta_1-\alpha_1|^{\frac{1}{q}}\to 0,\, n\to\infty,
         \end{equation*}
         since $\varsigma_n\to \frac{\delta}{2}\theta_{A,1}\theta_{B,3}\sigma_1$.
         Applying H\"older inequality yields, for all $x\in L_\infty(\mathbb{R},\mu)$,
            \begin{align*}
           & \|(AB_n-B_nA)x\|_{L_\infty}=\mathop{\esssup}_{t\in \mathbb{R}} \big|\left(\frac{2\varsigma_n}{\delta}-\theta_{A,1}\theta_{B,3}\sigma_1\right) I_{[\alpha,\beta]}(t)\sin(\omega t) \\
          & \hspace{3cm}  \cdot
          \int\limits_{\alpha_1}^{\beta_1}\cos(\omega s)x(s)d\mu_s\big|
          \\
          &
          \leq  \left|\frac{2\varsigma_n}{\delta}-\theta_{A,1}\theta_{B,3}\sigma_1\right| \int\limits_{\mathbb{R}}| I_{[\alpha_1,\beta_1]}(s) \cos(\omega s)|d\mu_s \|x\|_{L_\infty}\\
          & \leq  \left|\frac{2\varsigma_n}{\delta}-\theta_{A,1}\theta_{B,3}\sigma_1\right|\cdot |\beta_1-\alpha_1|\|x\|_{L_\infty}.
         \end{align*}
         Therefore, for $p=\infty$ we have
           \begin{equation*}
           \|AB_n-B_nA\|_{L_\infty}\le \left|\frac{2\varsigma_n}{\delta}-\theta_{A,1}\theta_{B,3}\sigma_1\right|\cdot |\beta_1-\alpha_1|\to 0,\, n\to\infty,
         \end{equation*}
         because $\varsigma_n\to \frac{\delta}{2}\theta_{A,1}\theta_{B,3}\sigma_1$.
          For all $x\in L_1(\mathbb{R},\mu)$,
            \begin{align*}
           & \|(AB_n-B_nA)x\|_{L_1}=\int\limits_{\mathbb{R}} \big|\left(\frac{2\varsigma_n}{\delta}-\theta_{A,1}\theta_{B,3}\sigma_1\right) I_{[\alpha,\beta]}(t)\sin(\omega t) \\
          & \hspace{3cm}  \cdot
          \int\limits_{\alpha_1}^{\beta_1}\cos(\omega s)x(s)d\mu_s\big|d\mu_t
          \\
          &
          \leq  \big|\frac{2\varsigma_n}{\delta}-\theta_{A,1}\theta_{B,3}\sigma_1\big|\cdot|\beta-\alpha| \int\limits_{\mathbb{R}}| I_{[\alpha_1,\beta_1]}(s) \cos(\omega s)x(s)|d\mu_s \\
          & \leq  \big|\frac{2\varsigma_n}{\delta}-\theta_{A,1}\theta_{B,3}\sigma_1\big|\cdot |\beta-\alpha|\cdot\|x\|_{L_1}.
         \end{align*}
         Therefore, for $p=1$ we have
           \begin{equation*}
           \|AB_n-B_nA\|_{L_1}\le \big|\frac{2\varsigma_n}{\delta}-\theta_{A,1}\theta_{B,3}\sigma_1\big|\cdot |\beta-\alpha|\to 0,\, n\to\infty,
         \end{equation*}
         because $\varsigma_n\to \frac{\delta}{2}\theta_{A,1}\theta_{B,3}\sigma_1$.
          Since the sequence of operators  $\{B_n\}$  converges in norm to $\tilde{B}$, then by Lemma \ref{LemmaOnseSequenceConvOptorsCommute}, $\tilde{B}$ commutes with $A$.

          \noindent\ref{Proposition2ExampleTheoremIntOpRepGenKernLpConvSeqComOp:item3}. By following the same procedure of the item \ref{Proposition2ExampleTheoremIntOpRepGenKernLpConvSeqComOp:item2}, the operator $AB-BA$ has the following estimation. For all $x\in L_p(\mathbb{R},\mu)$ and for all parameters 
          $$(\theta_{A,1},\theta_{B,1}, \theta_{B,3}, \alpha_1,\beta_1,\alpha,\beta,\sigma_1,\sigma_2,\delta,\omega)\in\Lambda
          $$
          we have
         \begin{equation*}
           \|(AB-BA)x\|_{L_p}\leq \big|\frac{2\theta_{B,1}}{\delta}-\theta_{A,1}\theta_{B,3}\sigma_1\big|\cdot \tilde{\lambda} \cdot \|x\|_{L_p}\to 0,
         \end{equation*}
           when   $(\frac{2\theta_{B,1}}{\delta}-\theta_{A,1}\theta_{B,3}\sigma_1)\tilde{\lambda}\to 0$, where $\tilde{\lambda}=|\beta-\alpha|^{\frac{1}{p}}\cdot|\beta_1-\alpha_1|^{\frac{1}{q}}$, if $1< p < \infty$, $\tilde{\lambda}=|\beta_1-\alpha_1|$, if $p=\infty$ and  $\tilde{\lambda}=|\beta-\alpha|$, if $p=1$.
  \QEDB
\end{proof}

 \begin{theorem}
 Let $A:L_p(\mathbb{R},\mu)\to L_p(\mathbb{R},\mu)$, $B:L_p(\mathbb{R},\mu)\to L_p(\mathbb{R},\mu)$, $1\le p\le \infty$ be operators defined as
  \begin{eqnarray*}
     (Ax)(t)&=&\int\limits_{\alpha_1}^{\beta_1}I_{[\alpha,\beta]}(t)\big(\theta_{A,1}\sin(\omega t)\cos(\omega s)+ \frac{1}{\delta\sigma_2}\cos(\omega t)\cos(\omega s)  \\ 
         & &
      \hspace{3cm}    +\frac{1}{\delta\sigma_1}\sin(\omega t)\sin(\omega s)\big)x(s)d\mu_s,\\
      (B x)(t)&=&\int\limits_{\alpha_1}^{\beta_1}I_{[\alpha,\beta]}(t)\big(\theta_{B,1}\sin(\omega t)\cos(\omega s)
         +\frac{2\sigma_1}{\sigma_2}\theta_{B,3}\cos(\omega t)\cos(\omega s) \\
      & & \hspace{3cm}  
         +\theta_{B,3}\sin(\omega t)\sin(\omega s)\big)x(s)d\mu_s
  \end{eqnarray*}
 for almost every $t$, where $\theta_{A,1}, \theta_{B,1}, \theta_{B,3},\,\omega \in\mathbb{R}$,  $\delta\in\mathbb{R}\setminus\{0\}$, $\alpha,\beta,\alpha_1,\beta_1 \in\mathbb{R}$, $\alpha_1<\beta_1$, $\alpha\leq \alpha_1$, $\beta\geq \beta_1$, $I_E(\cdot)$ is the indicator function of the set $E$ and, either the number $\frac{\omega }{\pi}(\beta_1-\alpha_1)\in \mathbb{Z}$ or $\frac{\omega }{\pi}(\beta_1+\alpha_1)\in \mathbb{Z}$,
  $\sigma_1,\sigma_2\in\mathbb{R}\setminus\{0\}$ are defined in \eqref{ConstantSigma1Case2Omega} and \eqref{ConstantSigma2Case2Omega}, respectively.
  Let $\{A_n\}: L_p(\mathbb{R},\mu)\to L_p(\mathbb{R},\mu)$, $\{B_n\}: L_p(\mathbb{R},\mu)\to L_p(\mathbb{R},\mu)$, with $1\le p \le \infty$ be sequences of operators
    \begin{eqnarray*}
     (A_n x)(t)&=&\int\limits_{\alpha_1}^{\beta_1}I_{[\alpha,\beta]}(t)\big(\theta_{n}\sin(\omega t)\cos(\omega s) \frac{1}{\delta\sigma_2}\cos(\omega t)\cos(\omega s) \\ 
         & &
         \hspace{3cm}  +\frac{1}{\delta\sigma_1}\sin(\omega t)\sin(\omega s)\big)x(s)d\mu_s,\\
      (B_n x)(t)&=&\int\limits_{\alpha_1}^{\beta_1}I_{[\alpha,\beta]}(t)\big(\theta_{B,1}\sin(\omega t)\cos(\omega s)
         +\frac{2\sigma_1}{\sigma_2}\varsigma_n\cos(\omega t)\cos(\omega s) \\
         & & \hspace{3cm}  
         +\varsigma_n\sin(\omega t)\sin(\omega s)\big)x(s)d\mu_s
     \end{eqnarray*}
    for almost every $t$, $\{\theta_n\}$ and $\{\varsigma_n\}$ are number sequences. Let
         \begin{eqnarray*}
          \Lambda&=&\big\{(\theta_{A,1},\theta_{B,1},\theta_{B,2}, \theta_{B,3}, \alpha_1,\beta_1,\alpha,\beta,\sigma_1,\sigma_2,\delta,\omega)\in\mathbb{R}^{12}:\, \delta\not=0, \alpha \leq \alpha_1,\beta\geq \beta_1,\\
         &&\  \int\limits_{\alpha_1}^{\beta_1}\sin(\omega s)\cos(\omega s)d\mu_s=0,  \sigma_1=\int\limits_{\alpha_1}^{\beta_1}(\sin(\omega s))^2d\mu_s\not=0,   \sigma_2=\beta_1-\alpha_1-\sigma_1\not=0 \big\}.
         \end{eqnarray*}
         Then
    \begin{enumerate}[label=\textup{\arabic*.}, ref=\arabic*]
    \item\label{Proposition3ExampleTheoremIntOpRepGenKernLpConvSeqComOp:item1}  $AB=\delta BA^2$, $A_nB=\delta BA_n^2$, $A_nB_n=\delta B_nA_n^2 $, $AB_n=\delta B_nA^2 $ for each positive integer $n$. Moreover, for all $x\in L_p(\mathbb{R},\mu)$, $1\le p\le \infty$ we have
       \begin{equation*}
         (AB-BA)x(t)=\theta_{A,1}\sigma_1 \theta_{B,3}\int\limits_{\alpha_1}^{\beta_1} I_{[\alpha,\beta]}(t)\sin(\omega t)\cos(\omega s)x(s)d\mu_s,
       \end{equation*}
       for almost every $t$;
           \item\label{Proposition3ExampleTheoremIntOpRepGenKernLpConvSeqComOp:item2} if $\theta_n\to 0$ when $n\to \infty$, then $A_n\to \tilde{A}$ (converges in norm)
               defined as
            \begin{equation*}
             (\tilde{A} x)(t)=\int\limits_{\alpha_1}^{\beta_1}I_{[\alpha,\beta]}(t)\left(\frac{1}{\delta\sigma_2}\cos(\omega t)\cos(\omega s) +\frac{1}{\delta\sigma_1}sin(\omega t)\sin(\omega s)\right) x(s)d\mu_s,
            \end{equation*}
            for almost every $t$, so $A_nB-BA_n\to 0$ (converges in norm) and $\tilde{A}B=B\tilde{A}$;
          \item\label{Proposition3ExampleTheoremIntOpRepGenKernLpConvSeqComOp:item3} if $\varsigma_n\to 0$ when $n\to \infty$, then $B_n\to \tilde{B}$ (converges in norm) defined as
\begin{equation*}
(\tilde{B} x)(t)=\int\limits_{\alpha_1}^{\beta_1}I_{[\alpha,\beta]}(t)\theta_{B,1}\sin(\omega t)\cos(\omega s)  x(s)d\mu_s,
            \end{equation*}
            for almost every $t$. Moreover $AB_n-B_nA\to 0$ (converge in norm) and $A\tilde{B}=\tilde{B}A$;
           \item\label{Proposition3ExampleTheoremIntOpRepGenKernLpConvSeqComOp:item4} if  $\theta_n\to 0$ and $\varsigma_n\to 0$ when $n\to \infty$, then $A_nB_n -B_nA_n \to 0$ (converge in norm) and $\tilde{A}$ commutes with $\tilde{B}$.
           \item\label{Proposition3ExampleTheoremIntOpRepGenKernLpConvSeqComOp:item5} if $\theta_{A,1}\sigma_1 \theta_{B,3}\tilde{\lambda}\to 0$ in $\mathbb{R}$ when $(\theta_{A,1},\theta_{B,1},\theta_{B,2}, \theta_{B,3}, \alpha_1,\beta_1,\alpha,\beta,\sigma_1,\sigma_2,\delta,\omega)\to \lambda_0\in cl(\Lambda)$ (the closure of $\Lambda$), where $(\theta_{A,1},\theta_{B,1}, \theta_{B,3}, \alpha_1,\beta_1,\alpha,\beta,\sigma_1,\sigma_2,\delta,\omega)\in\Lambda$ and, $\tilde{\lambda}=|\beta-\alpha|^{\frac{1}{p}}|\beta_1-\alpha_1|^{\frac{1}{q}}$ for $1< p <\infty$ with $\frac{1}{p}+\frac{1}{q}=1$,  $\tilde{\lambda}=|\beta_1-\alpha_1|$ for $p=\infty$
               and $\tilde{\lambda}=|\beta-\alpha|$ for $p=1$, then for all $x\in L_p(\mathbb{R},\mu)$, $1\leq p\leq \infty$, it holds that $\|(AB-BA)x\|_{L_p}\to 0$. 
         \end{enumerate}
         \end{theorem}

 \begin{proof}
\noindent\ref{Proposition3ExampleTheoremIntOpRepGenKernLpConvSeqComOp:item1}. It follows from direct computation.

         \noindent\ref{Proposition3ExampleTheoremIntOpRepGenKernLpConvSeqComOp:item2}.  We have
         \begin{eqnarray*}
           (A_n-\tilde A)x(t)&=&\int\limits_{\alpha_1}^{\beta_1}I_{[\alpha_1,\beta_1]}(t)\theta_n \sin(\omega t)\cos(\omega s)x(s)d\mu_s,
         \end{eqnarray*}
         for almost every $t$. Moreover, applying H\"older inequality yields, for every $x\in L_p(\mathbb{R},\mu)$, $1< p<\infty $,
         \begin{align*}
         & \|(A_n-\tilde A)x\|^p_{L_p}=  \int\limits_{\mathbb{R}} \big| \int\limits_{\alpha_1}^{\beta_1}I_{[\alpha,\beta]}(t)\theta_n \sin(\omega t)\cos(\omega s)x(s)d\mu_s\big|^pd\mu_t\\
         & \leq |\theta_n|^p |\beta-\alpha| \big|\int\limits_{\mathbb{R}} I_{[\alpha_1,\beta_1]}(s) \sin(\omega s)x(s)d\mu_s\big|^p \leq |\theta_n|^p |\beta-\alpha||\beta_1-\alpha_1|^{\frac{p}{q}}\|x\|^p_{L_p}.
         \end{align*}
         Therefore, for $1< p<\infty $, we have
         \begin{equation*}
          \|(A_n-\tilde A)\|_{L_p}\le  |\theta_n|\cdot|\beta-\alpha|^{\frac{1}{p}}\cdot|\beta_1-\alpha_1|^{\frac{1}{q}}\to 0,
         \end{equation*}
         when $n\to \infty$, since $\theta_n\to 0$. Applying H\"older inequality yields, for every $x\in L_\infty(\mathbb{R},\mu)$,
         \begin{eqnarray*}
         && \|(A_n-\tilde A) x\|_{L_\infty}= \mathop{\esssup}_{t\in \mathbb{R}} \big|\ \int\limits_{\alpha_1}^{\beta_1}I_{[\alpha,\beta]}(t)\theta_n \sin(\omega t)\cos(\omega s)x(s)d\mu_s\big|
         \\
         &&
         \leq
         |\theta_n|\cdot \int\limits_{\mathbb{R}} \left|I_{[\alpha_1,\beta_1]}(s) \sin(\omega s)\right|d\mu_s\cdot \|x\|_{L_\infty} \le |\theta_n|\cdot|\beta_1-\alpha_1|\cdot\|x\|_{L_\infty}.
         \end{eqnarray*}
          Therefore, for $ p=\infty $,
         \begin{equation*}
          \|A_n-\tilde A\|_{L_\infty}\le  |\theta_n|\cdot|\beta_1-\alpha_1|\to 0,
         \end{equation*}
         when $n\to\infty$ since $\theta_n\to 0$.
         For every $x\in L_1(\mathbb{R},\mu)$,
         \begin{eqnarray*}
         && \|(A_n-\tilde A) x\|_{L_1}= \int\limits_{\mathbb{R}} \big|\ \int\limits_{\alpha_1}^{\beta_1}I_{[\alpha,\beta]}(t)\theta_n \sin(\omega t)\cos(\omega s)x(s)d\mu_s\big|d\mu_t
         \\
         &&
         \leq
         |\theta_n|\cdot |\beta-\alpha|\int\limits_{\mathbb{R}} |I_{[\alpha_1,\beta_1]}(s) \sin(\omega s)x(s)|d\mu_s \leq |\theta_n|\cdot|\beta-\alpha|\cdot\|x\|_{L_1}.
         \end{eqnarray*}
          Therefore, for $ p=1 $,
         \begin{equation*}
          \|A_n-\tilde A\|_{L_1}\le  |\theta_n|\cdot|\beta_1-\alpha_1|\to 0,
         \end{equation*}
         when $n\to\infty$ since $\theta_n\to 0$.
         Therefore, $A_n\to\tilde A$ converge in norm in $L_p$, $1\le p\le \infty$.

         Now we prove the convergence for the sequence $\{A_nB-BA_n\}$. Applying H\"older inequality yields, for all $x\in L_p(\mathbb{R},\mu)$, $1< p<\infty $,
         \begin{align*}
          & \resizebox{0.85\hsize}{!}{$\|(A_nB-BA_n)x\|^p_{L_p}= \big|\theta_{n}\sigma_1 \theta_{B,3}\int\limits_{\alpha_1}^{\beta_1} I_{[\alpha,\beta]}(t)\sin(\omega t)\cos(\omega s)x(s)d\mu_s\big|^p d\mu_t$}
            \\
          & \resizebox{1\hsize}{!}{$\displaystyle \leq  |\theta_{n}\sigma_1 \theta_{B,3}|^p
           |\beta-\alpha|
           \big|\, \int\limits_{\mathbb{R}} I_{[\alpha_1,\beta_1]}(s)\cos(\omega s)x(s)d\mu_s\big|^p
             \leq
           |\theta_{n}\sigma_1 \theta_{B,3}|^p |\beta-\alpha|
            |\beta_1-\alpha_1|^{\frac{p}{q}}\|x\|^p. $}
         \end{align*}
         Therefore, for $1< p <\infty$ we have
         \begin{equation*}
         \|A_nB-BA_n\|_{L_p} \le  \left|\theta_{n}\sigma_1 \theta_{B,3}\right|\cdot
         |\beta-\alpha|^{\frac{1}{p}}\cdot |\beta_1-\alpha_1|^{\frac{1}{q}}\to 0,\, n\to\infty,
         \end{equation*}
         since $\theta_n\to 0$.
         Applying H\"older inequality yields, for  all $x\in L_\infty(\mathbb{R},\mu)$,
            \begin{eqnarray*}
           &&\|(A_nB-BA_n)x\|_{L_\infty}=\mathop{\esssup}_{t\in \mathbb{R}} \big|\theta_{n}\sigma_1 \theta_{B,3} I_{[\alpha,\beta]}(t)\sin(\omega t)\cdot \int\limits_{\alpha_1}^{\beta_1}\cos(\omega s)x(s)d\mu_s\big|\\
          &&\leq  |\theta_{n}\sigma_1 \theta_{B,3}| \int\limits_{\mathbb{R}}| I_{[\alpha_1,\beta_1]}(s) \cos(\omega s)|d\mu_s \|x\|_{L_\infty}
          \leq  |\theta_{n}\sigma_1 \theta_{B,3}|\cdot |\beta_1-\alpha_1|\|x\|_{L_\infty}.
         \end{eqnarray*}
         Therefore, for $p=\infty$, we have
           \begin{equation*}
           \|A_nB-BA_n\|_{L_\infty}\le \left|\theta_{n}\sigma_1 \theta_{B,3}\right|\cdot |\beta_1-\alpha_1|\to 0,\, n\to\infty,
         \end{equation*}
         because $\theta_n\to 0$.
          For  all $x\in L_1(\mathbb{R},\mu)$,
            \begin{align*}
           &\|(A_nB-BA_n)x\|_{L_1}=\int\limits_{\mathbb{R}} \big|\theta_{n}\sigma_1 \theta_{B,3} I_{[\alpha,\beta]}(t)\sin(\omega t)\cdot \int\limits_{\alpha_1}^{\beta_1}\cos(\omega s)x(s)d\mu_s\big|d\mu_t\\
          &\leq  |\theta_{n}\sigma_1 \theta_{B,3} |\cdot|\beta-\alpha| \int\limits_{\mathbb{R}}| I_{[\alpha_1,\beta_1]}(s) \cos(\omega s)x(s)|d\mu_s
          \leq  \left|\theta_{n}\sigma_1 \theta_{B,3}\right|\cdot |\beta-\alpha|\|x\|_{L_1}.
         \end{align*}
         Therefore, for $p=1$, we have
           \begin{equation*}
           \|A_nB-BA_n\|_{L_1}\le \left|\theta_{n}\sigma_1 \theta_{B,3}\right|\cdot |\beta-\alpha|\to 0,\, n\to\infty,
         \end{equation*}
         because $\theta_n\to 0$.
          Since the sequence of operators  $\{A_n\}$  converges in norm to $\tilde{A}$, then by Lemma \ref{LemmaOnseSequenceConvOptorsCommute}, $\tilde{A}$ commutes with $B$.

         \noindent\ref{Proposition3ExampleTheoremIntOpRepGenKernLpConvSeqComOp:item3}.
         We have
         \begin{eqnarray*}
          && (B_n-\tilde B)x(t)=\int\limits_{\alpha_1}^{\beta_1}I_{[\alpha_1,\beta_1]}(t) \frac{2\sigma_1}{\sigma_2}\varsigma_n \cos(\omega t)\cos(\omega s)x(s)d\mu_s\\
          && + \int\limits_{\alpha_1}^{\beta_1}I_{[\alpha_1,\beta_1]}(t)\varsigma_{n}\sin(\omega t)\sin(\omega s)x(s)d\mu_s
          =(B_n-\tilde B)_1x(t)+(B_n-\tilde B)_2x(t),
         \end{eqnarray*}
         respectively, for almost every $t$. Moreover, by applying H\"older inequality we  have for all $x\in L_p(\mathbb{R},\mu)$, $1< p<\infty $ the following estimate
         \begin{eqnarray*}
         &&\resizebox{0.85\hsize}{!}{$\displaystyle \|(B_n-\tilde B)_1x\|^p_{L_p}=  \int\limits_{\mathbb{R}} \big| \, \int\limits_{\alpha_1}^{\beta_1}I_{[\alpha,\beta]}(t)\varsigma_n \frac{2\sigma_1}{\sigma_2} \cos(\omega t)\cos(\omega s)x(s)d\mu_s\big|^pd\mu_t$}\\
         && \resizebox{1\hsize}{!}{$\displaystyle \leq \left|\varsigma_n \frac{2\sigma_1}{\sigma_2}\right|^p |\beta-\alpha| \big|\int\limits_{\mathbb{R}} I_{[\alpha,\beta]}(s) \cos(\omega s)x(s)d\mu_s\big|^p \leq\left|\varsigma_n \frac{2\sigma_1}{\sigma_2}\right|^p |\beta-\alpha||\beta_1-\alpha_1|^{\frac{p}{q}}\|x\|^p_{L_p}.$}
         \end{eqnarray*}
         Therefore, for $1< p<\infty $, we get the following norm estimate
         \begin{equation*}
          \|(B_n-\tilde B)_1\|_{L_p}\le  \left|\varsigma_n \frac{2\sigma_1}{\sigma_2}\right|\cdot|\beta-\alpha|^{\frac{1}{p}}\cdot|\beta_1-\alpha_1|^{\frac{1}{q}}.
         \end{equation*}
         Applying H\"older inequality yields, for all $x\in L_\infty(\mathbb{R},\mu)$, \begin{eqnarray*}
         && \hspace{-0.1cm} \|(B_n-\tilde B)_1x\|_{L_\infty}= \mathop{\esssup}_{t\in \mathbb{R}} \big| \int\limits_{\alpha_1}^{\beta_1}I_{[\alpha,\beta]}(t)\varsigma_n \frac{2\sigma_1}{\sigma_2} \cos(\omega t)\cos(\omega s)x(s)d\mu_s\big|\\
         && \hspace{-0.1cm} \leq \left|\varsigma_n \frac{2\sigma_1}{\sigma_2}\right|\cdot \int\limits_{\mathbb{R}} |I_{[\alpha_1,\beta_1]}(s) \cos(\omega s)|d\mu_s\cdot \|x\|_{L_\infty} \le \left|\varsigma_n \frac{2\sigma_1}{\sigma_2}\right|\cdot|\beta_1-\alpha_1|\cdot\|x\|_{L_\infty}.
         \end{eqnarray*}
          Therefore, for $p=\infty $, we have the following norm estimate
         \begin{equation*}
          \|(B_n-\tilde B)_1\|_{L_\infty}\le  \left|\varsigma_n \frac{2\sigma_1}{\sigma_2}\right|\cdot|\beta_1-\alpha_1|.
         \end{equation*}
         For all $x\in L_1(\mathbb{R},\mu)$,
         \begin{eqnarray*}
         && \hspace{-0.1cm} \|(B_n-\tilde B)_1x\|_{L_1}= \int\limits_{\mathbb{R}} \big| \int\limits_{\alpha_1}^{\beta_1}I_{[\alpha,\beta]}(t)\varsigma_n \frac{2\sigma_1}{\sigma_2} \cos(\omega t)\cos(\omega s)x(s)d\mu_s\big|d\mu_t\\
         && \hspace{-0.1cm} \leq \big|\varsigma_n \frac{2\sigma_1}{\sigma_2}\big|\cdot \int\limits_{\mathbb{R}} |I_{[\alpha_1,\beta_1]}(s) \cos(\omega s)x(s)|d\mu_s \le \big|\varsigma_n \frac{2\sigma_1}{\sigma_2}\big|\cdot|\beta-\alpha|\cdot\|x\|_{L_1}.
         \end{eqnarray*}
          Therefore, for $ p=1 $, we have the following norm estimate
         \begin{equation*}
          \|(B_n-\tilde B)_1\|_{L_1}\leq  \big|\varsigma_n \frac{2\sigma_1}{\sigma_2}\big|\cdot|\beta-\alpha|.
         \end{equation*}
         Analogously, we  get, for  $1< p<\infty $, the estimate
         \begin{equation*}
          \|(B_n-\tilde B)_2\|_{L_p}\le  \left|\varsigma_{n}\right| \cdot|\beta-\alpha|^{\frac{1}{p}}\cdot|\beta_1-\alpha_1|^{\frac{1}{q}}.
         \end{equation*}
         For $p=\infty$,
         \begin{equation*}
          \|(B_n-\tilde B)_2\|_{L_\infty}\leq  \left|\varsigma_{n}\right| \cdot|\beta_1-\alpha_1|.
         \end{equation*}
          For $p=1$,
         \begin{equation*}
          \|(B_n-\tilde B)_2\|_{L_1}\leq  \left|\varsigma_{n}\right| \cdot|\beta-\alpha|.
         \end{equation*}
         Therefore, by Minkowski inequality \cite{AdamsG,FollandRA,KolmogorovVol2} we have, for $1< p <\infty$,
         \begin{eqnarray*}
             \|B_n-\tilde B\|_{L_p}&\leq & \|(B_n-\tilde B)_1\|_{L_p}+\|(B_n-\tilde B)_2\|_{L_p} \\
             &\leq & |\varsigma_n |\cdot |\beta-\alpha|^{\frac{1}{p}}\cdot|\beta_1-\alpha_1|^{\frac{1}{q}}  \left(\frac{2|\sigma_1|}{|\sigma_2|}+
              1 \right)\to 0,
         \end{eqnarray*}
when $n\to \infty$, since $\varsigma_n\to 0$. On the other hand, by triangle inequality, for $p=\infty$ and for
 $p=1$,
         \begin{align*}
          &   \|B_n-\tilde B\|_{L_\infty}\leq  \|(B_n-\tilde B)_1\|_{L_\infty}+\|(B_n-\tilde B)_2\|_{L_\infty} \leq |\varsigma_n| \cdot|\beta_1-\alpha_1|\cdot  \left(\frac{2|\sigma_1|}{|\sigma_2|}+
              1  \right)\to 0,
          \\
           &   \|B_n-\tilde B\|_{L_1}\leq  \|(B_n-\tilde B)_1\|_{L_1}+\|(B_n-\tilde B)_2\|_{L_1} \leq |\varsigma_n| \cdot|\beta-\alpha|\cdot  \left(\frac{2|\sigma_1|}{|\sigma_2|}+
              1  \right)\to 0,
         \end{align*}
         when $n\to \infty$ since $\varsigma_n\to 0$. Therefore, $B_n$ converges in norm to $\tilde{B}$ in $L_p$, $1\leq p\leq \infty$.

         In order to prove that $\{AB_n-B_nA\}$ converges in norm to 0. We apply the same procedure as on item \ref{Proposition3ExampleTheoremIntOpRepGenKernLpConvSeqComOp:item2} for $1< p<\infty $, so we get for $1< p<\infty$
         \begin{equation*}
           \|AB_n-B_nA\|_{L_p}\le \left|\theta_{A,1}\sigma_1 \varsigma_n\right|\cdot |\beta-\alpha|^{\frac{1}{p}}\cdot|\beta_1-\alpha_1|^{\frac{1}{q}}\to 0,\, n\to\infty,
         \end{equation*}
         since $\varsigma_n\to 0$. For $p=\infty$ and $p=1$ we get
           \begin{align*}
          & \|AB_n-B_nA\|_{L_\infty}\le \left|\theta_{A,1}\sigma_1 \varsigma_n\right|\cdot |\beta_1-\alpha_1|\to 0,\, n\to\infty,\\
          & \|AB_n-B_nA\|_{L_1}\leq \left|\theta_{A,1}\sigma_1 \varsigma_n\right|\cdot |\beta-\alpha|\to 0,\, n\to\infty,
         \end{align*}
         because $\varsigma_n\to 0$.  Since the sequence of operators  $\{B_n\}$ converges in norm to $\tilde{B}$, then by Lemma \ref{LemmaOnseSequenceConvOptorsCommute}, $\tilde{B}$ commutes with $A$.

        \noindent\ref{Proposition3ExampleTheoremIntOpRepGenKernLpConvSeqComOp:item4}. The conclusion follows from
         items  \ref{Proposition3ExampleTheoremIntOpRepGenKernLpConvSeqComOp:item2}, \ref{Proposition3ExampleTheoremIntOpRepGenKernLpConvSeqComOp:item3}, Lemma \ref{LemmaOnseSequenceConvOptorsCommute} and Lemma \ref{LemmaTwoOpSequencesConvOptorsCommute}.

         \noindent\ref{Proposition3ExampleTheoremIntOpRepGenKernLpConvSeqComOp:item5}. By following the same procedure as on item \ref{Proposition3ExampleTheoremIntOpRepGenKernLpConvSeqComOp:item2}, the operator $AB-BA$ has the following estimation. For all $x\in L_p(\mathbb{R},\mu)$ and for all
         $\resizebox{0.45\hsize}{!}{$\displaystyle (\theta_{A,1},\theta_{B,1}, \theta_{B,3}, \alpha_1,\beta_1,\alpha,\beta,\sigma_1,\sigma_2,\delta,\omega)\in\Lambda $}$ we have
         \begin{equation*}
           \|(AB-BA)x\|_{L_p}\leq |\theta_{A,1}\sigma_1 \theta_{B,3}|\cdot \tilde{\lambda} \cdot \|x\|_{L_p}\to 0,
         \end{equation*}
           when   $\theta_{A,1}\sigma_1 \theta_{B,3}\tilde{\lambda}\to 0$, where $\tilde{\lambda}=|\beta_1-\alpha_1|^{\frac{1}{p}}\cdot|\beta_1-\alpha_1|^{\frac{1}{q}}$, if $1< p < \infty$, $\tilde{\lambda}=|\beta_1-\alpha_1|$, if $p=1$ and
           $\tilde{\lambda}=|\beta-\alpha|$, if $p=1$.
    \QEDB
   \end{proof}

 \begin{theorem}
 Let $A:L_p(\mathbb{R},\mu)\to L_p(\mathbb{R},\mu)$, $B:L_p(\mathbb{R},\mu)\to L_p(\mathbb{R},\mu)$, $1\le p\le \infty$ be operators defined as
 \begin{eqnarray*}
         (A x)(t)&=&\int\limits_{\alpha_1}^{\beta_1}I_{[\alpha,\beta]}(t)\big(\theta_{A,1}\sin(\omega t)\cos(\omega s)+ \theta_{A,2}\cos(\omega t)\cos(\omega s)  \\
         & &  + \frac{\delta\theta_{A,2}^2\sigma_2^2}{\sigma_1}\sin(\omega t)\sin(\omega s)\big)x(s)d\mu_s,\\
         (B x)(t)&=&\int\limits_{\alpha_1}^{\beta_1}I_{[\alpha,\beta]}(t)\theta_{B,1}
         \sin(\omega t)\cos(\omega s)x(s)d\mu_s
        \end{eqnarray*}
        for almost every $t$, where $\theta_{A,1}, \theta_{A,2},\, \omega\in\mathbb{R}$,  $\delta\in\mathbb{R}\setminus\{0\}$, $\alpha,\beta,\alpha_1,\beta_1 \in\mathbb{R}$, $\alpha_1<\beta_1$, $\alpha\le \alpha_1$, $\beta\geq \beta_1$ and, either the number $\frac{\omega }{\pi}(\beta_1-\alpha_1)\in \mathbb{Z}$ or $\frac{\omega }{\pi}(\beta_1+\alpha_1)\in \mathbb{Z}$,
         $\sigma_1,\sigma_2\in\mathbb{R}\setminus\{0\}$ are defined in \eqref{ConstantSigma1Case2Omega} and \eqref{ConstantSigma2Case2Omega}, respectively.
        Let $\{A_n\}: L_p(\mathbb{R},\mu)\to L_p(\mathbb{R},\mu)$, $\{B_n\}: L_p(\mathbb{R},\mu)\to L_p(\mathbb{R},\mu)$ $1\le p \le \infty$ be sequences of operators
        \begin{eqnarray*}
         (A_n x)(t)&=&\int\limits_{\alpha_1}^{\beta_1}I_{[\alpha,\beta]}(t)\big(\theta_{A,1}\sin(\omega t)\cos(\omega s)+ \theta_{n}\cos(\omega t)\cos(\omega s)  \\
         & &   + \frac{\delta\theta_{n}^2\sigma_2^2}{\sigma_1}\sin(\omega t)\sin(\omega s)\big)x(s)d\mu_s,\\
         (B_n x)(t)&=&\int\limits_{\alpha_1}^{\beta_1}I_{[\alpha,\beta]}(t)\varsigma_n
         \sin(\omega t)\cos(\omega s)x(s)d\mu_s,
        \end{eqnarray*}
        for almost every $t$, $\{\theta_n\}$ and $\{\varsigma_n\}$ are number sequences. Let
         \begin{eqnarray*}
          \Lambda&=&\big\{(\theta_{A,1},\theta_{A,2},\theta_{B,1}, \alpha_1,\beta_1,\alpha,\beta,\sigma_1,\sigma_2,\delta,\omega)\in\mathbb{R}^{11}:\, \delta\not=0,\, \alpha \leq \alpha_1,\ \beta\geq \beta_1,\\
         &&\hspace{-1mm}  \int\limits_{\alpha_1}^{\beta_1}\sin(\omega s)\cos(\omega s)d\mu_s=0,  \sigma_1=\int\limits_{\alpha_1}^{\beta_1}(\sin(\omega s))^2d\mu_s\not=0,  \sigma_2=\beta_1-\alpha_1-\sigma_1\not=0 \big\}.
         \end{eqnarray*}
          Then
       \begin{enumerate}[label=\textup{\arabic*.}, ref=\arabic*]
        \item\label{PropositionExampleTheoremIntOpRepGenKernLpConvSeqComOp:item1}  $AB=\delta BA^2$,  $A_nB=\delta BA_n^2$, $A_nB_n=\delta B_nA_n^2 $, $AB_n=\delta B_nA^2 $ for each positive integer $n$. Moreover, for all $x\in L_p(\mathbb{R},\mu)$, $1\le p\le \infty$ we have
        \begin{equation*}
         (AB-BA)x(t)=\theta_{A,2}\sigma_2 \theta_{B,1}(\delta\theta_{A,2}\sigma_2-1)\int\limits_{\alpha_1}^{\beta_1} I_{[\alpha,\beta]}(t)\sin(\omega t)\cos(\omega s)x(s)d\mu_s,
         \end{equation*}
        for almost every $t$.

         \item\label{PropositionExampleTheoremIntOpRepGenKernLpConvSeqComOp:item2} if $\theta_n\to 0$ when $n\to \infty$, then $A_n\to \bar{A}$ (converge in norm)
               defined as
            \begin{equation*}
             (\bar{A} x)(t)=\int\limits_{\alpha_1}^{\beta_1}I_{[\alpha,\beta]}(t)\theta_{A,1}\sin(\omega t)\cos(\omega s)x(s)d\mu_s,
            \end{equation*}
           for almost every $t$, so $A_nB-BA_n\to 0$ (converge in norm) and $\bar{A}B=B\bar{A}$;
           \item\label{PropositionExampleTheoremIntOpRepGenKernLpConvSeqComOp:item3} if $\theta_n\to \frac{1}{\delta \sigma_2}$ when $n\to \infty$, then $A_n\to \tilde{A}$ (converge in norm) defined as
          \begin{eqnarray*}
             (\tilde{A} x)(t)&=&\int\limits_{\alpha_1}^{\beta_1}I_{[\alpha,\beta]}(t)\big(\theta_{A,1}\sin(\omega t)\cos(\omega s)+ \frac{1}{\delta \sigma_2}\cos(\omega t)\cos(\omega s)  \\
            & &  + \frac{1}{\delta\sigma_1}\sin(\omega t)\sin(\omega s)\big)x(s)d\mu_s,
            \end{eqnarray*}
            for almost every $t$, so $A_nB-BA_n\to 0$ (converge in norm) and $\tilde{A}B=B\tilde{A}$;
          \item\label{PropositionExampleTheoremIntOpRepGenKernLpConvSeqComOp:item4} if $\varsigma_n\to 0$ when $n\to \infty$, then $B_n\to 0$ (converge in norm) and so $AB_n-B_nA\to 0$ (converge in norm);
           \item\label{PropositionExampleTheoremIntOpRepGenKernLpConvSeqComOp:item5} if (either $\theta_n\to 0$ or $\theta_n\to \frac{1}{\delta \sigma_2}$) and $\varsigma_n\to 0$ when $n\to \infty$, then $A_nB_n -B_nA_n \to 0$ (converge in norm).
           \item\label{PropositionExampleTheoremIntOpRepGenKernLpConvSeqComOp:item6} if $\theta_{A,2}\sigma_2 \theta_{B,1}(\delta\theta_{A,2}\sigma_2-1)\tilde{\lambda}\to 0$ in $\mathbb{R}$ when $$(\theta_{A,1},\theta_{A,2},\theta_{B,1}, \alpha_1,\beta_1,\alpha,\beta,\sigma_1,\sigma_2,\delta,\omega)\to \lambda_0\in cl(\Lambda)
               \mbox{ (the closure of } \Lambda),
               $$
                where
               $(\theta_{A,1},\theta_{A,2},\theta_{B,1}, \alpha_1,\beta_1,\alpha,\beta,\sigma_1,\sigma_2,\delta,\omega)\in\Lambda$ and, $\tilde{\lambda}=|\beta-\alpha|^{\frac{1}{p}}|\beta_1-\alpha_1|^{\frac{1}{q}}$ for $1<p < \infty$ with $\frac{1}{p}+\frac{1}{q}=1$, $\tilde{\lambda}=|\beta_1-\alpha_1|$ for $p=\infty$ and $\tilde{\lambda}=|\beta-\alpha|$ for $p=1$,  then for all $x\in L_p(\mathbb{R},\mu)$, $1\leq p\leq \infty$, it holds that $\|(AB-BA)x\|_{L_p}\to 0$. 
           \end{enumerate}
         \end{theorem}

        \begin{proof}
         \ref{PropositionExampleTheoremIntOpRepGenKernLpConvSeqComOp:item1}. It follows from direct computation.

         \noindent\ref{PropositionExampleTheoremIntOpRepGenKernLpConvSeqComOp:item2}.
         We have
         \begin{eqnarray*}
           && (A_n-\bar A)x(t)=\int\limits_{\alpha_1}^{\beta_1}I_{[\alpha_1,\beta_1]}(t)\theta_n \cos(\omega t)\cos(\omega s)x(s)d\mu_s\\
           &&\
          + \int\limits_{\alpha_1}^{\beta_1}I_{[\alpha_1,\beta_1]}(t)\frac{\delta\theta^2_{n}\sigma_2^2}{\sigma_1}\sin(\omega t)\sin(\omega s)x(s)d\mu_s
          =(A_n-\bar A)_1x(t)+(A_n-\bar A)_2x(t),
         \end{eqnarray*}
         respectively, for almost every $t$. Moreover, by applying H\"older inequality we  have for all $x\in L_p(\mathbb{R},\mu)$, $1< p<\infty $ the following estimations:
         \begin{eqnarray*}
         && \|(A_n-\bar A)_1x\|^p_{L_p}=  \int\limits_{\mathbb{R}} \left| \int\limits_{\alpha_1}^{\beta_1}I_{[\alpha,\beta]}(t)\theta_n \cos(\omega t)\cos(\omega s)x(s)d\mu_s\right|^pd\mu_t\\
         &&\leq |\theta_n|^p |\beta-\alpha| \big|\int\limits_{\mathbb{R}} I_{[\alpha,\beta]}(s) \cos(\omega s)x(s)d\mu_s\big|^p \hspace{-1mm}\leq |\theta_n|^p |\beta-\alpha||\beta_1-\alpha_1|^{\frac{p}{q}}\|x\|^p_{L_p}.
         \end{eqnarray*}
         Therefore, for $1< p<\infty $ we have the following norm estimation:
         \begin{equation*}
          \|(A_n-\bar A)_1\|_{L_p}\le  |\theta_n|\cdot|\beta-\alpha|^{\frac{1}{p}}\cdot|\beta_1-\alpha_1|^{\frac{1}{q}}.
         \end{equation*}
         By applying H\"older inequality we have for all $x\in L_\infty(\mathbb{R},\mu)$ the following:
         \begin{eqnarray*}
         && \|(A_n-\bar A)_1x\|_{L_\infty}= \mathop{\esssup}_{t\in \mathbb{R}} \left| \int\limits_{\alpha_1}^{\beta_1}I_{[\alpha,\beta]}(t)\theta_n \cos(\omega t)\cos(\omega s)x(s)d\mu_s\right|
         \\ &&\leq
         |\theta_n|\cdot \int\limits_{\mathbb{R}} \left|I_{[\alpha_1,\beta_1]}(s) \cos(\omega s)\right|d\mu_s\cdot \|x\|_{L_\infty} \le |\theta_n|\cdot|\beta_1-\alpha_1|\cdot\|x\|_{L_\infty}.
         \end{eqnarray*}
          Therefore, for $ p=\infty $ we have the following norm estimation:
         \begin{equation*}
          \|(A_n-\bar A)_1\|_{L_\infty}\le  |\theta_n|\cdot|\beta_1-\alpha_1|.
         \end{equation*}
         For all $x\in L_1(\mathbb{R},\mu)$,
         \begin{align*}
         & \|(A_n-\bar A)_1x\|_{L_1}= \int\limits_{\mathbb{R}} \big| \int\limits_{\alpha_1}^{\beta_1}I_{[\alpha,\beta]}(t)\theta_n \cos(\omega t)\cos(\omega s)x(s)d\mu_s\big|d\mu_t
         \\
         &
         \leq
         |\theta_n|\cdot |\beta-\alpha|\int\limits_{\mathbb{R}} |I_{[\alpha_1,\beta_1]}(s) \cos(\omega s)x(s)|d\mu_s \leq |\theta_n|\cdot|\beta-\alpha|\cdot\|x\|_{L_1}.
         \end{align*}
          Therefore, for $ p=1 $ we have the following norm estimation:
         \begin{equation*}
          \|(A_n-\bar A)_1\|_{L_1}\leq  |\theta_n|\cdot|\beta-\alpha|.
         \end{equation*}
         Analogously, we  have for  $1< p<\infty $ the following estimations:
         \begin{equation*}
          \|(A_n-\bar A)_2\|_{L_p}\le  \left|\frac{\delta\theta^2_{n}\sigma_2^2}{\sigma_1}\right| \cdot|\beta-\alpha|^{\frac{1}{p}}\cdot|\beta_1-\alpha_1|^{\frac{1}{q}}.
         \end{equation*}
         For $p=\infty$ and $p=1$ we have
         \begin{equation}
          \|(A_n-\bar A)_2\|_{L_\infty}\le  \left|\frac{\delta\theta^2_{n}\sigma_2^2}{\sigma_1}\right| \cdot|\beta_1-\alpha_1|\
          \|(A_n-\bar A)_2\|_{L_1}\le  \left|\frac{\delta\theta^2_{n}\sigma_2^2}{\sigma_1}\right| \cdot|\beta-\alpha|.
         \end{equation}
         Therefore, by Minkowsky inequality \cite{AdamsG,FollandRA,KolmogorovVol2} we have for $1< p <\infty$
         \begin{eqnarray*}
             \|A_n-\bar A\|_{L_p}&\leq & \|(A_n-\bar A)_1\|_{L_p}+\|(A_n-\bar A)_2\|_{L_p} \\
             &\leq & |\theta_n|\cdot |\beta-\alpha|^{\frac{1}{p}}\cdot|\beta_1-\alpha_1|^{\frac{1}{q}}  \left(1+
              \frac{|\delta\theta_{n}\sigma_2^2|}{|\sigma_1|}  \right)\to 0,
         \end{eqnarray*}
         when $n\to \infty$, since $\theta_n\to 0$. By triangle inequality we have for $p=\infty$
         \begin{align*}
            & \resizebox{1\hsize}{!}{$\displaystyle  \|A_n-\bar A\|_{L_\infty}\leq  \|(A_n-\bar A)_1\|_{L_\infty}+\|(A_n-\bar A)_2\|_{L_\infty} 
             \leq  |\theta_n| \cdot|\beta_1-\alpha_1|  \left(1+
              \frac{|\delta\theta_{n}\sigma_2^2|}{|\sigma_1|}  \right)\to 0, $}
         \end{align*}
         when $n\to \infty$ since $\theta_n\to 0$. Again, by triangle inequality we have for $p=1$
         \begin{align*}
            & \resizebox{1\hsize}{!}{$\displaystyle  \|A_n-\bar A\|_{L_1}\leq  \|(A_n-\bar A)_1\|_{L_1}+\|(A_n-\bar A)_2\|_{L_1} 
             \leq  |\theta_n| \cdot|\beta-\alpha|  \left(1+
              \frac{|\delta\theta_{n}\sigma_2^2|}{|\sigma_1|}  \right)\to 0, $}
         \end{align*}
         when $n\to \infty$ since $\theta_n\to 0$. Therefore, $A_n$ converges in norm to $\bar{A}$ in $L_p$, $1\le p\le \infty$.

         Now we prove the convergence for the sequence $\{A_nB-BA_n\}$. By applying H\"older inequality we have for all $x\in L_p(\mathbb{R},\mu)$, $1< p<\infty $ the following:
         \begin{align*}
            &\|(A_nB-BA_n)x\|^p_{L_p}= \left|\theta_{n}\sigma_2 \theta_{B,1}(\delta\theta_{n}\sigma_2-1)\right|^p\cdot
           \int\limits_{\mathbb{R}} \big|I_{[\alpha,\beta]}(t)\sin(\omega t)
          \\
          &
           \hspace{3cm} \cdot\int\limits_{\alpha_1}^{\beta_1} \cos (\omega s)x(s) d\mu_s\big|^p d\mu_t
           \\
           &
           \leq
           \left|\theta_{n}\sigma_2 \theta_{B,1}(\delta\theta_{n}\sigma_2-1)\right|^p\cdot
           |\beta-\alpha|
           \big|\int\limits_{\mathbb{R}} I_{[\alpha_1,\beta_1]}(s)\cos(\omega s)x(s)d\mu_s\big|^p
           \\
           & \le
           \left|\theta_{n}\sigma_2 \theta_{B,1}(\delta\theta_{n}\sigma_2-1)\right|^p |\beta-\alpha|  |\beta_1-\alpha_1|^{\frac{p}{q}}\cdot \|x\|^p.
         \end{align*}
         Therefore, for $1< p <\infty$ we have
         \begin{equation*}
         \|A_nB-BA_n\|_{L_p} \le  \left|\theta_{n}\sigma_2 \theta_{B,1}(\delta\theta_{n}\sigma_2-1)\right|\cdot
         |\beta-\alpha|^{\frac{1}{p}}\cdot |\beta_1-\alpha_1|^{\frac{1}{q}}\to 0,\, n\to\infty,
         \end{equation*}
         since $\theta_n\to 0$.
         By applying H\"older inequality, we have for  all $x\in L_\infty(\mathbb{R},\mu)$ the following:
            \begin{align*}
          & \|(A_nB-BA_n)x\|_{L_\infty}=\mathop{\esssup}_{t\in \mathbb{R}} \big|\theta_{n}\sigma_2 \theta_{B,1}(\delta\theta_{n}\sigma_2-1) I_{[\alpha,\beta]}(t)\sin(\omega t)
          \\
          & \hspace{3cm}\cdot\int\limits_{\alpha_1}^{\beta_1}\cos(\omega s)x(s)d\mu_s\big|
          \\
          &
          \le \left|\theta_{n}\sigma_2 \theta_{B,1}(\delta\theta_{n}\sigma_2-1)\right| \int\limits_{\mathbb{R}}| I_{[\alpha_1,\beta_1]}(s) \cos(\omega s)|d\mu_s \|x\|_{L_\infty}\\
          & \le \left|\theta_{n}\sigma_2 \theta_{B,1}(\delta\theta_{n}\sigma_2-1)\right|\cdot |\beta_1-\alpha_1|\|x\|_{L_\infty}
         \end{align*}
         Therefore, for $p=\infty$ we have
           \begin{equation*}
           \|A_nB-BA_n\|_{L_\infty}\le \left|\theta_{n}\sigma_2 \theta_{B,1}(\delta\theta_{n}\sigma_2-1)\right|\cdot |\beta_1-\alpha_1|\to 0,\, n\to\infty,
         \end{equation*}
         because $\theta_n\to 0$.
          For  all $x\in L_1(\mathbb{R},\mu)$,
            \begin{align*}
          & \|(A_nB-BA_n)x\|_{L_1}=\int\limits_{\mathbb{R}} \big|\theta_{n}\sigma_2 \theta_{B,1}(\delta\theta_{n}\sigma_2-1) I_{[\alpha,\beta]}(t)\sin(\omega t)  \\
          & \hspace{4cm}\cdot\int\limits_{\alpha_1}^{\beta_1}\cos(\omega s)x(s)d\mu_s\big|d\mu_t
          \\
          &
          \leq \left|\theta_{n}\sigma_2 \theta_{B,1}(\delta\theta_{n}\sigma_2-1)\right||\beta-\alpha| \int\limits_{\mathbb{R}}| I_{[\alpha_1,\beta_1]}(s) \cos(\omega s)x(s)|d\mu_s\\
          & \le \left|\theta_{n}\sigma_2 \theta_{B,1}(\delta\theta_{n}\sigma_2-1)\right|\cdot |\beta-\alpha|\|x\|_{L_1}.
         \end{align*}
         Therefore, for $p=1$ we have
           \begin{equation*}
           \|A_nB-BA_n\|_{L_1}\le \left|\theta_{n}\sigma_2 \theta_{B,1}(\delta\theta_{n}\sigma_2-1)\right|\cdot |\beta-\alpha|\to 0,\, n\to\infty,
         \end{equation*}
         because $\theta_n\to 0$.
          Since the sequence of operators  $\{A_n\}$  converges in norm to $\bar{A}$, then by Lemma \ref{LemmaOnseSequenceConvOptorsCommute}, $\bar{A}$ commutes with $B$.

         \noindent\ref{PropositionExampleTheoremIntOpRepGenKernLpConvSeqComOp:item3}. We have
         \begin{eqnarray*}
           (A_n-\tilde A)x(t)&=&\int\limits_{\alpha_1}^{\beta_1}I_{[\alpha_1,\beta_1]}(t) \left(\theta_n-\frac{1}{\delta \sigma_2}\right) \cos(\omega t)\cos(\omega s)x(s)d\mu_s\\
          & &+ \int\limits_{\alpha_1}^{\beta_1}I_{[\alpha_1,\beta_1]}(t)\left(\frac{\delta\theta^2_{n}\sigma_2^2}{\sigma_1}-\frac{1}{\delta\sigma_1}\right)\sin(\omega t)\sin(\omega s)x(s)d\mu_s\\
          &=&(A_n-\tilde A)_1x(t)+(A_n-\tilde A)_2x(t),
         \end{eqnarray*}
         respectively, for almost every $t$. By applying the same procedure of the item \ref{PropositionExampleTheoremIntOpRepGenKernLpConvSeqComOp:item2} we get for $1< p<\infty$
         \begin{eqnarray*}
             \|A_n-\tilde A\|_{L_p}&\leq & \|(A_n-\tilde A)_1\|_{L_p}+\|(A_n-\tilde A)_2\|_{L_p} \\
             &\leq &|\beta-\alpha|^{\frac{1}{p}}\cdot|\beta_1-\alpha_1|^{\frac{1}{q}}  \left(\left|\theta_n-\frac{1}{\delta \sigma_2}\right|+
              \frac{|\delta^2\theta^2_{n}\sigma_2^2-1|}{|\delta\sigma_1|}  \right)\to 0,
         \end{eqnarray*}
         when $n\to \infty$ since $\theta_n\to \frac{1}{\delta \sigma_2}$. For $p=\infty$ we get
         \begin{align*}
             \|A_n-\tilde A\|_{L_\infty}\leq & \|(A_n-\tilde A)_1\|_{L_\infty}+\|(A_n-\tilde A)_2\|_{L_\infty} \\
             \leq &|\beta_1-\alpha_1|  \left(\left|\theta_n-\frac{1}{\delta \sigma_2}\right|+
              \frac{|\delta^2\theta^2_{n}\sigma_2^2-1|}{|\delta\sigma_1|}  \right)\to 0,
         \end{align*}
         when $n\to \infty$ since $\theta_n\to \frac{1}{\delta \sigma_2}$. Moreover, for $p=1$ we have
          \begin{align*}
             \|A_n-\tilde A\|_{L_1}\leq & \|(A_n-\tilde A)_1\|_{L_1}+\|(A_n-\tilde A)_2\|_{L_1} \\
             \leq &|\beta-\alpha|  \left(\left|\theta_n-\frac{1}{\delta \sigma_2}\right|+
              \frac{|\delta^2\theta^2_{n}\sigma_2^2-1|}{|\delta\sigma_1|}  \right)\to 0,
         \end{align*}
         when $n\to \infty$ since $\theta_n\to \frac{1}{\delta \sigma_2}$.
         Therefore, $A_n$ converges in norm to $\tilde{A}$ in $L_p$, with $1\le p\le \infty$.
         By using applying results from item \ref{PropositionExampleTheoremIntOpRepGenKernLpConvSeqComOp:item2} we have for $1< p <\infty$
         \begin{equation*}
         \|A_nB-BA_n\|_{L_p} \le  \left|\theta_{n}\sigma_2 \theta_{B,1}(\delta\theta_{n}\sigma_2-1)\right|\cdot
         |\beta-\alpha|^{\frac{1}{p}}\cdot |\beta_1-\alpha_1|^{\frac{1}{q}}\to 0,\, n\to\infty
         \end{equation*}
         since $\theta_n\to \frac{1}{\delta \sigma_2}$.  For $p=\infty$ we have
           \begin{equation*}
           \|A_nB-BA_n\|_{L_\infty}\le \left|\theta_{n}\sigma_2 \theta_{B,1}(\delta\theta_{n}\sigma_2-1)\right|\cdot |\beta_1-\alpha_1|\to 0,\, n\to\infty,
         \end{equation*}
         because $\theta_n\to \frac{1}{\delta\sigma_2}$.
         And for $p=1$ we have
           \begin{equation*}
           \|A_nB-BA_n\|_{L_1}\le \left|\theta_{n}\sigma_2 \theta_{B,1}(\delta\theta_{n}\sigma_2-1)\right|\cdot |\beta-\alpha|\to 0,\, n\to\infty,
         \end{equation*}
         because $\theta_n\to \frac{1}{\delta\sigma_2}$.
         Since the sequence of operators  $\{A_n\}$  converges in norm to $\tilde{A}$, then, by Lemma \ref{LemmaOnseSequenceConvOptorsCommute}, $\tilde{A}$ commutes with $B$.

         \noindent\ref{PropositionExampleTheoremIntOpRepGenKernLpConvSeqComOp:item4}. By applying the same procedure of the item \ref{PropositionExampleTheoremIntOpRepGenKernLpConvSeqComOp:item2} we get
         we get for $1< p<\infty$
         \begin{eqnarray*}
             \|B_n\|_{L_p}  &\leq &|\beta-\alpha|^{\frac{1}{p}}\cdot|\beta_1-\alpha_1|^{\frac{1}{q}}|\varsigma_n|  \to 0,
         \end{eqnarray*}
         when $n\to \infty$ since $\varsigma_n\to 0$. For $p=\infty$ we get
         \begin{eqnarray*}
             \|B_n\|_{L_\infty}
             &\leq &|\beta_1-\alpha_1|  |\varsigma_n|\to 0,
         \end{eqnarray*}
         when $n\to \infty$ since $\varsigma_n\to 0$.
          And for $p=1$ we get
         \begin{eqnarray*}
             \|B_n\|_{L_1}
             &\leq &|\beta-\alpha|  |\varsigma_n|\to 0,
         \end{eqnarray*}
         when $n\to \infty$ since $\varsigma_n\to 0$.
          Therefore, $B_n$ converges in norm to $0$ in $L_p$, $1\le p\le \infty$.

         In order to prove that $\{AB_n-B_nA\}$ converges in norm to 0. We apply the same procedure of the item \ref{PropositionExampleTheoremIntOpRepGenKernLpConvSeqComOp:item2} for $1< p<\infty $, so we get for $1< p<\infty$
         \begin{equation*}
           \|AB_n-B_nA\|_{L_p}\le \left|\theta_{A,2}\sigma_2 \varsigma_n(\delta\theta_{A,2}\sigma_2-1)\right|\cdot |\beta-\alpha|^{\frac{1}{p}}|\beta_1-\alpha_1|^{\frac{1}{q}}\to 0,\, n\to\infty,
         \end{equation*}
         since $\varsigma_n\to 0$. For $p=\infty$ we get
           \begin{equation*}
           \|AB_n-B_nA\|_{L_\infty}\le \left|\theta_{A,2}\sigma_2 \varsigma_n(\delta\theta_{A,2}\sigma_2-1)\right|\cdot |\beta_1-\alpha_1|\to 0,\, n\to\infty,
         \end{equation*}
         because $\varsigma_n\to 0$. And for $p=1$ we get
           \begin{equation*}
           \|AB_n-B_nA\|_{L_1}\le \left|\theta_{A,2}\sigma_2 \varsigma_n(\delta\theta_{A,2}\sigma_2-1)\right|\cdot |\beta-\alpha|\to 0,\, n\to\infty,
         \end{equation*}
         because $\varsigma_n\to 0$
           Since the sequence of operators  $\{B_n\}$ converges in norm, then by Lemma \ref{LemmaOnseSequenceConvOptorsCommute}, it converges to an operator which commutes with $A$.

         \noindent\ref{PropositionExampleTheoremIntOpRepGenKernLpConvSeqComOp:item5}. The conclusion follows from
         items  \ref{PropositionExampleTheoremIntOpRepGenKernLpConvSeqComOp:item2}, \ref{PropositionExampleTheoremIntOpRepGenKernLpConvSeqComOp:item3}, \ref{PropositionExampleTheoremIntOpRepGenKernLpConvSeqComOp:item4}, Lemma \ref{LemmaOnseSequenceConvOptorsCommute} and Lemma \ref{LemmaTwoOpSequencesConvOptorsCommute}.

         \noindent\ref{PropositionExampleTheoremIntOpRepGenKernLpConvSeqComOp:item6}. By following the same procedure of the item \ref{PropositionExampleTheoremIntOpRepGenKernLpConvSeqComOp:item2}, the operator $AB-BA$ has the following estimation. For all $x\in L_p(\mathbb{R},\mu)$ and for every
         $$ (\theta_{A,1},\theta_{A,2},\theta_{B,1}, \alpha_1,\beta_1,\alpha,\beta,\sigma_1,\sigma_2,\delta,\omega)\in \Lambda
         $$
         we have
         \begin{equation*}
           \|(AB-BA)x\|_{L_p}\leq |\theta_{A,2}\sigma_2 \theta_{B,1}(\delta\theta_{A,2}\sigma_2-1)|\cdot \tilde{\lambda} \cdot \|x\|_{L_p}\to 0,
         \end{equation*}
           when   $\theta_{A,2}\sigma_2 \theta_{B,1}(\delta\theta_{A,2}\sigma_2-1)\tilde{\lambda}\to 0$, where $\tilde{\lambda}=|\beta-\alpha|^{\frac{1}{p}}\cdot|\beta_1-\alpha_1|^{\frac{1}{q}}$, if $1< p <\infty$,  $\tilde{\lambda}=|\beta_1-\alpha_1|$, if $p=\infty$ and $\tilde{\lambda}=|\beta-\alpha|$, if $p=1$. 
    \QEDB
    \end{proof}

    \begin{theorem}\label{Proposition6Case2BSequencesNoncomutOpConvComutOp}
       Let $A:L_p(\mathbb{R},\mu)\to L_p(\mathbb{R},\mu)$, $B:L_p(\mathbb{R},\mu)\to L_p(\mathbb{R},\mu)$, $1\le p\le \infty$ be operators defined as
       \begin{eqnarray*}
           (Ax)(t) &=&  \int\limits_{\alpha_1}^{\beta_1}  I_{[\alpha,\beta]}(t)\big(\frac{1}{\delta\sigma_2}\cos(\omega t)\cos(\omega s)+\frac{1}{\delta\sigma_1}\sin(\omega t)\sin(\omega s) \\
           & &  +\theta_{A,4}\cos(\omega t)\sin(\omega s)\big)x(s)d\mu_s, \\
           (Bx)(t) &=&  \int\limits_{\alpha_1}^{\beta_1}  I_{[\alpha,\beta]}(t)\big(\theta_{B,2}\cos(\omega t)\cos(\omega s) +\frac{2\sigma_2\theta_{B,2}}{\sigma_1}\sin(\omega t)\sin(\omega s)\big)x(s)d\mu_s,
           \end{eqnarray*}
           for almost every $t$, where $\theta_{A,4}, \theta_{B,2},\omega\in\mathbb{R}$,  $\delta \in\mathbb{R}\setminus\{0\}$, $\alpha,\beta,\alpha_1,\beta_1 \in\mathbb{R}$, $\alpha_1<\beta_1$, $\alpha\leq \alpha_1$, $\beta\geq \beta_1$ and, either the number $\frac{\omega }{\pi}(\beta_1-\alpha_1)\in \mathbb{Z}$ or $\frac{\omega }{\pi}(\beta_1+\alpha_1)\in \mathbb{Z}$,
           $\sigma_1,\sigma_2\in\mathbb{R}\setminus\{0\}$ are defined in \eqref{ConstantSigma1Case2Omega} and \eqref{ConstantSigma2Case2Omega}, respectively.
          Let $\{A_n\}: L_p(\mathbb{R},\mu)\to L_p(\mathbb{R},\mu)$, $\{B_n\}: L_p(\mathbb{R},\mu)\to L_p(\mathbb{R},\mu)$ $1\le p \le \infty$ be sequences of operators
        \begin{eqnarray*}
           (A_n x)(t) &=&  \int\limits_{\alpha_1}^{\beta_1}  I_{[\alpha,\beta]}(t)\big(\frac{1}{\delta\sigma_2}\cos(\omega t)\cos(\omega s)+\frac{1}{\delta\sigma_1}\sin(\omega t)\sin(\omega s) \\
           & &  +\theta_{n}\cos(\omega t)\sin(\omega s)\big)x(s)d\mu_s, \\
           (B_n x)(t) &=&  \int\limits_{\alpha_1}^{\beta_1}  I_{[\alpha,\beta]}(t)\left(\varsigma_{n}\cos(\omega t)\cos(\omega s) +\frac{2\sigma_2\varsigma_{n}}{\sigma_1}\sin(\omega t)\sin(\omega s)\right)x(s)d\mu_s,
           \end{eqnarray*}
           for almost every $t$, $\{\theta_n\}$ and $\{\varsigma_n\}$ are number sequences.
           Let
         \begin{align*}
          & \Lambda=\big\{(\theta_{A,4},\theta_{B,2}, \alpha_1,\beta_1,\alpha,\beta,\sigma_1,\sigma_2,\delta,\omega)\in\mathbb{R}^{10}: \delta\not=0, \alpha \leq \alpha_1, \beta\geq \beta_1,\\
         &  \ \int\limits_{\alpha_1}^{\beta_1}\sin(\omega s)\cos(\omega s)d\mu_s=0,  \sigma_1=\int\limits_{\alpha_1}^{\beta_1}(\sin(\omega s))^2d\mu_s\not=0,  \sigma_2=\beta_1-\alpha_1-\sigma_1\not=0 \big\}.
         \end{align*}
           Then
         \begin{enumerate}[label=\textup{\arabic*.}, ref=\arabic*]
           \item\label{Proposition6ExampleTheoremIntOpRepGenKernLpConvSeqComOp:item1}  $AB=\delta BA^2$,  $A_nB=\delta BA_n^2$, $A_nB_n=\delta B_nA_n^2 $, $AB_n=\delta B_nA^2 $ for each positive integer $n$. Moreover, for all $x\in L_p(\mathbb{R},\mu)$, $1\le p\le \infty$ we have
       \begin{equation*}
         (AB-BA)x(t)=\theta_{A,4}\sigma_2 \theta_{B,2}\int\limits_{\alpha_1}^{\beta_1} I_{[\alpha,\beta]}(t)\cos(\omega t)\sin(\omega s)x(s)d\mu_s,
       \end{equation*}
       for almost every $t$.
           \item\label{Proposition6ExampleTheoremIntOpRepGenKernLpConvSeqComOp:item2} if $\theta_n\to 0$ when $n\to \infty$, then $A_n\to \tilde{A}$ (converge in norm)
               defined as
            \begin{equation*}
           (\tilde A x)(t) =  \int\limits_{\alpha_1}^{\beta_1}  I_{[\alpha,\beta]}(t)\big(\frac{1}{\delta\sigma_2}\cos(\omega t)\cos(\omega s)+\frac{1}{\delta\sigma_1}\sin(\omega t)\sin(\omega s)\big)x(s)d\mu_s,
           \end{equation*}
            for almost every $t$, so $A_nB-BA_n\to 0$ (converge in norm) and $\tilde{A}B=B\tilde{A}$;
          \item\label{Proposition6ExampleTheoremIntOpRepGenKernLpConvSeqComOp:item3} if $\varsigma_n\to 0$ when $n\to \infty$, then $B_n\to 0$ (converge in norm) and so $AB_n-B_nA\to 0$ (converge in norm);
           \item\label{Proposition6ExampleTheoremIntOpRepGenKernLpConvSeqComOp:item4} if $\theta_n\to 0$ and $\varsigma_n\to 0$ when $n\to \infty$, then $A_nB_n -B_nA_n \to 0$ (converge in norm).
           \item\label{Proposition6ExampleTheoremIntOpRepGenKernLpConvSeqComOp:item5} if $\theta_{A,4}\sigma_2 \theta_{B,2}\tilde{\lambda}\to 0$ in $\mathbb{R}$ when $(\theta_{A,4},\theta_{B,2}, \alpha_1,\beta_1,\alpha,\beta,\sigma_1,\sigma_2,\delta,\omega)\to \lambda_0\in cl(\Lambda)$  (the closure of $\Lambda$), where
               $(\theta_{A,4},\theta_{B,2}, \alpha_1,\beta_1,\alpha,\beta,\sigma_1,\sigma_2,\delta,\omega)\in\Lambda$, and  $\tilde{\lambda}=|\beta-\alpha|^{\frac{1}{p}}|\beta_1-\alpha_1|^{\frac{1}{q}}$ for $1< p < \infty$ with $\frac{1}{p}+\frac{1}{q}=1$,  $\tilde{\lambda}=|\beta_1-\alpha_1|$ for $p=\infty$ and $\tilde{\lambda}=|\beta-\alpha|$ for $p=1$, then for all $x\in L_p(\mathbb{R},\mu)$, $1\leq p\leq \infty$, it holds that $\|(AB-BA)x\|_{L_p}\to 0$. 
         \end{enumerate}
    \end{theorem}

    \begin{proof}
      \noindent\ref{Proposition6ExampleTheoremIntOpRepGenKernLpConvSeqComOp:item1}. It follows from direct computation.

         \noindent\ref{Proposition6ExampleTheoremIntOpRepGenKernLpConvSeqComOp:item2}.  We have
         \begin{eqnarray*}
           (A_n-\tilde A)x(t)&=&\int\limits_{\alpha_1}^{\beta_1}I_{[\alpha_1,\beta_1]}(t)\theta_n \cos(\omega t)\sin(\omega s)x(s)d\mu_s,
         \end{eqnarray*}
        for almost every $t$. Moreover, by applying H\"older inequality we  have for every $x\in L_p(\mathbb{R},\mu)$, $1< p<\infty $ the following estimations:
         \begin{eqnarray*}
         && \|(A_n-\tilde A)x\|^p_{L_p}=  \int\limits_{\mathbb{R}} \big| \int\limits_{\alpha_1}^{\beta_1}I_{[\alpha,\beta]}(t)\theta_n \cos(\omega t)\sin(\omega s)x(s)d\mu_s\big|^pd\mu_t
         \\&&\leq
         |\theta_n|^p |\beta-\alpha| \big|\int\limits_{\mathbb{R}} I_{[\alpha,\beta]}(s) \sin(\omega s)x(s)d\mu_s\big|^p \le |\theta_n|^p |\beta-\alpha||\beta_1-\alpha_1|^{\frac{p}{q}}\cdot\|x\|^p_{L_p}.
         \end{eqnarray*}
         Therefore, for $1< p<\infty $ we have the following:
         \begin{equation*}
          \|(A_n-\tilde A)\|_{L_p}\le  |\theta_n|\cdot|\beta-\alpha|^{\frac{1}{p}}\cdot|\beta_1-\alpha_1|^{\frac{1}{q}}\to 0,
         \end{equation*}
         when $n\to \infty$, since $\theta_n\to 0$. By applying H\"older inequality we have for all $x\in L_\infty(\mathbb{R},\mu)$ the following:
         \begin{eqnarray*}
         && \|(A_n-\tilde A) x\|_{L_\infty}= \mathop{\esssup}_{t\in \mathbb{R}} \big| \int\limits_{\alpha_1}^{\beta_1}I_{[\alpha,\beta]}(t)\theta_n \cos(\omega t)\sin(\omega s)x(s)d\mu_s\big|\\
         &&\leq
         |\theta_n|\cdot \int\limits_{\mathbb{R}} \left|I_{[\alpha_1,\beta_1]}(s) \sin(\omega s)\right|d\mu_s\cdot \|x\|_{L_\infty} \le |\theta_n|\cdot|\beta_1-\alpha_1|\cdot\|x\|_{L_\infty}.
         \end{eqnarray*}
          Therefore, for $ p=\infty $ we have the following:
         \begin{equation*}
          \|A_n-\tilde A\|_{L_\infty}\le  |\theta_n|\cdot|\beta_1-\alpha_1|\to 0,
         \end{equation*}
         when $n\to\infty$ since $\theta_n\to 0$.
         For all $x\in L_1(\mathbb{R},\mu)$,
         \begin{multline*}
          \|(A_n-\tilde A) x\|_{L_1}= \int\limits_{\mathbb{R}} \big| \int\limits_{\alpha_1}^{\beta_1}I_{[\alpha,\beta]}(t)\theta_n \cos(\omega t)\sin(\omega s)x(s)d\mu_s\big|d\mu_t\\
         \leq
         |\theta_n|\cdot|\beta-\alpha| \int\limits_{\mathbb{R}} |I_{[\alpha_1,\beta_1]}(s) \sin(\omega s)x(s)|d\mu_s \leq |\theta_n|\cdot|\beta-\alpha|\cdot\|x\|_{L_1}.
         \end{multline*}
          Therefore, for $ p=1 $ we have the following:
         \begin{equation*}
          \|A_n-\tilde A\|_{L_1}\le  |\theta_n|\cdot|\beta-\alpha|\to 0,
         \end{equation*}
         when $n\to\infty$ since $\theta_n\to 0$.
         Therefore, $A_n\to\tilde A$ converge in norm in $L_p$, $1\le p\le \infty$.

         Now we prove the convergence for the sequence $\{A_nB-BA_n\}$. By applying H\"older inequality we have for all $x\in L_p(\mathbb{R},\mu)$, $1< p<\infty $ the following:
         \begin{eqnarray*}
          &&\resizebox{0.85\hsize}{!}{$ \|(A_nB-BA_n)x\|^p_{L_p}= \big|\theta_{n}\sigma_2 \theta_{B,2}\int\limits_{\alpha_1}^{\beta_1} I_{[\alpha,\beta]}(t)\cos(\omega t)\sin(\omega s)x(s)d\mu_s\big|^p d\mu_t $}
            \\
          && \resizebox{1\hsize}{!}{$ \leq  |\theta_{n}\sigma_2 \theta_{B,2}|^p
           |\beta-\alpha|
           \big|\int\limits_{\mathbb{R}} I_{[\alpha_1,\beta_1]}(s)\sin(\omega s)x(s)d\mu_s\big|^p
            \leq
           |\theta_{n}\sigma_2 \theta_{B,2}|^p |\beta-\alpha|^{\frac{1}{p}}
            |\beta_1-\alpha_1|^{\frac{1}{q}}\cdot \|x\|^p.$}
         \end{eqnarray*}
         Therefore, for $1< p <\infty$ we have
         \begin{equation*}
         \|A_nB-BA_n\|_{L_p} \leq  \left|\theta_{n}\sigma_2 \theta_{B,2}\right|\cdot
         |\beta-\alpha|^{\frac{1}{p}}\cdot |\beta_1-\alpha_1|^{\frac{1}{q}}\to 0,\, n\to\infty,
         \end{equation*}
         since $\theta_n\to 0$.
         By applying H\"older inequality, we have for  all $x\in L_\infty(\mathbb{R},\mu)$ the following:
            \begin{eqnarray*}
          && \|(A_nB-BA_n)x\|_{L_\infty}=\mathop{\esssup}_{t\in \mathbb{R}} \big|\theta_{n}\sigma_2 \theta_{B,2} I_{[\alpha,\beta]}(t)\sin(\omega t)\cdot \int\limits_{\alpha_1}^{\beta_1}\cos(\omega s)x(s)d\mu_s\big|\\
          &&\leq  \left|\theta_{n}\sigma_2 \theta_{B,2}\right| \int\limits_{\mathbb{R}}| I_{[\alpha_1,\beta_1]}(s) \cos(\omega s)|d\mu_s \|x\|_{L_\infty}
           \leq  \left|\theta_{n}\sigma_2 \theta_{B,2}\right|\cdot |\beta_1-\alpha_1|\|x\|_{L_\infty}.
         \end{eqnarray*}
         Therefore, for $p=\infty$ we have
           \begin{equation*}
           \|A_nB-BA_n\|_{L_\infty}\le \left|\theta_{n}\sigma_2 \theta_{B,2}\right|\cdot |\beta_1-\alpha_1|\to 0,\, n\to\infty,
         \end{equation*}
         because $\theta_n\to 0$.  For  all $x\in L_1(\mathbb{R},\mu)$,
            \begin{eqnarray*}
          && \|(A_nB-BA_n)x\|_{L_1}=\int\limits_{ \mathbb{R}} |\theta_{n}\sigma_2 \theta_{B,2} I_{[\alpha,\beta]}(t)\sin(\omega t)\cdot \int\limits_{\alpha_1}^{\beta_1}\cos(\omega s)x(s)d\mu_s|d\mu_t\\
          &&\leq  |\theta_{n}\sigma_2 \theta_{B,2}||\beta-\alpha| \int\limits_{\mathbb{R}}| I_{[\alpha_1,\beta_1]}(s) \cos(\omega s)|d\mu_s \|x\|_{L_1}
           \leq  |\theta_{n}\sigma_2 \theta_{B,2}|\cdot |\beta-\alpha|\|x\|_{L_1}.
         \end{eqnarray*}
         Therefore, for $p=1$ we have
           \begin{equation*}
           \|A_nB-BA_n\|_{L_1}\le \left|\theta_{n}\sigma_2 \theta_{B,2}\right|\cdot |\beta-\alpha|\to 0,\, n\to\infty,
         \end{equation*}
         because $\theta_n\to 0$.
         Since the sequence of operators  $\{A_n\}$  converges in norm to $\tilde{A}$, then by Lemma \ref{LemmaOnseSequenceConvOptorsCommute}, $\tilde{A}$ commutes with $B$.

         \noindent\ref{Proposition6ExampleTheoremIntOpRepGenKernLpConvSeqComOp:item3}.
         We have
         \begin{eqnarray*}
           (B_n)x(t)&=& \int\limits_{\alpha_1}^{\beta_1}I_{[\alpha_1,\beta_1]}(t) \varsigma_n \cos(\omega t)\cos(\omega s)x(s)d\mu_s+\\
          & &+ \int\limits_{\alpha_1}^{\beta_1}I_{[\alpha_1,\beta_1]}(t)\varsigma_{n}\frac{2\sigma_2}{\sigma_1}\sin(\omega t)\sin(\omega s)x(s)d\mu_s\\
          &=&(B_{n,1})x(t)+(B_{n,2})x(t),
         \end{eqnarray*}
         respectively, for almost every $t$. Moreover, by applying H\"older inequality we  have for all $x\in L_p(\mathbb{R},\mu)$, $1< p<\infty $ the following estimations:
         \begin{eqnarray*}
          \|(B_{n,1})x\|^p_{L_p}&=&  \int\limits_{\mathbb{R}} \big| \int\limits_{\alpha_1}^{\beta_1}I_{[\alpha,\beta]}(t)\varsigma_n \cos(\omega t)\cos(\omega s)x(s)d\mu_s\big|^pd\mu_t\\
         &\leq& |\varsigma_n |^p |\beta-\alpha|\cdot \big|\ \int\limits_{\mathbb{R}} I_{[\alpha_1,\beta_1]}(s) \cos(\omega s)x(s)d\mu_s\big|^p \\ &\leq&|\varsigma_n |^p |\beta-\alpha|\cdot|\beta_1-\alpha_1|^{\frac{p}{q}}\cdot\|x\|^p_{L_p}.
         \end{eqnarray*}
         Therefore, for $1< p<\infty $ we have the following norm estimation:
         \begin{equation*}
          \|(B_{n,1})\|_{L_p}\le  |\varsigma_n |\cdot|\beta-\alpha|^{\frac{1}{p}}\cdot|\beta_1-\alpha_1|^{\frac{1}{q}}.
         \end{equation*}
         By applying H\"older inequality we for all $x\in L_\infty(\mathbb{R},\mu)$ the following:
         \begin{eqnarray*}
         && \hspace{-0.1cm} \|(B_{n,1})x\|_{L_\infty}=  \mathop{\esssup}_{t\in \mathbb{R}} \big| \int\limits_{\alpha_1}^{\beta_1}I_{[\alpha,\beta]}(t)\varsigma_n \cos(\omega t)\cos(\omega s)x(s)d\mu_s\big|\\
         && \hspace{-0.1cm} \leq |\theta_n |\cdot \int\limits_{\mathbb{R}} |I_{[\alpha_1,\beta_1]}(s) \cos(\omega s)|d\mu_s\cdot \|x\|_{L_\infty} \leq |\varsigma_n |\cdot|\beta_1-\alpha_1|\cdot\|x\|_{L_\infty}.
         \end{eqnarray*}
          Therefore, for $p=\infty $ we have the following norm estimation:
         \begin{equation*}
          \|B_{n,1}\|_{L_\infty}\le  \left|\varsigma_n \right|\cdot|\beta_1-\alpha_1|.
         \end{equation*}
         For all $x\in L_1(\mathbb{R},\mu)$,
         \begin{eqnarray*}
         && \hspace{-0.1cm} \|(B_{n,1})x\|_{L_1}=  \int\limits_{\mathbb{R}} \big| \int\limits_{\alpha_1}^{\beta_1}I_{[\alpha,\beta]}(t)\varsigma_n \cos(\omega t)\cos(\omega s)x(s)d\mu_s\big|d\mu_t\\
         && \hspace{-0.1cm} \leq |\theta_n |\cdot |\beta-\alpha|\int\limits_{\mathbb{R}} |I_{[\alpha_1,\beta_1]}(s) \cos(\omega s)x(s)|d\mu_s \leq |\varsigma_n |\cdot|\beta-\alpha|\cdot\|x\|_{L_1}.
         \end{eqnarray*}
          Therefore, for $p=1 $ we have the following norm estimation:
         \begin{equation*}
          \|B_{n,1}\|_{L_1}\le  \left|\varsigma_n \right|\cdot|\beta-\alpha|.
         \end{equation*}
         Analogously, we  have for  $ 1<p<\infty $ the following estimations:
         \begin{equation*}
          \|B_{n,2}\|_{L_p}\le  \left|\varsigma_{n}\frac{2\sigma_2}{\sigma_1}\right| \cdot|\beta-\alpha|^{\frac{1}{p}}\cdot|\beta_1-\alpha_1|^{\frac{1}{q}},
         \end{equation*}
          for $p=\infty$ we have
         \begin{equation*}
          \|B_{n,2}\|_{L_\infty}\le  \left|\varsigma_{n}\frac{2\sigma_2}{\sigma_1}\right| \cdot|\beta_1-\alpha_1|
         \end{equation*}
         and for $p=1$ we have
         \begin{equation*}
          \|B_{n,2}\|_{L_1}\le  \left|\varsigma_{n}\frac{2\sigma_2}{\sigma_1}\right| \cdot|\beta-\alpha|
         \end{equation*}
         Therefore, by Minkowski inequality \cite{AdamsG,FollandRA,KolmogorovVol2} we have for $1< p <\infty$
         \begin{eqnarray*}
             \|B_n\|_{L_p}&\leq & \|B_{n,1}\|_{L_p}+\|B_{n,2}\|_{L_p} \\
             &\leq & |\varsigma_n |\cdot |\beta-\alpha|^{\frac{1}{p}}\cdot|\beta_1-\alpha_1|^{\frac{1}{q}}  \left(\frac{2|\sigma_2|}{|\sigma_1|}+
              1 \right)\to 0,
         \end{eqnarray*}
         when $n\to \infty$, since $\varsigma_n\to 0$. On the other hand, by triangle inequality we have for $p=\infty$ and $p=1$,
         \begin{align*}
            & \|B_n\|_{L_\infty}\leq  \|B_{n,1}\|_{L_\infty}+\|B_{n,2}\|_{L_\infty}
             \leq  |\varsigma_n| \cdot|\beta_1-\alpha_1|  \left(\frac{2|\sigma_2|}{|\sigma_1|}+
              1  \right)\to 0,
           \\
           &\|B_n\|_{L_1}\leq  \|B_{n,1}\|_{L_1}+\|B_{n,2}\|_{L_1}
             \leq  |\varsigma_n| \cdot|\beta-\alpha|  \left(\frac{2|\sigma_2|}{|\sigma_1|}+
              1  \right)\to 0,
         \end{align*}
         when $n\to \infty$ since $\varsigma_n\to 0$. Therefore, $B_n$ converges in norm to $0$ in $L_p$, $1\le p\le \infty$.

         In order to prove that $\{AB_n-B_nA\}$ converges in norm to 0. We apply the same procedure of the item \ref{Proposition6ExampleTheoremIntOpRepGenKernLpConvSeqComOp:item2} for $1<p<\infty $, so we get for $1< p<\infty$
         \begin{equation*}
           \|AB_n-B_nA\|_{L_p}\le \left|\theta_{A,4}\sigma_2 \varsigma_n\right|\cdot |\beta-\alpha|^{\frac{1}{p}}|\beta_1-\alpha_1|^{\frac{1}{q}}\to 0,\, n\to\infty,
         \end{equation*}
         since $\varsigma_n\to 0$. For $p=\infty$ we get
           \begin{equation*}
           \|AB_n-B_nA\|_{L_\infty}\le \left|\theta_{A,4}\sigma_2 \varsigma_n\right|\cdot |\beta_1-\alpha_1|\to 0,\, n\to\infty,
         \end{equation*}
         because $\varsigma_n\to 0$. For $p=1$ we get
           \begin{equation*}
           \|AB_n-B_nA\|_{L_1}\le \left|\theta_{A,4}\sigma_2 \varsigma_n\right|\cdot |\beta-\alpha|\to 0,\, n\to\infty,
         \end{equation*}
         because $\varsigma_n\to 0$.

        \noindent\ref{Proposition6ExampleTheoremIntOpRepGenKernLpConvSeqComOp:item4}. By using results from items
         \ref{Proposition6ExampleTheoremIntOpRepGenKernLpConvSeqComOp:item2}, \ref{Proposition6ExampleTheoremIntOpRepGenKernLpConvSeqComOp:item3} and norms properties we have for all $1\le p\le \infty$ the following:
         \begin{eqnarray*}
           \|A_nB_n-B_nA_n\|_{L_p}&\leq&  \|A_nB_n\|_{L_p}+ \|B_nA_n\|_{L_p}\\
           &\leq &\|A_n\|_{L_p}\|B_n\|_{L_p}+\|B_n\|_{L_p}\|A_n\|_{L_p}= 2 \|A_n\|_{L_p}\|B_n\|_{L_p}\to 0,
         \end{eqnarray*}
         when $n\to\infty$, since $\|A_n\|_{L_p}$ is bounded (because $A_n$ converges in norm to $\tilde{A}$) and $\|B_n\|\to 0$.

         \noindent\ref{Proposition6ExampleTheoremIntOpRepGenKernLpConvSeqComOp:item5}. By following the same procedure of the item \ref{Proposition6ExampleTheoremIntOpRepGenKernLpConvSeqComOp:item2}, the operator $AB-BA$ has the following estimation. For all $x\in L_p(\mathbb{R},\mu)$ and all $\resizebox{0.45\hsize}{!}{$(\theta_{A,4},\theta_{B,2}, \alpha_1,\beta_1,\alpha,\beta,\sigma_1,\sigma_2,\delta,\omega)\in\Lambda $}$ we have
         \begin{equation*}
           \|(AB-BA)x\|_{L_p}\leq |\theta_{A,4}\sigma_2 \theta_{B,2}|\cdot \tilde{\lambda} \cdot \|x\|_{L_p}\to 0,
         \end{equation*}
           when   $\theta_{A,4}\sigma_2 \theta_{B,2}\tilde{\lambda}\to 0$, where $\tilde{\lambda}=|\beta-\alpha|^{\frac{1}{p}}\cdot|\beta_1-\alpha_1|^{\frac{1}{q}}$, if $1< p <\infty$, $\tilde{\lambda}=|\beta_1-\alpha_1|$, if $p=\infty$ and $\tilde{\lambda}=|\beta-\alpha|$, if $p=1$. 
            \QEDB
    \end{proof}

      \begin{theorem}\label{Proposition7ExampleTheoremIntOpRepGenKernLpConvSeqComOp}
       Let $A:L_p(\mathbb{R},\mu)\to L_p(\mathbb{R},\mu)$, $B:L_p(\mathbb{R},\mu)\to L_p(\mathbb{R},\mu)$, $1\le p\le \infty$ be operators defined as
       \begin{eqnarray*}
           (Ax)(t) &=&  \int\limits_{\alpha_1}^{\beta_1}  I_{[\alpha,\beta]}(t)\big(\frac{1}{\delta\sigma_2}\cos(\omega t)\cos(\omega s)-\frac{1}{\delta\sigma_1}\sin(\omega t)\sin(\omega s) \\
           & &  +\theta_{A,4}\cos(\omega t)\sin(\omega s)\big)x(s)d\mu_s, \\
           (Bx)(t) &=&  \int\limits_{\alpha_1}^{\beta_1}  I_{[\alpha,\beta]}(t)\theta_{B,2}\cos(\omega t)\cos(\omega s)x(s)d\mu_s,
           \end{eqnarray*}
           for almost every $t$, where $\theta_{A,4}, \theta_{B,2},\omega\in\mathbb{R}$,  $\delta \in\mathbb{R}\setminus\{0\}$, $\alpha,\beta,\alpha_1,\beta_1 \in\mathbb{R}$, $\alpha_1<\beta_1$, $\alpha_1\le \alpha$, $\beta_1\ge \beta$ and, either the number $\frac{\omega }{\pi}(\beta_1-\alpha_1)\in \mathbb{Z}$ or $\frac{\omega }{\pi}(\beta_1+\alpha_1)\in \mathbb{Z}$,
           $\sigma_1,\sigma_2\in\mathbb{R}\setminus\{0\}$ are defined in \eqref{ConstantSigma1Case2Omega} and \eqref{ConstantSigma2Case2Omega}, respectively.
          Let $\{A_n\}: L_p(\mathbb{R},\mu)\to L_p(\mathbb{R},\mu)$, $\{B_n\}: L_p(\mathbb{R},\mu)\to L_p(\mathbb{R},\mu)$ $1\le p \le \infty$ be sequences of operators
        \begin{eqnarray*}
           (A_n x)(t) &=&  \int\limits_{\alpha_1}^{\beta_1}  I_{[\alpha,\beta]}(t)\big(\frac{1}{\delta\sigma_2}\cos(\omega t)\cos(\omega s)+\frac{1}{\delta\sigma_1}\sin(\omega t)\sin(\omega s)
           \\
           & &  -\theta_{n}\cos(\omega t)\sin(\omega s)\big)x(s)d\mu_s,
           \\
           (B_n x)(t) &=&  \int\limits_{\alpha_1}^{\beta_1}  I_{[\alpha,\beta]}(t)\varsigma_{n}\cos(\omega t)\cos(\omega s)x(s)d\mu_s,
           \end{eqnarray*}
           for almost every $t$, $\{\theta_n\}$ and $\{\varsigma_n\}$ are number sequences.
           Let
         \begin{align*}
         & \Lambda=\big\{(\theta_{A,4},\theta_{B,2}, \alpha_1,\beta_1,\alpha,\beta,\sigma_1,\sigma_2,\delta,\omega)\in\mathbb{R}^{10}:\, \delta\not=0,\alpha \leq \alpha_1, \beta\geq \beta_1, \\
         & \int\limits_{\alpha_1}^{\beta_1}\sin(\omega s)\cos(\omega s)d\mu_s=0,
         \sigma_1=\int\limits_{\alpha_1}^{\beta_1}(\sin(\omega s))^2d\mu_s\not=0, \sigma_2=\beta_1-\alpha_1-\sigma_1\not=0
         \big\}.
         \end{align*}
           Then
         \begin{enumerate}[label=\textup{\arabic*.}, ref=\arabic*]
           \item\label{Proposition7ExampleTheoremIntOpRepGenKernLpConvSeqComOp:item1}  $AB=\delta BA^2$,  $A_nB=\delta BA_n^2$, $A_nB_n=\delta B_nA_n^2 $, $AB_n=\delta B_nA^2 $ for each positive integer $n$. Moreover, for all $x\in L_p(\mathbb{R},\mu)$, $1\le p\le \infty$ we have
       \begin{equation*}
         (AB-BA)x(t)=\theta_{A,4}\sigma_2 \theta_{B,2}\int\limits_{\alpha_1}^{\beta_1} I_{[\alpha,\beta]}(t)\cos(\omega t)\sin(\omega s)x(s)d\mu_s,
       \end{equation*}
        for almost every $t$.
        \item\label{Proposition7ExampleTheoremIntOpRepGenKernLpConvSeqComOp:item2} if $\theta_n\to 0$ when $n\to \infty$, then $A_n\to \tilde{A}$ (converge in norm)
               defined as
            \begin{equation*}
           (\tilde A x)(t) =  \int\limits_{\alpha_1}^{\beta_1}  I_{[\alpha,\beta]}(t)\big(\frac{1}{\delta\sigma_2}\cos(\omega t)\cos(\omega s)-\frac{1}{\delta\sigma_1}\sin(\omega t)\sin(\omega s)\big)x(s)d\mu_s,
           \end{equation*}
           for almost every $t$, so $A_nB-BA_n\to 0$ (converge in norm) and $\tilde{A}B=B\tilde{A}$;
          \item\label{Proposition7ExampleTheoremIntOpRepGenKernLpConvSeqComOp:item3} if $\varsigma_n\to 0$ when $n\to \infty$, then $B_n\to 0$ (converge in norm) and so $AB_n-B_nA\to 0$ (converge in norm);
          \item\label{Proposition7ExampleTheoremIntOpRepGenKernLpConvSeqComOp:item4} if $\theta_n\to 0$ and $\varsigma_n\to 0$ when $n\to \infty$, then $A_nB_n -B_nA_n \to 0$ (converge in norm).
             \item\label{Proposition7ExampleTheoremIntOpRepGenKernLpConvSeqComOp:item5} if $\theta_{A,4}\sigma_2 \theta_{B,2}\tilde{\lambda}\to 0$ in $\mathbb{R}$ when $(\theta_{A,4},\theta_{B,2}, \alpha_1,\beta_1,\alpha,\beta,\sigma_1,\sigma_2,\delta,\omega)\to \lambda_0\in cl(\Lambda)$,
                 (the closure of $\Lambda$), where
                   $(\theta_{A,4},\theta_{B,2}, \alpha_1,\beta_1,\alpha,\beta,\sigma_1,\sigma_2,\delta,\omega)\in\Lambda$
                    and $\tilde{\lambda}=|\beta-\alpha|^{\frac{1}{p}}|\beta_1-\alpha_1|^{\frac{1}{q}}$ for $1< p < \infty$ with $\frac{1}{p}+\frac{1}{q}=1$, $\tilde{\lambda}=|\beta_1-\alpha_1|$ for $p=\infty$ and $\tilde{\lambda}=|\beta-\alpha|$ for $p=1$, then for all $x\in L_p(\mathbb{R},\mu)$, $1\leq p\leq \infty$, it holds that $\|(AB-BA)x\|_{L_p}\to 0$.
     \end{enumerate}
    \end{theorem}

    \begin{proof}
   \noindent\ref{Proposition7ExampleTheoremIntOpRepGenKernLpConvSeqComOp:item1}. It follows from direct computation.
         \noindent\ref{Proposition7ExampleTheoremIntOpRepGenKernLpConvSeqComOp:item2}.
          We have
         \begin{eqnarray*}
           (A_n-\tilde A)x(t)&=&\int\limits_{\alpha_1}^{\beta_1}I_{[\alpha_1,\beta_1]}(t)\theta_n \cos(\omega t)\sin(\omega s)x(s)d\mu_s,
         \end{eqnarray*}
         for almost every $t$. Moreover, by applying H\"older inequality we  have for all $x\in L_p(\mathbb{R},\mu)$, $1\le p<\infty $ the following estimations:
         \begin{eqnarray*}
         && \|(A_n-\tilde A)x\|^p_{L_p}=  \int\limits_{\mathbb{R}} \big| \int\limits_{\alpha_1}^{\beta_1}I_{[\alpha,\beta]}(t)\theta_n \cos(\omega t)\sin(\omega s)x(s)d\mu_s\big|^pd\mu_t\\
         && \leq
         |\theta_n|^p |\beta-\alpha| \big|\int\limits_{\mathbb{R}} I_{[\alpha,\beta]}(s) \sin(\omega s)x(s)d\mu_s\big|^p \le |\theta_n|^p |\beta-\alpha||\beta_1-\alpha_1|^{\frac{p}{q}}\cdot\|x\|^p_{L_p}.
         \end{eqnarray*}
         Therefore, for $1< p<\infty $ we have the following:
         \begin{equation*}
          \|(A_n-\tilde A)\|_{L_p}\leq  |\theta_n|\cdot|\beta-\alpha|^{\frac{1}{p}}\cdot|\beta_1-\alpha_1|^{\frac{1}{q}}\to 0,
         \end{equation*}
         when $n\to \infty$, since $\theta_n\to 0$. By applying H\"older inequality we have for all $x\in L_\infty(\mathbb{R},\mu)$ the following:
         \begin{eqnarray*}
         && \|(A_n-\tilde A) x\|_{L_\infty}=  \mathop{\esssup}_{t\in \mathbb{R}} \big| \ \int\limits_{\alpha_1}^{\beta_1}I_{[\alpha,\beta]}(t)\theta_n \cos(\omega t)\sin(\omega s)x(s)d\mu_s\big|\\
         &&\leq
         |\theta_n|\cdot \int\limits_{\mathbb{R}} \left|I_{[\alpha_1,\beta_1]}(s) \sin(\omega s)\right|d\mu_s\cdot \|x\|_{L_\infty} \le |\theta_n|\cdot|\beta_1-\alpha_1|\cdot\|x\|_{L_\infty}.
         \end{eqnarray*}
          Therefore, for $p=\infty $ we have the following:
         \begin{equation*}
          \|A_n-\tilde A\|_{L_\infty}\le  |\theta_n|\cdot|\beta_1-\alpha_1|\to 0,
         \end{equation*}
         when $n\to\infty$ since $\theta_n\to 0$.
        For all $x\in L_1(\mathbb{R},\mu)$,
         \begin{align*}
         & \|(A_n-\tilde A) x\|_{L_1}=  \int\limits_{\mathbb{R}} \big| \int\limits_{\alpha_1}^{\beta_1}I_{[\alpha,\beta]}(t)\theta_n \cos(\omega t)\sin(\omega s)x(s)d\mu_s\big|d\mu_t\\
         &
         \leq
         |\theta_n||\beta-\alpha|\cdot \int\limits_{\mathbb{R}} |I_{[\alpha_1,\beta_1]}(s) \sin(\omega s)x(s)|d\mu_s \leq |\theta_n|\cdot|\beta-\alpha|\cdot\|x\|_{L_1}.
         \end{align*}
          Therefore, for $ p=1 $ we have the following:
         \begin{equation*}
          \|A_n-\tilde A\|_{L_1}\le  |\theta_n|\cdot|\beta-\alpha|\to 0,
         \end{equation*}
         when $n\to\infty$ since $\theta_n\to 0$.
         Therefore, $A_n\to\tilde A$ converge in norm in $L_p$, $1\le p\le \infty$.

         Now we prove the convergence for the sequence $\{A_nB-BA_n\}$. By applying H\"older inequality we have for all $x\in L_p(\mathbb{R},\mu)$, $1< p<\infty $ the following:
         \begin{eqnarray*}
           \|(A_nB-BA_n)x\|^p_{L_p}&=& \big|\theta_{n}\sigma_2 \theta_{B,2}\int\limits_{\alpha_1}^{\beta_1} I_{[\alpha,\beta]}(t)\cos(\omega t)\sin(\omega s)x(s)d\mu_s\big|^p d\mu_t
            \\
          &  \leq & |\theta_{n}\sigma_2 \theta_{B,2}|^p\cdot
           |\beta-\alpha|\cdot
           \big|\int\limits_{\mathbb{R}} I_{[\alpha_1,\beta_1]}(s)\sin(\omega s)x(s)d\mu_s\big|^p \\
           &  \leq &
           |\theta_{n}\sigma_2 \theta_{B,2}|^p |\beta-\alpha| \cdot
            |\beta_1-\alpha_1|^{\frac{p}{q}}\cdot \|x\|^p.
         \end{eqnarray*}
         Therefore, for $1< p <\infty$ we have
         \begin{equation*}
         \|A_nB-BA_n\|_{L_p} \le  |\theta_{n}\sigma_2 \theta_{B,2}|\cdot
         |\beta-\alpha|^{\frac{1}{p}}\cdot |\beta_1-\alpha_1|^{\frac{1}{q}}\to 0,\, n\to\infty,
         \end{equation*}
         since $\theta_n\to 0$.
         By applying H\"older inequality, we have for  all $x\in L_\infty(\mathbb{R},\mu)$ the following:
            \begin{eqnarray*}
           &&\|(A_nB-BA_n)x\|_{L_\infty}=\mathop{\esssup}_{t\in \mathbb{R}} \big|\theta_{n}\sigma_2 \theta_{B,2} I_{[\alpha,\beta]}(t)\sin(\omega t)\cdot \int\limits_{\alpha_1}^{\beta_1}\cos(\omega s)x(s)d\mu_s\big|\\
          &&\ \leq  |\theta_{n}\sigma_2 \theta_{B,2}| \int\limits_{\mathbb{R}}| I_{[\alpha_1,\beta_1]}(s) \cos(\omega s)|d\mu_s \|x\|_{L_\infty}
           \leq  |\theta_{n}\sigma_2 \theta_{B,2}|\cdot |\beta_1-\alpha_1|\|x\|_{L_\infty}.
         \end{eqnarray*}
         Therefore, for $p=\infty$ we have
           \begin{equation*}
           \|A_nB-BA_n\|_{L_\infty}\le |\theta_{n}\sigma_2 \theta_{B,2}|\cdot |\beta_1-\alpha_1|\to 0,\, n\to\infty,
         \end{equation*}
         because $\theta_n\to 0$.
          For  all $x\in L_1(\mathbb{R},\mu)$,
            \begin{eqnarray*}
           &&\|(A_nB-BA_n)x\|_{L_1}=\int\limits_{\mathbb{R}} \big|\theta_{n}\sigma_2 \theta_{B,2} I_{[\alpha,\beta]}(t)\sin(\omega t)\cdot \int\limits_{\alpha_1}^{\beta_1}\cos(\omega s)x(s)d\mu_s\big|d\mu_t\\
          &&\ \leq  |\theta_{n}\sigma_2 \theta_{B,2}| |\beta-\alpha|\int\limits_{\mathbb{R}}| I_{[\alpha_1,\beta_1]}(s) \cos(\omega s)x(s)|d\mu_s
           \leq  |\theta_{n}\sigma_2 \theta_{B,2}|\cdot |\beta-\alpha|\|x\|_{L_1}.
         \end{eqnarray*}
         Therefore, for $p=1$ we have
           \begin{equation*}
           \|A_nB-BA_n\|_{L_1}\le |\theta_{n}\sigma_2 \theta_{B,2}|\cdot |\beta-\alpha|\to 0,\, n\to\infty,
         \end{equation*}
         because $\theta_n\to 0$.
          Since the sequence of operators  $\{A_n\}$  converges in norm to $\tilde{A}$, then by Lemma \ref{LemmaOnseSequenceConvOptorsCommute}, $\tilde{A}$ commutes with $B$.

         \noindent\ref{Proposition7ExampleTheoremIntOpRepGenKernLpConvSeqComOp:item3}.
         By applying the same procedure of the item \ref{Proposition7ExampleTheoremIntOpRepGenKernLpConvSeqComOp:item2}, we have, for $1< p<\infty $  the following:
         \begin{equation*}
          \|B_{n}\|_{L_p}\le  \left|\varsigma_n \right|\cdot|\beta-\alpha|^{\frac{1}{p}}\cdot|\beta_1-\alpha_1|^{\frac{1}{q}}\to 0,
         \end{equation*}
         when $n\to\infty$ since $\varsigma_n\to 0$. For $p=\infty$ we have
         \begin{equation*}
          \|B_{n}\|_{L_\infty}\le  \left|\varsigma_{n}\right| \cdot|\beta_1-\alpha_1|\to 0,
         \end{equation*}
         when $n\to\infty$ since $\varsigma_n\to 0$.
         And for $p=1$ we have
         \begin{equation*}
          \|B_{n}\|_{L_1}\le  \left|\varsigma_{n}\right| \cdot|\beta-\alpha|\to 0,
         \end{equation*}
         when $n\to\infty$ since $\varsigma_n\to 0$.
         Therefore, $B_n$ converges in norm to $0$ in $L_p$, $1\le p\le \infty$.

         In order to prove that $\{AB_n-B_nA\}$ converges in norm to 0. We apply the same procedure of the item \ref{Proposition7ExampleTheoremIntOpRepGenKernLpConvSeqComOp:item2} for $1< p<\infty $, so we get for $1< p<\infty$
         \begin{equation*}
           \|AB_n-B_nA\|_{L_p}\le \left|\theta_{A,4}\sigma_2 \varsigma_n\right|\cdot |\beta-\alpha|^{\frac{1}{p}}|\beta_1-\alpha_1|^{\frac{1}{q}}\to 0,\, n\to\infty,
         \end{equation*}
         since $\varsigma_n\to 0$. For $p=\infty$ we get
           \begin{equation*}
           \|AB_n-B_nA\|_{L_\infty}\le \left|\theta_{A,4}\sigma_2 \varsigma_n\right|\cdot |\beta_1-\alpha_1|\to 0,\, n\to\infty,
         \end{equation*}
         because $\varsigma_n\to 0$. And for $p=1$ we get
           \begin{equation*}
           \|AB_n-B_nA\|_{L_1}\le \left|\theta_{A,4}\sigma_2 \varsigma_n\right|\cdot |\beta-\alpha|\to 0,\, n\to\infty,
         \end{equation*}
         because $\varsigma_n\to 0$.

        \noindent\ref{Proposition7ExampleTheoremIntOpRepGenKernLpConvSeqComOp:item4}. By using results from items
         \ref{Proposition7ExampleTheoremIntOpRepGenKernLpConvSeqComOp:item2}, \ref{Proposition7ExampleTheoremIntOpRepGenKernLpConvSeqComOp:item3} and norms properties we have for all $1\le p\le \infty$ the following:
         \begin{eqnarray*}
           \|A_nB_n-B_nA_n\|_{L_p}&\leq&  \|A_nB_n\|_{L_p}+ \|B_nA_n\|_{L_p}\\
           &\leq &\|A_n\|_{L_p}\|B_n\|_{L_p}+\|B_n\|_{L_p}\|A_n\|_{L_p}= 2 \|A_n\|_{L_p}\|B_n\|_{L_p}\to 0,
         \end{eqnarray*}
         when $n\to\infty$, since $\|A_n\|_{L_p}$ is bounded (because $A_n$ converges in norm to $\tilde{A}$) and $\|B_n\|\to 0$.

         \noindent\ref{Proposition7ExampleTheoremIntOpRepGenKernLpConvSeqComOp:item5}. By following the same procedure of the item \ref{Proposition7ExampleTheoremIntOpRepGenKernLpConvSeqComOp:item2}, the operator $AB-BA$ has the following estimation. For all $x\in L_p(\mathbb{R},\mu)$ and for all
         $$
         (\theta_{A,4},\theta_{B,2}, \alpha_1,\beta_1,\alpha,\beta,\sigma_1,\sigma_2,\delta,\omega)\in \Lambda
         $$
         we have
         \begin{equation*}
           \|(AB-BA)x\|_{L_p}\leq |\theta_{A,4}\sigma_2 \theta_{B,2}|\cdot \tilde{\lambda} \cdot \|x\|_{L_p}\to 0,
         \end{equation*}
           when   $\theta_{A,4}\sigma_2 \theta_{B,2}\tilde{\lambda}\to 0$, where $\tilde{\lambda}=|\beta-\alpha|^{\frac{1}{p}}\cdot|\beta_1-\alpha_1|^{\frac{1}{q}}$, if $1< p <\infty$, $\tilde{\lambda}=|\beta_1-\alpha_1|$, if $p=\infty$ and and $\tilde{\lambda}=|\beta-\alpha|$, if $p=1$. 
            \QEDB
    \end{proof}

\begin{remark}
In connection to Propositions \ref{Proposition0ExampleTheoremIntOpRepGenKernLpConvSeqComOp}--\ref{Proposition7ExampleTheoremIntOpRepGenKernLpConvSeqComOp}
the estimation of operators in these propositions can be compared to the one obtained in Remark \ref{RemarkConstantsEstimationOpFourierSeq}. However this depends on the magnitude of $|\beta-\alpha|$ and $|\beta_1-\alpha_1|$.
\end{remark}

 \begin{theorem}\label{PropositionExampleLaurentPolyTheoremIntOpRepGenKernLpConvSeqComOp}
 Let $A:L_p(\mathbb{R},\mu)\to L_p(\mathbb{R},\mu)$, $B:L_p(\mathbb{R},\mu)\to L_p(\mathbb{R},\mu)$, $1< p\leq \infty$ be operators defined as
  \begin{equation*}
  \resizebox{1\hsize}{!}{$ (Ax)(t)=\int\limits_{1}^{2} I_{[\alpha,\infty[}(t)\left(\frac{\gamma_{A,2}}{t}-\frac{2\gamma_{A,2}\ln 2}{ts} \right)x(s)d\mu_s,\,
  (Bx)(t)=\int\limits_{1}^{2} I_{[\alpha,\infty[}(t)\frac{\gamma_{B,2}}{t} x(s)d\mu_s$},
  \end{equation*}
for almost every $t$, where $\gamma_{A,2}, \gamma_{B,2}\in\mathbb{R}$,
 $0<\alpha\le 1$.
 Let $\{A_n\}: L_p(\mathbb{R},\mu)\to L_p(\mathbb{R},\mu)$,  $1< p \leq \infty$ be a sequence of operators
  \begin{equation*}
   (A_n x)(t)=\int\limits_{1}^{2} I_{[\alpha,\infty[}(t)\left(\frac{\gamma_{n}}{t}-\frac{2\gamma_{n}\ln 2}{ts} \right)x(s)d\mu_s,
  \end{equation*}
 for almost every $t$, $\{\gamma_n\}$ is a number sequence.
    Let
         \begin{equation*}
          \Lambda=\left\{(\gamma_{A,2}, \gamma_{B,2}, \alpha)\in\mathbb{R}^{3}:\, 0<\alpha\leq 1 \right\}.
         \end{equation*}
  Then
  \begin{enumerate}[label=\textup{\arabic*.}, ref=\arabic*]
           \item\label{PropositionExampleLaurentPolyTheoremIntOpRepGenKernLpConvSeqComOp:item1}  $AB=\delta BA^2=0$, $\delta\in\mathbb{R}$, $A_nB=\delta BA_n^2$, $A_nB_n=\delta B_nA_n^2 $, $AB_n=\delta B_nA^2 $ for each positive integer $n$. Moreover, for all $x\in L_p(\mathbb{R},\mu)$, $1< p\le \infty$ we have
       \begin{equation*}
         (AB-BA)x(t)=-(BAx)(t)=-\gamma_{A,2}\gamma_{B,2}\ln2 \int\limits_{1}^{2} I_{[\alpha,\infty[}(t)\frac{1}{t}\left(1-\frac{2\ln 2}{s}\right)x(s)d\mu_s,
       \end{equation*}
       for almost every $t$.
           \item\label{PropositionExampleLaurentPolyTheoremIntOpRepGenKernLpConvSeqComOp:item2} if $\gamma_n\to 0$ when $n\to \infty$, then $A_n\to 0$ (converge in norm)
         and so  $A_nB-BA_n\to 0$ (converge in norm).

          \item\label{PropositionExampleLaurentPolyTheoremIntOpRepGenKernLpConvSeqComOp:item3} if $1<p<\infty$ and
          $\gamma_n \gamma_{B,2}\alpha^{\frac{1-p}{p}}\to 0$ or,  if $p=\infty$ and
          $\gamma_{A,2}\gamma_{B,2}\frac{1}{\alpha}\to 0$ when
            $(\gamma_{A,2},\gamma_{B,2}, \alpha)\to \lambda_0\in cl(\Lambda)$ (the closure of $\Lambda$), where $(\gamma_{A,2},\gamma_{B,2}, \alpha)\in\Lambda$, then for all $x\in L_p(\mathbb{R},\mu)$, $1< p\leq \infty$, it holds that $\|(AB-BA)x\|_{L_p}\to 0$. 
         \end{enumerate}
       \end{theorem}

    \begin{proof}
     \noindent\ref{PropositionExampleLaurentPolyTheoremIntOpRepGenKernLpConvSeqComOp:item1}. It follows by direct computation.

       \noindent\ref{PropositionExampleLaurentPolyTheoremIntOpRepGenKernLpConvSeqComOp:item2}. We first prove that $A_n\to 0$ when $n\to\infty$. By applying H\"older inequality, for all $x\in L_p(\mathbb{R})$ $1<p<\infty$, we have the following:
       \begin{eqnarray*}
         \|A_n x\|^p_{L_p}&=&\int\limits_\mathbb{R}\big| \int\limits_{1}^{2} I_{[\alpha,\infty[}(t)\big(\frac{\gamma_{n}}{t}-\frac{2\gamma_{n}\ln 2}{ts} \big)x(s)d\mu_s\big|^p d\mu_t\\
        &=& |\gamma_n |^p \int\limits_{\alpha}^{\infty} \frac{1}{t^p} d\mu_t \big|\int\limits_{\mathbb{R}} I_{[1,2]}(s)\left(1-\frac{2\ln 2}{s}\right)x(s)d\mu_s\big|^p\\
         &\leq & |\gamma_n |^p \frac{|\alpha|^{1-p}}{p-1} \left(\int\limits_{1}^{2} \big|1-\frac{2\ln 2}{s}\big|^q d\mu_s\right)^{\frac{p}{q}} \big(\int\limits_{\mathbb{R}} |x(s)|^p d\mu_s\big)\\
         &\leq & |\gamma_n|^p \frac{|\alpha|^{1-p}}{p-1} (1+2\ln 2)^p\|x\|^p.
       \end{eqnarray*}
        Therefore,
          \begin{equation*}
            \|A_n\|_{L_p}\le |\gamma_n | \frac{|\alpha|^{\frac{1-p}{p}}}{(p-1)^{\frac{1}{p}}} (1+2\ln 2)\to 0,
          \end{equation*}
          when $\gamma_n\to 0$.

          Similarly, by applying H\"older inequality,  for all $x\in L_\infty(\mathbb{R},\mu)$ we have the following:
       \begin{eqnarray*}
        && \|A_nx\|_{L_\infty}=\mathop{\esssup}_{t\in \mathbb{R}} \big| \gamma_n \int\limits_{1}^{2} I_{[\alpha,\infty[}(t)\frac{1}{t}\left(1-\frac{2\ln 2}{s}\right)x(s)d\mu_s\big| \\
         &&\ =\left| \frac{\gamma_n}{\alpha}\right| \big|\int\limits_{\mathbb{R}} I_{[1,2]}(s)\left(1-\frac{2\ln 2}{s}\right)x(s)d\mu_s\big|
         \leq   \left| \frac{\gamma_n}{\alpha}\right| \left(\int\limits_{1}^{2} \left|1-\frac{2\ln 2}{s}\right| d\mu_s\right) \|x\|_{L_\infty}
         \\
        &&\ \leq  \frac{|\gamma_n|}{\alpha}  (1+2\ln 2)\|x\|_{L_\infty}.
       \end{eqnarray*}
         Therefore,
          \begin{equation*}
            \|A_n\|_{L_\infty}\le  \frac{|\gamma_n|}{\alpha} (1+2\ln 2)\to 0,
          \end{equation*}
          when $\gamma_n\to 0$. We have
       \begin{equation*}
         (A_nB-BA_n)x(t)=-(BA_nx)(t)=-\gamma_{n}\gamma_{B,2}\ln2 \int\limits_{1}^{2} I_{[\alpha,\infty[}(t)\frac{1}{t}\left(1-\frac{2\ln 2}{s}\right)x(s)d\mu_s,
       \end{equation*}
       for almost every $t$. Therefore, by applying H\"older inequality, for all $x\in L_p(\mathbb{R},\mu)$ $1<p<\infty$, we have the following:
       \begin{eqnarray*}
        && \|(A_nB-BA_n)x\|^p_{L_p}=\int\limits_{\mathbb{R}} \big| \gamma_n \gamma_{B,2}\ln2 \int\limits_{1}^{2} I_{[\alpha,\infty[}(t)\frac{1}{t}\left(1-\frac{2\ln 2}{s}\right)x(s)d\mu_s\big|^p d\mu_t\\
         &&=|\gamma_n \gamma_{B,2}\ln 2|^p \int\limits_{\alpha}^{\infty} \frac{1}{t^p} d\mu_t \big|\int\limits_{\mathbb{R}} I_{[1,2]}(s)\left(1-\frac{2\ln 2}{s}\right)x(s)d\mu_s\big|^p\\
         &&\leq  |\gamma_n \gamma_{B,2}\ln 2|^p \frac{|\alpha|^{1-p}}{p-1} \left(\int\limits_{1}^{2} \left|1-\frac{2\ln 2}{s}\right|^q d\mu_s\right)^{\frac{p}{q}} \left(\int\limits_{\mathbb{R}} |x(s)|^p d\mu_s\right)\\
         &&\leq  |\gamma_n \gamma_{B,2}\ln 2|^p \frac{|\alpha|^{1-p}}{p-1} (1+2\ln 2)^p\|x\|^p.
       \end{eqnarray*}
         Therefore,
          \begin{equation*}
            \|A_nB-BA_n\|_{L_p}\le |\gamma_n \gamma_{B,2}\ln 2| \frac{|\alpha|^{\frac{1-p}{p}}}{(p-1)^{\frac{1}{p}}} (1+2\ln 2)\to 0,
          \end{equation*}
          when $\gamma_n\to 0$.
       Similarly, by applying H\"older inequality,  for all $x\in L_\infty(\mathbb{R},\mu)$ we have the following:
       \begin{eqnarray*}
         &&\|(A_nB-BA_n)x\|_{L_\infty}=\mathop{\esssup}_{t\in \mathbb{R}} \big| \gamma_n \gamma_{B,2}\ln2 \int\limits_{1}^{2} I_{[\alpha,\infty[}(t)\frac{1}{t}\left(1-\frac{2\ln 2}{s}\right)x(s)d\mu_s\big| \\
         &&=\left|\gamma_n \gamma_{B,2}\frac{\ln 2}{\alpha}\right| \big|\int\limits_{\mathbb{R}} I_{[1,2]}(s)\left(1-\frac{2\ln 2}{s}\right)x(s)d\mu_s\big|\\
         &&\leq  \left|\gamma_n \gamma_{B,2}\frac{\ln 2}{\alpha}\right| \left(\int\limits_{1}^{2} \left|1-\frac{2\ln 2}{s}\right| d\mu_s\right) \|x\|_{L_\infty}
         \leq  \left|\gamma_n \gamma_{B,2}\frac{\ln 2}{\alpha}\right|  (1+2\ln 2)\|x\|_{L_\infty}.
       \end{eqnarray*}
         Therefore,
          $
            \|A_nB-BA_n\|_{L_\infty}\le \left|\gamma_n \gamma_{B,2}\frac{\ln 2}{\alpha}\right| (1+2\ln 2)\to 0,
          $
          when $\gamma_n\to 0$.

         \noindent\ref{Proposition7ExampleTheoremIntOpRepGenKernLpConvSeqComOp:item5}. By following the same procedure of the item \ref{Proposition7ExampleTheoremIntOpRepGenKernLpConvSeqComOp:item2}, the operator $AB-BA$ has the following estimation. For all $x\in L_p(\mathbb{R},\mu)$, $1<p<\infty$ and for all $(\gamma_{A,2},\gamma_{B,2}, \alpha)\in\Lambda$,
         \begin{equation*}
           \|(AB-BA)x\|_{L_p}\leq |\gamma_{A,2} \gamma_{B,2}\ln 2| \frac{|\alpha|^{\frac{1-p}{p}}}{(p-1)^{\frac{1}{p}}} (1+2\ln 2) \cdot \|x\|_{L_p}\to 0,
         \end{equation*}
           when   $\gamma_{A,2} \gamma_{B,2} \alpha^{\frac{1-p}{p}}\to 0$ and $\tilde{\lambda}=|\beta_1-\alpha_1|$.
           If $p=\infty$,
           for all $x\in L_\infty(\mathbb{R},\mu)$ and for all $(\gamma_{A,2},\gamma_{B,2}, \alpha)\in \Lambda$ we have
         \begin{equation*}
           \|(AB-BA)x\|_{L_\infty}\leq \left|\gamma_{A,2} \gamma_{B,2}\frac{\ln 2}{\alpha}\right| (1+2\ln 2)\|x\|_{L_\infty}\to 0,
         \end{equation*}
         when $\gamma_{A,2} \gamma_{B,2}\frac{1}{\alpha}\to 0$.
       \QEDB
        \end{proof}

        \begin{remark}
        In connections to Theorems \ref{Proposition0ExampleTheoremIntOpRepGenKernLpConvSeqComOp}--\ref{PropositionExampleLaurentPolyTheoremIntOpRepGenKernLpConvSeqComOp}
        , if $\lambda_0\in \Lambda$ then the commutator $[A,B]=AB-BA$ exists for parameters at point $\lambda_0\in\Lambda$ and it is equal zero under some convergence conditions, however there are issues regarding convergence of operators $A$ and $B$. If $\lambda_0\in$ cl$(\Lambda)\setminus \Lambda$ then
        $[A,B]=AB-BA$, operators $A$, $B$ might not exist for parameters at point $\lambda_0$. However
        the commutator $[A,B]=AB-BA$ and operators $A$, $B$  exist for parameters in  $\Lambda$ and when moving
        these parameters in these set to a point $\lambda_0 \in$ cl$(\Lambda)\setminus \Lambda$, the commutator $[A,B]=AB-BA$ converges to zero under some conditions.
        \end{remark}

\section*{Acknowledgments}
This work was supported by the Swedish International Development Cooperation Agency (Sida), bilateral capacity development program in Mathematics with Mozambique. Domingos Djinja is grateful to
Mathematics and Applied Mathematics research environment MAM, Division of Mathematics and Physics, School of Education, Culture and Communication, M{\"a}lardalen University for excellent environment for research in Mathematics.


\end{document}